\documentclass[10pt]{amsart}
\usepackage[margin =1in]{geometry}
\newcommand{\sym}{\mathrm{sym}}
\usepackage{amsfonts}
\usepackage{amsmath}
\usepackage{hyperref}
\usepackage{todonotes}
\usepackage{mathrsfs}
\usepackage{amssymb}
\usepackage{enumerate}

\usepackage{comment}
\usepackage{esint}
\newcommand{\T}{\mathbb{T}}
\newcommand{\E}{\mathbb{E}}
\newcommand{\EE}{\mathcal{E}}
\newcommand{\Z}{\mathbb{Z}}
\newcommand{\R}{\mathbb{R}}
\renewcommand{\div}{\mathrm{div}}
\newcommand{\vp}{\varphi}

\usepackage{mathtools}
\mathtoolsset{showonlyrefs}

\allowdisplaybreaks
\newcommand{\vt}{\vartheta}
\newcommand{\Hb}{\mathbb H}
\newcommand{\Vb}{\mathbb V}
\newcommand{\Wb}{\mathbb W}
\newcommand{\F}{\mathscr{F}}
\newcommand{\A}{\mathfrak{A}}
\newcommand{\M}{\mathcal{M}}
\renewcommand{\P}{\mathbb{P}}
\renewcommand{\E}{\mathbb{E}}
\newcommand{\D}{\mathscr{D}}
\numberwithin{equation}{section}
\def\theequation{\arabic{section}.\arabic{equation}}
\newcommand{\be}{\begin{eqnarray}}
	\newcommand{\ee}{\end{eqnarray}}
\newcommand{\ce}{\begin{eqnarray*}}
	\newcommand{\de}{\end{eqnarray*}}
\newtheorem{theorem}{Theorem}[section]
\newtheorem{lemma}[theorem]{Lemma}
\newtheorem{proposition}[theorem]{Proposition}

\theoremstyle{definition}
\newtheorem{remark}[theorem]{Remark}
\newtheorem{assumption}[theorem]{Assumption}
\newtheorem{convention}[theorem]{Convention}
\newtheorem{definition}[theorem]{Definition}

\newcommand{\Sol}{\mathrm{Sol}}

\renewcommand{\d}{\mathrm{d}}
\newcommand{\reg}{{\infty}}

\title[Navier--Stokes--Fourier system with thermal noise]{\LARGE The incompressible Navier--Stokes--Fourier system with  thermal noise}

\date{\today}

\author[Benjamin Gess]{\large Benjamin Gess}
\address[B. Gess]{Institut f\"ur Mathematik, Technische Universit\"at Berlin, 10623 Berlin, Germany \\
and\\
Max--Planck--Institut f\"ur Mathematik in den Naturwissenschaften\\
04103 Leipzig, Germany}
\email{Benjamin.Gess@mis.mpg.de}

\author[Max Sauerbrey]{\large Max Sauerbrey}
\address[M. Sauerbrey]{Max--Planck--Institut f\"ur Mathematik in den Naturwissenschaften\\
	04103 Leipzig, Germany}
\email{maxsauerbrey97@gmail.com}

\author[Zhengyan Wu]{\large Zhengyan Wu}
\address[Z. Wu]{Department of Mathematics, Technische Universit\"at M\"unchen, Boltzmannstr. 3, 85748 Garching, Germany}
\email{wuzh@cit.tum.de} 

\thanks{The first author acknowledges support by the Max Planck Society
	through the Research Group “Stochastic Analysis in the Sciences (SAiS)” and the DFG
	CRC/TRR 388 “Rough Analysis, Stochastic Dynamics and Related Fields”, Project A11.  The third author is supported by the US Army Research Office, grant W911NF2310230.  This research was partially carried out during the second author's visits at the TU Berlin supported by the European Union (ERC,
	FluCo, grant agreement No. 101088488). Views
	and opinions expressed are however those of the author(s) only and do not necessarily reflect those
	of the European Union or of the European Research Council. Neither the European Union nor the
	granting authority can be held responsible for them. We  thank Antonio Agresti for helpful comments on this manuscript.}

\begin{document}
	\maketitle
	
\begin{abstract}
	We establish a solution theory  for the  incompressible Navier--Stokes--Fourier system with thermal noise, posed on the three-dimensional torus. While in the incompressible  deterministic  setting the  equation for the velocity can be solved independently of the temperature, the inclusion of the effects of thermal fluctuations by means of the GENERIC framework leads to a nonlinear gradient noise term, which couples the dynamics of both variables. Therefore, the  analysis poses new challenges, which are absent in the deterministic incompressible Navier--Stokes--Fourier equations. In particular, the a priori estimates used in the deterministic setting are not readily generalizable, the noise introduces strongly nonlinear gradient terms and the total energy lacks convexity. \\
    These challenges are overcome in the present work by a novel variable transformation, and novel entropy dissipation estimates. Thereby, the existence of global-in-time weak solutions for $L_x^2$ initial data, the existence of local-in-time strong solutions for regular initial data, and weak-strong uniqueness are obtained. 
\end{abstract}

\section{Introduction}
In this paper, we establish a solution theory for the \emph{incompressible Navier-Stokes-Fourier system}  subject to \emph{thermal fluctuations}, posed on the \emph{three dimensional torus $\T^3$}, which reads
\begin{align}\begin{split}\label{Eq_intro}
		\partial_tu\,=\,&\nabla\cdot(\nabla_{\sym}u)\,-\,\Pi \nabla\cdot (u\otimes u)\,-\,\Pi \nabla\cdot(\sqrt{\vt}\circ\xi_{F}),\qquad \nabla\cdot u\,=\,0,\\
		\partial_t\vt \,=\, &\Delta \vt-\nabla\cdot(u\vt)\,+\,|\nabla_{\sym} u|^2\,-\,\nabla\cdot(\vt\circ \xi_G)\,-\, \sqrt{\vt}\nabla  u :\circ\, \xi_F, 
	\end{split}
\end{align}
where $\nabla_{\sym}u=\frac{1}{2}(\nabla u+\nabla u^T)$ denotes the symmetrized gradient and $\Pi u =  u -\nabla \Delta^{-1} (\nabla \cdot u)$ the Helmholtz projection onto  divergence free vector fields. Moreover,  
$\xi_F$ and $\xi_G$ are centered, jointly Gaussian random fields with covariance structure
\begin{align}\begin{split}\label{Eq_Noise_Correlation}
		\mathbb{E}[\xi_F^{ij}(s,x )\xi_G^{k}(t,y)]\,=\,&0,
		\\
		\mathbb{E}[ \xi_F^{ij}(s,x)\xi_F^{kl}(t,y)]\,=\,&\big(\delta_{ik} \delta_{jl}+\delta_{il}\delta_{jk}
		\big)\delta(t-s)F(x,y),
		\\
		\mathbb{E}[\xi_G^i (s,x)\xi_G^{j}(t,y)]\,=\,&  \delta_{ij}\delta(t-s)G(x,y) ,
	\end{split}
\end{align}
for $ 1\le i,j,k,l\le 3$ and sufficiently regular 
covariance kernels  $F,G\colon \T^3\times \T^3\to \R$. 
The unknowns  $(u,\vt)$ denote the state variables of an incompressible heat-conducting fluid, namely $u$ is the \emph{velocity} and  $\vt$  is the \emph{temperature} at a given space-time point $(t,x)$. 

For space-time white noise, that is, $F(x,y)=G(x,y) = \delta(x-y)$, \eqref{Eq_intro} is the incompressible version of the \emph{Landau--Lifshitz--Navier--Stokes equation}, cf.\ \cite[Chapter IX]{lifshitz}, and can also be derived from the \emph{GENERIC framework} \cite{Ottinger}, as we recall in Subsection \ref{der:gen} below. These equations join several challenges, including their super-critical nature, in the language of regularity structures, due to the irregularity of space-time white noise, as well as the occurrence of irregular coefficients and the lack of coercivity. In this work, we concentrate on the latter challenges by considering spatially more regular noise. Notably, nevertheless, in suitable joint scaling regimes of vanishing noise and de-correlation length, the correct large deviation behavior is retained \cite{FG23}. 
 
The use of Stratonovich calculus $\circ$ in \eqref{Eq_intro} is essential, since it leads to the  conservation of 
\emph{total energy }
\begin{equation}\label{eqn_enregy_intro}\EE (u,\vt) \,=\, \int_{\T^3} \frac{1}{2}|u|^2 \,+\, \vt \,\d x ,
\end{equation}
as can be seen informally by integration by parts. In contrast, an interpretation in It\^o sense would lead to a strict energy production, violating physical principles. We say that $(u,\vt)$ is \emph{globally admissible}, if $\P$-a.s, 
\begin{align}
	\sup_{t}\EE (u_t,\vt_t) \,\le\, \EE (u_0,\vt_0).
\end{align}
The  three {main results} of this work are:
\begin{enumerate}[(A)]
	\item\label{item_a} The global in time existence of {globally admissible weak martingale solutions}: Theorem \ref{thm-weak-existence}.
	\item\label{item_b} The local in time existence of {strong solutions} for regular data: Theorem \ref{Prop_ex_strong_solution}.
	\item\label{item_c} A corresponding {weak-strong uniqueness} principle: Theorem \ref{thm_WS_main}.
\end{enumerate}
Let us comment on the {challenges} when establishing these results. Since in the  case that $\xi_F = 0$, the velocity equation in \eqref{Eq_intro} can be solved independently of the equation for $\vt$, the fact that \eqref{Eq_intro} is a fully coupled  system is specific to the stochastic case: 
\begin{enumerate}[(i)]
	
	\item As, in contrast to the deterministic Navier--Stokes equation, energy is not dissipated, additional gradient estimates for  $u$ and $\vt$ are not readily available, making the construction of weak solutions more subtle. In the deterministic literature on compressible Navier--Stokes--Fourier equations, this is typically resolved 
	by closing additionally estimates in terms of the \emph{entropy} 
	\begin{equation}\label{eq_ent}
		\mathcal{S} \,=\, \int_{\T^3 }\log(\vt)\,\d x,
	\end{equation} 
	cf.\ \cite[Section 5.2]{feireisl_book} or \cite[Section 2.2.3]{feireisl}. It turns out however, that the entropy estimate for \eqref{Eq_intro} {does 
	 not close at all noise intensities}, see Appendix \ref{app:entropy} for details. Therefore, novel a priori estimates are required. 
	
	\item The lack of an entropy estimate is amplified by the presence of strongly nonlinear  gradient terms and {irregular coefficients} on the right-hand side of \eqref{Eq_intro}, which necessitate a control on $\nabla u$ and $\nabla \vt$ along  an approximation scheme. Moreover, as we detail in Subsection \ref{SS:transformation_is_necessary}, the {gradient nonlinearities do not admit a sign} after writing \eqref{Eq_intro} in It\^o form, so that relaxing the equation to an inequality as in \cite[Section 4.3.1]{feireisl_book} is not a solution.

	\item Regarding weak-strong uniqueness, a relative energy argument, see \cite{wiedemann}, is obstructed by the {lack of strict convexity of the energy} functional \eqref{eqn_enregy_intro}. Moreover, even at small noise intensities, an adaption of the {relative entropy method \cite{FN12} seems not to be possible}, due to the gradient noise terms on the right-hand side of \eqref{Eq_intro}.
\end{enumerate}

\noindent
These difficulties are addressed in this work by the following {novel techniques}:

\begin{enumerate}[(i)]
	
	\item We introduce a {mathematical entropy}
	\begin{equation}\label{eqn1009}
	-\int_{\T^3}\sqrt{2\vt} \,\d x.\end{equation}
     which leads to  an a priori estimate, on the $L^1_{\omega,t,x}$-norm of suitable gradient terms, cf. Subsection \ref{SS:transformation_is_necessary}. 
	
	\item We {eliminate the irregular coefficients} by a transformation of the SPDE, by writing the equations in terms of  $\psi   = \sqrt{2\vt}$ instead of $\vt$ itself. This results  in the equation 
	\begin{align}\begin{split}\label{Eq_new_var_formal}
			\partial_t\psi 
			\,=\, \Delta \psi \,+\, \frac{1}{\psi }\bigl(
			|\nabla \psi|^2 \,+ \,|\nabla_\sym u|^2 
			\bigr) \,-\, \nabla \cdot ( u\psi )\,-\, \nabla \cdot ( \psi \circ\xi_G) \,+\, \frac{1}{2}\psi \circ\nabla \cdot \xi_G\,-\,  \frac{1}{\sqrt{2}}\nabla  u :\circ \xi_F,
	\end{split}\end{align}
	for the new variable 
	and
	\begin{equation}\label{Eq_new_u}
		\partial_tu\,=\,\nabla\cdot(\nabla_{\sym}u)\,-\,\Pi \nabla \cdot (u\otimes u)\,-\,\frac{1}{\sqrt{2}}\Pi \nabla\cdot(\psi\circ \xi_{F}),\qquad \nabla\cdot u\,=\,0,
	\end{equation}
	for the velocity. While difficulties arising from the {gradient nonlinearities} still persist, these do now, after rewriting the equation in It\^o form, {admit a sign}. This allows, in contrast to \eqref{Eq_intro}, to give an attainable notion of weak solutions to \eqref{Eq_new_var_formal}--\eqref{Eq_new_u} in terms of an inequality, see Subsection \ref{SS:transformation_is_necessary} below.

	\item Passing to the new variable  $\psi$ turns the {energy} into the {strictly convex} functional
	\begin{equation}\label{eqn333}
		\EE (u,\psi)  \,=\, \frac{1}{2} \int_{\T^3} |u|^2 \,+\, \psi^2 \,\d x .
	\end{equation}
	This allows us to prove weak-strong uniqueness by establishing a bound on the {novel  relative energy}
	\begin{equation}\label{Eq_rel_energy_intro}\EE ( u,\psi | V, \phi) \,=\, \frac{1}{2} \int |u-V|^2 + |\psi-\Phi|^2\, \d x,
	\end{equation}
	for a weak solution $(u,\psi)$ and a strong solution $(V,\Phi)$. That this approach is intrinsic for the Navier--Stokes--Fourier system with  thermal noise is confirmed by the fact that all noise terms in the relative energy expansion cancel. Moreover, since {$\psi$ may have temporal jumps}, we employ {stochastic analysis tools for jump processes} when calculating the latter, which seems to be new regarding stochastic fluid equations as well.
\end{enumerate}
\noindent

We next recall how to rephrase the deterministic  Navier--Stokes--Fourier system in the General Equation for Non-Equilibrium Reversible-Irreversible Coupling (GENERIC) framework. Based on the fluctuation-dissipation principle (see \cite{Ottinger}), this gives  rise to the thermal noise terms in \eqref{Eq_intro}. We then discuss the necessity of rewriting the system in terms of $(u,\psi)$ in greater detail, introduce the solution concepts, comment further on the strategies of proofs, and conclude the introduction with comments on  related literature.

\subsection{Derivation from the  GENERIC framework}
\label{der:gen}
Equation \eqref{Eq_intro} can be obtained as the incompressible limit of the compressible Navier--Stokes--Fourier system subjected to thermal noise. The latter was originally derived by Landau and Lifshitz by means of the  fluctuation dissipation principle as described in \cite{lifshitz} and is   also referred to as the Landau--Lifshitz--Navier--Stokes equation. It  fits  moreover into the GENERIC framework proposed by \"Ottinger \cite{Ottinger},  which systematizes the derivation of thermal noise terms in  hydrodynamical equations. While consequently \eqref{Eq_intro} can be obtained  by applying the GENERIC framework to derive the full Landau--Lifshitz--Navier--Stokes equation and subsequently imposing $\nabla \cdot u = 0$ as well as a constant mass density, we detail in this section how to obtain it directly from \cite{Ottinger}.

To this end, we first consider  the deterministic situation, in which a heat-conducting, incompressible fluid, with normalized mass and constant specific heat, is modeled by the system 
\begin{align}
		\partial_tu\,=\,&\nabla\cdot \sigma \,-\, \nabla\cdot (u\otimes u)\,-\,\nabla p ,\qquad \nabla\cdot u\,=\,0,\\
	\partial_t\vt \,=\, &
	\nabla ( \kappa \nabla \vt )\,-\, \nabla\cdot(u\vt)\,+\,\sigma : \nabla_{\sym} u,
\end{align}
where $p$ is the pressure, $\kappa$ the thermal conductivity and $\sigma$ the viscous stress tensor. In the incompressible case, the latter is given by $\sigma = 2 \eta \nabla_{\sym} u$, where $\eta $ is the shear viscosity. Assuming  constant material coefficients $\kappa = 1$ and $\eta = 1/2$, we find that
\begin{align}\begin{split} \label{eqn1037}
		\partial_tu\,=\,&\nabla\cdot(\nabla_{\sym}u)\,-\,\Pi \nabla\cdot (u\otimes u),\qquad \nabla\cdot u\,=\,0,\\
		\partial_t\vt \,=\, &\Delta \vt-\nabla\cdot(u\vt)\,+\,|\nabla_{\sym} u|^2, 
	\end{split}
\end{align}
since $\Pi  \nabla p = 0$, i.e., \eqref{Eq_intro} with $\xi_F= 0$ and $\xi_G=0$. To derive the thermal noise term, one needs to recast this equation into the form 
\begin{align}\label{eqn1038}
	\partial_t \mathbf{z} \,=\, \mathcal{L} \cdot  \frac{\delta \mathcal{E}}{\delta \mathbf{z}} \,+\, 
	\mathcal{M} \cdot \frac{\delta \mathcal{S} }{\delta \mathbf{z}},
\end{align}
where $\mathbf{z} = (u,\vt)$,  $\EE$ is the total energy \eqref{eqn_enregy_intro} of the system, $\mathcal{S}$
is the entropy \eqref{eq_ent} and $\mathcal{L}$ and  $\mathcal{M}$ are   antisymmetric and symmetric non-negative definite operators, respectively. The notation $ \frac{\delta A}{\delta \mathbf{z}} $ stands here for the first  variation  of a functional so that 
\[
\frac{\delta \mathcal{E}}{\delta \mathbf{z}}\,=\, 
\begin{bmatrix}
	u\\1
\end{bmatrix},
\qquad \frac{\delta \mathcal{S} }{\delta \mathbf{z}}\,=\,\begin{bmatrix}
	0 \\1/\vt
\end{bmatrix} .
\]
One finds that the choice
\begin{align}\begin{split}\label{eqn_M_and_L}
\mathcal{L}
 \,&=\, 
\begin{bmatrix}
	-\Pi (\nabla \cdot ( (\cdot )\otimes u ) +  (\nabla (\cdot ))\cdot u ) &- \Pi (\vt \nabla (\cdot ))
	\\
	-\nabla \cdot  (\vt  (\cdot ))& 0
\end{bmatrix},\\
\mathcal{M}
\,&=\, 
\begin{bmatrix}
	- \Pi \nabla \cdot (\vt \nabla_{\sym} (\cdot))  & \Pi ( \nabla \cdot ( \vt (\cdot) \nabla_\sym u ))
	\\
	-\vt \nabla_\sym u : \nabla(\cdot )
	& \vt |\nabla_{\sym}u|^2(\cdot ) \,-\, \nabla (\vt^2 \nabla (\cdot ))
\end{bmatrix},
\end{split}
\end{align}
turns indeed \eqref{eqn1038} into \eqref{eqn1037}. Moreover,  we have the required conditions $ \mathcal{L} \frac{\delta \mathcal{S}}{\delta \mathbf{z}} =\mathcal{M} \frac{\delta \mathcal{E}}{\delta \mathbf{z}} =0$ and the bracket $\{A,B\}_{\mathcal{L}} = \frac{\delta A}{\delta \mathbf{z}}\cdot  \mathcal{L} \cdot \frac{\delta B}{\delta \mathbf{z}} $ associated to $\mathcal{L}$ satisfies Jacobi's identity
\begin{equation}\label{Jacobis_identity}
\{ A, \{ B, C\}_{\mathcal{L}}\}_{\mathcal{L}}\,+\, 
\{ B, \{ C,A \}_{\mathcal{L}}\}_{\mathcal{L}}\,+\, 
\{ C, \{ A,B\}_{\mathcal{L}}\}_{\mathcal{L}} \,=\, 0,
\end{equation}
for any triplet of  functionals $A,B$ and $C$.
The details of the latter  we defer to Appendix \ref{sec-app-A}.
 Thereby, all assumptions of the GENERIC framework are satisfied and thermal noise in accordance with the fluctuation dissipation relation is introduced by modulating \eqref{eqn1038} to
\begin{align}\label{eqn1039}
	\partial_t \mathbf{z} \,=\, \mathcal{L} \cdot  \frac{\delta \mathcal{E}}{\delta \mathbf{z}} \,+\, 
	\mathcal{M} \cdot \frac{\delta \mathcal{S} }{\delta \mathbf{z}} \,+\, \mathcal{B } \xi ,
\end{align}
where $\mathcal{B} \mathcal{B}^\top = \mathcal{M}$ and $\xi$ is space-time white noise in a space through which we factor the  linear operator $\mathcal{M}$. We choose
\begin{align}
	\mathcal{B}
	\,&=\, 
	\begin{bmatrix}
		-\Pi (\nabla \cdot  (\sqrt{\vt }(\cdot)_\sym) & 0
		\\
	- \sqrt{\vt} \nabla_{\sym} u :(\cdot) & -\nabla \cdot  ( \vt (\cdot) )
	\end{bmatrix},
\end{align}
acting on the product of matrix- and vector-valued functions, where $(\cdot)_\sym = \frac{1}{2}((\cdot)+ (\cdot )^\top )$ denotes the symmetrization map.
Therefore, we have
\begin{align}
	\mathcal{B}^\top
	\,&=\, 
	\begin{bmatrix}
		\sqrt{\vt} \nabla_\sym (\cdot)  & -\sqrt{\vt} \nabla_{\sym} u (\cdot )
		\\
		0 & \vt \nabla (\cdot )
	\end{bmatrix}
\end{align}
and consequently the desired factorization of $\mathcal{M}$. 

Now let $( \xi_{\mathrm{mat}},  \xi_{\mathrm{vec}})$ be a space-time white noise, with values in the product of matrices and vectors. Then the symmetrization of $\xi_\mathrm{mat} $ is a centered Gaussian with covariance
\begin{equation}
\mathbb{E}\bigl[ \xi_{\mathrm{mat},\sym}^{ij}(s,x)\xi^{kl}_{\mathrm{mat},\sym}(t,y)\bigr]\,=\,\frac{1}{2}\big(\delta_{ik} \delta_{jl}+\delta_{il}\delta_{jk})\delta(t-s)\delta (x-y),
\end{equation}
for  $ 1\le i,j,k\le 3$. Thereby, after rewriting $\mathcal{B} \xi$ in \eqref{eqn1039} in terms of $(\xi_{\mathrm{mat},\sym},\xi_{\mathrm{vec}})$ and convolving these in space, \eqref{Eq_intro}  indeed becomes the Stratonovich interpretation of \eqref{eqn1039}.

\subsection{Noise regularity and  It\^o-to-Stratonovich corrections}\label{SS:noise_ass}
In order to  write \eqref{Eq_intro} in It\^o form, we impose some structural assumptions on the noises: 
 For the remainder of this manuscript, we write the temporally white but spatially correlated Gaussian noise terms $\xi_F$ and $\xi_G$ from \eqref{Eq_intro} and \eqref{Eq_new_var_formal}--\eqref{Eq_new_u}, as 
\begin{equation}\label{Eq_noise_expansion}
	\xi_F=\sum_{n\in\Z^3}f_n\dot{B}_n,\qquad  \xi_G=\sum_{n\in\Z^3}g_n\dot{W}_n,
\end{equation}
for independent, symmetric $3\times 3$-matrix valued Brownian motions $(B_n)_{n\in \Z^3}$ with covariance  
\begin{equation}\label{Eq_noise_matrix_covariation}
	\E[ B_n^{ij} (1)B_n^{kl} (1)] \,=\,  \delta_{ik} \delta_{jl}+\delta_{il}\delta_{jk},
\end{equation}
and  standard $3$-dimensional Brownian motions $(W_n)_{n\in \Z^3}$. The correlation functions  from \eqref{Eq_Noise_Correlation}  can then be recovered via the formulae
\[
F(x,y) \,=\, \sum_{n\in \Z^3} f_n(x)f_n(y) ,\qquad 
G(x,y) \,=\, \sum_{n\in \Z^3} g_n(x)g_n(y) .
\]
Thereby, the following assumption, which we impose throughout, requires a sufficiently smooth spatial correlation structure of $\xi_F$ and $\xi_G$.

\begin{assumption}[Regularity and symmetries of the noise]\label{Ass:noise}
	The families $(f_n)_{n\in \Z^3}$ and $(g_n)_{n\in \Z^3}$ {consist of continuously differentiable functions} satisfying
	\begin{align}\begin{split}
			&F_1=\sum_{n\in\Z^3}|f_n|^2,\qquad \sum_{n\in\Z^3}f_n\nabla f_n=0,\qquad F_2=\sum_{n\in\Z^3}|\nabla f_n|^2,\\
			&G_1=\sum_{n\in\Z^3}|g_n|^2,\qquad \sum_{n\in\Z^3}g_n\nabla g_n=0,\qquad G_2=\sum_{n\in\Z^3}|\nabla g_n|^2
			,
			\label{Eq_decay_condition_FG}
		\end{split}
	\end{align}
	for constants $F_1,F_2,G_1, G_2<\infty$ independent of $x$.  
\end{assumption}
The cancellations appearing in \eqref{Eq_decay_condition_FG} are actually equivalent to $F_1$ and $G_1$ being independent of $x$, but anyways displayed for later reference. These conditions  express that the driving noises are probabilistically stationary in space, and are satisfied by any spatial mollification of white noise.
In any case, with the above choice of noise, the It\^o formulation of \eqref{Eq_new_var_formal}--\eqref{Eq_new_u} reads
\begin{align}\begin{split}\label{Eq_new_var}
		\d u\,=\,&\nabla\cdot(\nabla_{\sym}u) \,\d t\,-\, \Pi \nabla \cdot (u \otimes u)\,\d t\,+\, \frac{F_1}{4}\Delta u\,\d t\, -\, \frac{1}{\sqrt{2}}\sum_{n \in \Z^3}\Pi \nabla\cdot(\psi f_n   \,\d B_n),\qquad \nabla\cdot u=0 ,
		\\
		\d \psi\,  =\, &(1+F_1{/2} +G_1/2)\Delta \psi \, \d t + \,\frac{1}{\psi }\bigl(
		|\nabla \psi|^2 \,+\, |\nabla_\sym u|^2 
		\bigr) \,\d t\,-\, \nabla \cdot ( u\psi )\,\d t \,-\, (F_2{/2}+ G_2/8) \psi \,\d t \\&- \sum_{n\in \Z^3} \nabla \cdot ( \psi g_n \,\d W_n ) \,
		+\, \frac{1}{2} \sum_{n \in \Z^3} \psi \nabla \cdot ( g_n \, \d  W_n) \,-\,  \frac{1}{\sqrt{2}}\sum_{n\in \Z^3}(\nabla  u : f_n\,\d  B_n),
		\\ (u(0),\psi(0))\, =\,&( u_0,\psi_0).
	\end{split}
\end{align}
The derivation of the latter is deferred to  Appendix \ref{sec-app-B}.
\subsection{The necessity of the  transformation $(u,\vt) \rightarrow(u,\psi)$} \label{SS:transformation_is_necessary}
To contrast \eqref{Eq_new_var}, we also consider the It\^o formulation of the equation for $\vt$ from \eqref{Eq_intro}, which reads
	\begin{align}\begin{split}\label{eqn1001}
		\d \vt\,  =\, &(1+F_1{/2} +G_1/2)\Delta \vt \, \d t + \,\bigl( (1+F_1/2) |\nabla_\sym u|^2 \,-\,F_1  |\nabla  \sqrt{\vt}|^2
		\bigr) \,\d t\,-\, \nabla \cdot ( u\vt  )\,\d t \,\\&-\, F_2 \vt \,\d t \,- \sum_{n\in \Z^3} \nabla \cdot ( \vt  g_n \,\d W_n ) \,-\,  \sum_{n\in \Z^3}( \sqrt{\vt }\nabla  u : f_n\,\d  B_n),
	\end{split}
\end{align}
as can be seen by direct calculation or from Theorem \ref{prop_prop_SL} below. Regarding the construction of weak solutions, one of the main advantages of \eqref{Eq_new_var} compared to the above is that there the gradient nonlinearities admit a sign. 
Indeed, identifying them along an approximation procedure is a central challenge in the construction of weak solutions. Next to the informal energy balance \eqref{eqn333}, we use for the this purpose an additional a priori estimate for \eqref{Eq_intro}, which is a bound on the dissipation of the quantity \eqref{eqn1009}. 
By means of It\^o's formula, the latter is seen to be a semimartingale decomposing into 
\begin{align}\label{eqn1010}
	- \int_{0}^t\int_{\T^3}\frac{	|\nabla \vt|^2}{\sqrt{8 \vt^3}}
 \,+\, \frac{|\nabla_\sym u|^2 }{\sqrt{2 \vt}} \,\d x\,\d s\,+\, (F_2{/2}+ G_2/8)\int_0^t \int_{\T^3} 
	 \sqrt{2 \vt} 
	\,\d x \, \d s
\end{align}
and a local martingale. Since both, the functional \eqref{eqn1009} itself and the second term of \eqref{eqn1010} can be controlled  using the total  energy, we obtain a bound on
\begin{align}\label{eqn1011}
	\E \biggl[
	\int_0^T \int_{\T^3} \frac{	|\nabla \vt|^2}{\sqrt{8 \vt^3}}
	\,+\, \frac{|\nabla_\sym u|^2 }{\sqrt{2 \vt}} \,\d x\,\d t
	\biggr],
\end{align}
after taking the expectation and a stopping time argument, if necessary. This estimate, together with a uniform bound on the energy turns out to not be enough to identify the nonlinearity $ (1+F_1/2) |\nabla_\sym u|^2 \,-\,F_1  |\nabla  \sqrt{\vt}|^2
$ in \eqref{eqn1001} along an approximation procedure for \eqref{Eq_intro}. 

One of the key observations of this work is that this situation can be improved by passing to the system \eqref{Eq_new_var} for $(u,\psi)$. 
In these variables, the estimated quantity \eqref{eqn1011} can be expressed as 
\begin{align}\label{eqn1012}
	\E \biggl[
	\int_0^T \int_{\T^3}\frac{1}{\psi}\bigl(	|\nabla \psi|^2
	\,+\, |\nabla_\sym u|^2\bigr) \,\d x \,\d t
	\biggr],
\end{align}
providing  an $L^1_{\omega, t, x}$-estimate on the gradient nonlinearity  on right-hand side of \eqref{Eq_new_var}.
Then, by weak-${}^*$ compactness in the space of Radon measures of the integrand in \eqref{eqn1012}  along a  suitable approximation scheme, we are able to construct solutions to the relaxed equation
\begin{align}\label{eqn1021}
\d \psi\,  =  q \,\d t\,+ \, &(1+F_1{/2} +G_1/2)\Delta \psi \, \d t\,-\, \nabla \cdot ( u\psi )\,\d t \,-\, (F_2{/2}+ G_2/8) \psi \,\d t \\&- \sum_{n\in \Z^3} \nabla \cdot ( \psi g_n \,\d W_n ) \,
+\, \frac{1}{2} \sum_{n \in \Z^3} \psi \nabla \cdot ( g_n \, \d  W_n) \,-\,  \frac{1}{\sqrt{2}}\sum_{n\in \Z^3}(\nabla  u : f_n\,\d  B_n),
\end{align}
where $q $ is a random measure on  $[0,T]\times \T^3$. Since now the nonlinearity admits  a sign, Fatou type arguments allow us to show  that this measure satisfies
\begin{align}\label{eqn1022}
	q \,\ge\, \frac{1}{\psi}(|\nabla \psi|^2 \,+ \,|\nabla_\sym u|^2),
\end{align}
relating it to the desired term. This is similar to the relaxations  in the weak solution theory for the compressible Navier--Stokes--Fourier system from   \cite{feireisl_book, feireisl}, where available a priori estimates are also too weak to identify the nonlinearity.

Let us point out that, while passing to this inequality may seem like a strong weakening of the notion of solutions, the above is compensated by global admissibility, i.e., the estimate 
\begin{align}\label{Eq:glob_Ad}
	\sup_{t}\EE (u_t,\psi_t) \,\le\, \EE (u_0,\psi_0),\qquad \text{$\P$-a.s.,}
\end{align} for the constructed solutions. Indeed, assuming $u$ solves the velocity equation from \eqref{Eq_new_var}, while $(\psi,q)$ satisfy the relaxed equations \eqref{eqn1021}--\eqref{eqn1022}, an informal application of It\^o's formula yields
\[
\EE(u_t, \psi_t ) \,=\, \EE(u_0,\psi_0)\,+\,\int_{[0,t]\times \T^3} \psi \, \d q \,-\, \int_0^t \int_{\T^3} |\nabla \psi|^2 \,+ \,|\nabla_\sym u|^2 \,  \d x \, \d s .
\]   
Then \eqref{Eq:glob_Ad} dictates that the above is actually an equality and the same holds for \eqref{eqn1022}. These formal considerations are confirmed by the fact that we are able to prove weak--strong uniqueness of globally admissible weak solutions, i.e., that given the existence of a  strong solution, weak solutions in the above sense are unique.

In conclusion, only after passing to the variables $(u,\psi)$ we obtain a meaningful notion of weak solutions, given the expected a priori estimates for the incompressible Navier--Stokes--Fourier system with thermal noise. This situation is  for example comparable to the introduction of renormalized solutions for the Boltzmann equation in \cite{DP_Lions_Boltzmann}, where only after   composing the density function with $\log(1 +\cdot )$ an attainable notion of  solutions is  defined.   In any case,  the temperature $\vt$ can of course be reconstructed via $\vt = \psi^2/2$.

\subsection{Definition of weak and strong solutions}\label{SS:defi_sol}

Here, we rigorously define  weak and strong solutions to \eqref{Eq_new_var}, for which we achieve the assertions \eqref{item_a}, \eqref{item_b} and \eqref{item_c}. In both, we require the variable $\psi$ to be at least non-negative, as it models the (rescaled) square-root of the  temperature variable $\vt$. For weak solutions moreover, we replace the equation for $\psi$ by its relaxed version \eqref{eqn1021}--\eqref{eqn1022} and we define global admissibility in a rigorous fashion. We remark that we use throughout this manuscript the following convention, when giving meaning to nonlinearities involving $1/\psi$.
\begin{convention}\label{conv_inverse}For a function  $\psi \colon [0,T]\times \T^3 \to \R$ we set $1/\psi =\infty$ on the set $\{\psi \le  0\}$.
\end{convention}

\begin{definition}[Globally admissible weak martingale solution to \eqref{Eq_new_var}]\label{defi_solution}
	Let Assumption \ref{Ass:noise} be satisfied, $u_0 \in L^2(\T^3; \R^3)$ be divergence-free and  $\psi_0\in L^2(\T^3)$ be non-negative.
	A  \emph{weak martingale solution 
		to \eqref{Eq_new_var}} 
	consists of a probability space $(\Omega,\A,\P)$, a filtration $\F$ satisfying the usual conditions, independent $3$-dimensional standard $\F$-Brownian motions $(W_n)_{n\in \Z^3}$ and $3\times 3$-matrix valued $\F$-Brownian motions $(B_n)_{n\in \Z^3}$ with covariance structure \eqref{Eq_noise_matrix_covariation}, a weakly continuous and adapted  $L^2(\T^3;\R^3)$-valued process
	$u$, and  progressively measurable   $\psi \in L^\infty(0,T;L^2( \T^3))$ and  $q\in \M_+([0,T]\times \T^3)$, $\P$-a.s., such that the following holds:
	\begin{enumerate}[(i)]
		\item \label{Item_defi_sol_1} We have the additional regularity {$\nabla u\in L^{7/5}([ 0,T]\times \T^3 ;\R^{3\times 3})$ and $\nabla \psi \in  L^{7/5}([0,T]\times \T^3;\R^3)$}, $\P$-almost surely.
		\item  \label{Item_defi_sol_2} It holds  $\psi \ge 0$, $\P\otimes \d t \otimes \d x$-a.e., and  $q \ge \frac{1}{\psi}(|\nabla \psi|^2 + |\nabla_\sym u|^2)$,  as measures  on $[0,T]\times \T^3$, $\P$-almost surely. 
		\item  \label{Item_defi_sol_3} For every divergence-free $\vp\in C^{\reg}(\T^3;\R^3)$ it holds
		\begin{align}\begin{split}
				\label{eqn9}& \int_{\T^3}  u_t\cdot \vp\,\d x \,-\, 
				\int_{\T^3} u_0 \cdot \vp\,\d x
				\\&\quad=\,  \int_0^t \int_{\T^3} -  \Bigl(\frac{1}{2}+ \frac{F_1}{4}\Bigr) \nabla u: \nabla  \vp \,+\, u\otimes u:\nabla \vp \,\d x\, \d s \,+\, \frac{1}{\sqrt{2}}\sum_{n \in \Z^3}\int_0^t \int_{\T^3}\psi f_n \nabla \vp\,\d x :\d B_{n,s},
			\end{split}
		\end{align}
		and $\nabla \cdot u_t = 0$ for all $t\in [0,T]$, $\P$-almost surely.
		\item  \label{Item_defi_sol_4}
		For every $\zeta\in C_c^{\reg}([0,T) \times \T^3)$ it holds
		\begin{align}\begin{split}\label{Eq_psi_weak}
				&-\,
				\int_0^T \int_{\T^3} \psi  \partial_t \zeta  \, \d x\, \d t \,-\, \int_{\T^3}
				\psi_0 \zeta
				\,\d x 
				\\&\quad=\,  \int_0^T \int_{\T^3} - \bigl(1+ {F_1}/2 +G_1/2\bigr) \nabla \psi\cdot \nabla  \zeta \,+\,\psi u\cdot \nabla \zeta   \,-\, (F_2/2+G_2/8) \psi \zeta \,\d x\,\d t \,+\, \int_{[0,T]\times\T^3} \zeta \,\d q
				\\&\qquad + \sum_{i=1}^3
				\sum_{n \in \Z^3}\int_0^T \int_{\T^3} \psi g_n \partial_i \zeta  + \frac{1}{2} \psi  \zeta \partial_ig_n \,\d x \, \d W_{n,t}^i
				\, +\, \frac{1}{\sqrt{2}}\sum_{i,j =1}^3\sum_{n \in \Z^3}\int_0^T \int_{\T^3}u_j\partial_i (\zeta f_n) \,\d x \,\d B_{n,t}^{ji},
			\end{split}
		\end{align}  $\P$-almost surely.
	\end{enumerate}
	A weak martingale solution to \eqref{Eq_new_var} is moreover called \emph{globally admissible}, whenever 
	\begin{equation}\label{energy-inequality-intro}\mathrm{ess \, sup}_{t\in [0,T]}
		\frac{1}{2}\int_{\T^3}
		|u_t|^2 +\psi_t^2 \,\d x \,\le\, 
		\frac{1}{2}\int_{\T^3}
		|u_0|^2 +\psi_0^2 \,\d x,
	\end{equation}
	$\P$-almost surely.
\end{definition}
\begin{remark}Let us comment on the above definition. 
	\begin{enumerate}[(i)]\label{rem_weak_sol}
		\item \label{rem_weak_sol_I1} 
		We call $\psi$ and $q$ progressively measurable, if their restrictions to $[0,t]$ yield  weakly-${}^*$ $\F_t$-measurable maps, for any $t\in [0,T]$. E.g., for $q$ that means that the random variable $\int_{[0,t]\times \T^3} \zeta \d q$ is $\F_t$-measurable for each $\zeta \in C([0,t] \times \T^3)$. Since the filtration is assumed to be right-continuous, this is equivalent to $\psi$ and $q$ being adapted in the sense of distributions as defined in \cite[Definition 2.2.13]{BFH}. 
		\item \label{rem_weak_sol_I2} The regularity required in Definition \ref{defi_solution}~\eqref{Item_defi_sol_1} is the one resulting from our proof of existence of globally admissible weak martingale solutions given in Section \ref{Sec:weak_ex} and plays a role in the proof of weak-strong uniqueness given in Section \ref{Sec:WS_uniqueness}. It can also be used to obtain improved integrability of $u$ and $\psi$: Together with the already assumed $L^\infty_tL^2_x$-regularity of $\psi$, e.g., it implies  
		\[
		\| \psi  \|_{L^{7/5}(0,T; {L}^{21/8}(\T^3 )) } \,\lesssim \, \|\psi \|_{L^{7/5}(0,T;W^{1,7/5}(\T^3 ))} \,<\,\infty ,
		\]
		$\P$-a.s., using the Sobolev embedding theorem. Then H\"older's inequality  yields that
		\begin{equation}\label{eqn2727724}
			\|\psi \|_{L^{7/3}([0,T]\times\T^3)} \,\le\, \| \psi \|_{L^\infty(0,T;L^2(\T^3))}^{2/5}\| \psi  \|_{L^{7/5}(0,T; {L}^{21/8}(\T^3 )) }^{3/5} \,<\,\infty,
		\end{equation}
		while  analogous considerations for $u$ result in 
		\begin{equation}\label{eqn2727724u}
			\|u\|_{L^{7/5}(0,T;W^{1,7/5} (\T^3;\R^3))}\,+\, 	\|u\|_{L^{7/3}([0,T]\times\T^3;\R^3)} \,<\,\infty,
		\end{equation}
		$\P$-almost surely.
		\item
		To give meaning to the remaining properties \eqref{Item_defi_sol_2}--\eqref{Item_defi_sol_4}, the central  consequence of the regularity assertion from Definition \ref{defi_solution}~\eqref{Item_defi_sol_1} is that $\nabla_\sym u$ and $\nabla \psi$ are defined as functions, which is needed in \eqref{Item_defi_sol_2}. In particular, the only terms on right-hand sides of \eqref{eqn9} and  \eqref{Eq_psi_weak}, which need more than the already imposed  $L^\infty_tL^2_x$-regularity of $u$ and $\psi$ to be defined, are the ones stemming from the Laplacian. While, as written above, we use that $\nabla u$ and $\nabla\psi$ lie in $L^1_{t,x}$ by \eqref{Item_defi_sol_1} to define the latter, this could be avoided by a further integration by prats. 
		\item \label{rem_weak_sol_I3} The variable $\psi$ is a priori only defined as a random function on $[0,T]\times \T^3$. However, we show in Lemma \ref{Lemma_cadlag} below that, as a consequence of Definition \ref{defi_solution}~\eqref{Item_defi_sol_4}, $\psi$ has a version which is c\`adl\`ag in the weak topology of $L^2(\T^3)$.  Throughout, this version is denoted by $\psi^+$  and the resulting left-continuous version by $\psi^-$.
	\end{enumerate}
\end{remark}

We also give our definition of strong solutions to \eqref{Eq_new_var}. Here, we do not pass to the relaxation \eqref{eqn1021}--\eqref{eqn1022} of the equation for $\psi$. Moreover, we do not define a notion of global admissibility for strong solutions as it turns out that they automatically conserve the total  energy \eqref{eqn333}.

\begin{definition}[Local strong solution] \label{defi_strong_sol}
	Let Assumption \ref{Ass:noise} hold, $(\Omega,\A,\P)$ be a probability space equipped with a filtration $\F$ satisfying the usual conditions and $3$-dimensional standard $\F$-Brownian motions $(W_n)_{n\in \Z^3}$ and $3\times 3$-matrix valued $\F$-Brownian motions $(B_n)_{n\in \Z^3}$ with covariance structure \eqref{Eq_noise_matrix_covariation} be independent. Then a local strong solution to \eqref{Eq_new_var}, defined up to a stopping time $\tau >0  $, is a tuple of continuous, adapted  processes $U \colon [0,\tau ) \to C^1(\T^3;\R^3)$ and $\Psi \colon [0,\tau ) \to C^1(\T^3)$, such that for a strictly increasing localizing sequence $0\le \tau_j\nearrow \tau $ the following holds for any $j\in \mathbb N$:
	\begin{enumerate}[(i)]
		\item \label{Item_pos_Phi} We have $\mathrm{inf_{(t,x) \in [0,\tau_j] \times \T^3}} \Psi(t,x)>0$, $\P$-almost surely.
		\item \label{Item_U_eq} For every divergence-free $\vp\in C^{\reg}(\T^3;\R^3)$ it holds
		\begin{align}\begin{split}\label{Eq_V_strong}
				& \int_{\T^3}  U_t\cdot \vp\,\d x \,-\, 
				\int_{\T^3} U_0 \cdot \vp\,\d x
				\\&\quad=\,  \int_0^t \int_{\T^3} -  \Bigl(\frac{1}{2}+ \frac{F_1}{4}\Bigr) \nabla U: \nabla  \vp \,+\, U\otimes U:\nabla \vp \,\d x\,\d s \,+\, \frac{1}{\sqrt{2}}\sum_{n \in \Z^3}\int_0^t \int_{\T^3}\Psi f_n \nabla \vp\,\d x :\d B_{n,s},
			\end{split}
		\end{align}
		and $\nabla \cdot U_t = 0$  for all $t\in [0,\tau_j]$, $\P$-almost surely.
		\item \label{Item_Psi_eq}
		For every $\eta\in C^{\reg}( \T^3)$ it holds
		\begin{align}\begin{split}\label{Eq_phi_strong}
				&\int_{\T^3} \Psi_t  \eta  \, \d x  \,-\, \int_{\T^3}
				\Psi_0 \eta
				\,\d x \,=\,  \int_0^t \int_{\T^3}  \frac{\eta}{\Psi} \bigl(
				|\nabla \Psi|^2+|\nabla_\sym U|^2
				\bigr) \,\d x \, \d s
				\\&\qquad+\int_0^t \int_{\T^3} - \bigl(1+ {F_1}/2 +G_1/2\bigr) \nabla \Psi\cdot \nabla  \eta \,+\,\Psi U\cdot \nabla \eta   \,-\, (F_2/2+G_2/8) \Psi \eta \,\d x\,\d s 
				\\&\qquad + \sum_{i=1}^3
				\sum_{n \in \Z^3}\int_0^t \int_{\T^3} \Psi g_n \partial_i \eta  + \frac{1}{2} \Psi \eta \partial_ig_n   \,\d x \, \d W_{n,s}^i
				\,-\, \frac{1}{\sqrt{2}}\sum_{n \in \Z^3}\int_0^t\int_{\T^3}\eta f_n \nabla U\,\d x :\d B_{n,s},
			\end{split}
		\end{align} for all $t\in [0,\tau_j ]$, $\P$-almost surely.
	\end{enumerate}
\end{definition}
\begin{remark}\label{rem_strong_sol}
	Also regarding Definition \ref{defi_strong_sol} some comments are in order.
	\begin{enumerate}[(i)]
		\item The  $C^0_tC^1_x$-regularity of $U$ and $\Psi$ required in Definition \ref{defi_strong_sol} is essentially tailored to the proof of weak-strong uniqueness given in Section \ref{Sec:WS_uniqueness}. At the same time, the fact that we allow for strong solutions which are defined only up to a stopping time comes from  the proof of existence given in Section \ref{Sec_strong_ex}. Whether strong solutions  to the incompressible Navier--Stokes equation exist globally in time in $d=3$ is of course a famous open problem in deterministic fluid mechanics.
		\item \label{Item_Rem_Strong_sol2} As remarked above, strong solutions have sufficient regularity in order to justify the calculations leading to the conservation of the total energy \eqref{eqn333} and are thereby automatically globally admissible, in the sense that they satisfy \eqref{Eq:glob_Ad} with the supremum taken over  $[0,\tau)$. The proof of this is given in Lemma \ref{prop_cons_of_energy} below. Along the same vein, the imposed regularity allows to characterize strong solutions through the incompressible  Navier--Stokes--Fourier system written in terms of the temperature, as we show in Proposition  \ref{prop_prop_SL}. Thereby, our results concerning strong solutions can be rephrased as results on strong solutions to \eqref{Eq_intro}.
	\end{enumerate}
\end{remark}
\subsection{Proof strategies}
\label{sec:ideas_of_proof}

We comment on the proof strategies for  assertions \eqref{item_a}, \eqref{item_b} and \eqref{item_c}. Their precise statement is deferred to the beginning of their devoted sections, see Theorem \ref{thm-weak-existence}, Theorem \ref{Prop_ex_strong_solution} and  Theorem \ref{thm_WS_main} below. Let us mention that, while we focused in Subsection \ref{SS:transformation_is_necessary} on challenges when proving the weak existence of solutions to the system \eqref{Eq_intro} for $(u,\vt)$, passing to the equation  for $(u,\psi)$ yields a key advantage in all three endeavors.

\medskip
\noindent\textbf{\eqref{item_a}  Global in time weak  existence: Theorem \ref{thm-weak-existence}.} The proof of weak existence is based on a stochastic compactness argument, see, e.g., the works \cite{BF20, BFH, brzesniak_moytl, FG95,MR2118862} in the context of fluid equations. 
As detailed in Subsection \ref{SS:transformation_is_necessary}, passing to the variables $(u,\psi)$ alleviates a central issue regarding the proof of weak existence. Moreover, the most  natural a priori estimate of the equation is the pathwise energy estimate
\begin{equation}
	\mathrm{ess \, sup}_{t\in [0,T]}
	\frac{1}{2}\int_{\T^3}
	|u_t|^2 + \psi_t^2 \,\d x \,\le \, 
	\frac{1}{2}\int_{\T^3}
	|u_0|^2 + \psi_0^2 \,\d x,
\end{equation}
for solutions to \eqref{Eq_new_var}. Next to this, as pointed out in Subsection \ref{SS:transformation_is_necessary}, there is also a bound on 
\begin{align}\label{eqn1023}
\E \biggl[
\int_0^T \frac{1}{\psi}\bigl(	|\nabla \psi|^2
\,+\, |\nabla_\sym u|^2\bigr)  \,\d t
\biggr],
\end{align}
in terms of the initial data, which can be obtained by integrating the equation for $\psi$ in \eqref{Eq_new_var} in space.
The main challenge is to find an approximation to \eqref{Eq_new_var}, which is compatible with these estimates and also ensures non-negativity of the variable $\psi$ in the limit.  We achieve this with the system
\begin{align}\begin{split}\label{Eq_approx_intro}
\d u_{\delta,\epsilon} \, =\, &  \nabla\cdot(\nabla_{\sym}u_{\delta,\epsilon})\,\d t \,-\, \epsilon \Delta^2 u_{\delta,\epsilon}\, \d t\,-\,
				\Pi \div   (u_{\delta,\epsilon} \otimes   u_{\delta,\epsilon})  \,\d t\,+\, \frac{F_1}{4}\Delta u_{\delta,\epsilon}\, \d t \\& - \frac{1}{\sqrt{2}}\sum_{n \in \Z^3}\Pi \nabla\cdot(\psi_{\delta,\epsilon} f_n  \, \d B_n)
				,\qquad \nabla \cdot u_{\delta,\epsilon}\,=\, 0,		\\
		\d \psi_{\delta,\epsilon}  \,=\, &(1+F_1 /2+G_1/2)\Delta \psi_{\delta,\epsilon} \,\d t\, -\,
		\epsilon\Delta^2 \psi_{\delta,\epsilon}\,
		+ \,h_\delta(\psi_{\delta,\epsilon})\bigl(
		|\nabla \psi_{\delta,\epsilon}|^2 \,+\, |\nabla_\sym u_{\delta,\epsilon}|^2+ \epsilon
		\bigr) \,\d t
		\\&- \,\nabla \cdot( u_{\delta,\epsilon}\psi_{\delta,\epsilon} )\,\d t\, -\, (F_2/2+ G_2/8) \psi_{\delta,\epsilon}\, \d t \\&- \sum_{n\in \Z^3} \nabla \cdot ( \psi_{\delta,\epsilon} g_n \,\d W_n ) \,
		+\, \frac{1}{2} \sum_{n \in \Z^3} \psi_{\delta,\epsilon} \nabla \cdot ( g_n \, \d  W_n) \,-\,  \frac{1}{\sqrt{2}}\sum_{n\in \Z^3}(\nabla  u_{\delta,\epsilon} : f_n\,\d  B_n),
		\\ (u_{\delta,\epsilon}(0),\psi_{\delta,\epsilon}(0)) \,=\,&( u_0,\psi_0),
	\end{split}\end{align}
for $\delta,\epsilon\in (0,1)$,
and taking first $\delta \to 0$ and then $\epsilon \to 0$. The function   $h_{\delta}(r)$ appearing in the above is an approximation of  $ 1/r$ according to the following assumption.

\begin{assumption}\label{Ass:nu_N}
	Let $h_\delta \in C_b^1(\mathbb{R})$ be decreasing on $\mathbb{R}$ such that $h_\delta(r) = 1/r$ for all $r \ge \delta$, and
	$$
	h_\delta(r) \nearrow 
	\begin{cases}
		1/r, & r>0,\\
		\infty, & r\le 0,
	\end{cases}
	$$
	as $\delta \to 0$.
\end{assumption}
Let us describe the roles of the different modifications of the originial system: Firstly, replacing the singular nonlinearity $1/r$ by $h_\delta$ slightly tames the nonlinearities in the equation for $\psi_{\delta,\epsilon}$. Together with the regularizing effect of the operator $-\epsilon\Delta^2$ added to both equations, this allows us to construct weak martingale solutions to \eqref{Eq_approx_intro} based on a standard Galerkin approximation using solely the energy estimate 
\begin{align}\label{eqn1024}
	\frac{1}{2}\int_{\T^3}
	|u_{\delta,\epsilon}(t)|^2 + \psi_{\delta,\epsilon}^2(t) \,\d x\,+\, \epsilon \int_0^t\int_{\T^3} | \Delta u_{\delta,\epsilon} |^2 + | \Delta \psi_{\delta,\epsilon} |^2\,\d x\,\d t  \,\le \, 
	\frac{1}{2}\int_{\T^3}
	|u_0|^2 + \psi_0^2 \,\d x\,+\, \epsilon t
\end{align}
for \eqref{Eq_approx_intro}. 
Here, the special choice of $h_\delta$ is relevant, since $rh_\delta(r) \le 1$ for all $r\in \R$. We remark that at this point, we do not know whether $\psi_{\delta,\epsilon}$ is non-negative, since the function $h_\delta(r)$ lacks the repulsive effect of $1/r$ for small values of $r$. 

Then, to pass to the limit $\delta\to 0$, we close also the approximate version of \eqref{eqn1023}, namely 
\begin{equation}\label{eq1026}
\E \biggl[
\int_0^T \int_{\T^3} 
h_\delta(\psi_{\delta,\epsilon}) \bigl(|\nabla \psi_{\delta,\epsilon}|^2 + |\nabla_{\mathrm{sym}} u_{\delta,\epsilon}|^2 + \epsilon \bigr)
\,\d x\,\d t 
\biggr]\,\le\, C(u_0,\psi_0,T),
\end{equation}
independently of $\epsilon$ and $\delta$. This yields tightness in the weak topology of Radon measures, which together with the estimate \eqref{eqn1024}, which is uniform in $\delta$, yields sufficient tightness to identify its limit as a solution to 
\begin{align}
		\d u_{\epsilon} \, =\, &  \nabla\cdot(\nabla_{\sym}u_{\epsilon})\,\d t \,-\, \epsilon \Delta^2 u_{\epsilon}\, \d t\,-\,
		\Pi \div   (u_{\epsilon} \otimes   u_{\epsilon})  \,\d t\,+\, \frac{F_1}{4}\Delta u_{\epsilon}\, \d t \\& - \frac{1}{\sqrt{2}}\sum_{n \in \Z^3}\Pi \nabla\cdot(\psi_{\epsilon} f_n  \, \d B_n)
		,\qquad \nabla \cdot u_{\epsilon}\,=\, 0,		\\
		\d \psi_{\epsilon}  \,=\, &(1+F_1 /2+G_1/2)\Delta \psi_{\epsilon} \,\d t\, -\,
		\epsilon\Delta^2\psi_{\epsilon}\,
		+ \d q_\epsilon\,- \,\nabla \cdot( u_{\epsilon}\psi_{\epsilon} )\,\d t\, -\, (F_2/2+ G_2/8) \psi_{\epsilon}\, \d t \\&- \sum_{n\in \Z^3} \nabla \cdot ( \psi_{\epsilon} g_n \,\d W_n ) \,
		+\, \frac{1}{2} \sum_{n \in \Z^3} \psi_{\epsilon} \nabla \cdot ( g_n \, \d  W_n) \,-\,  \frac{1}{\sqrt{2}}\sum_{n\in \Z^3}(\nabla  u_{\epsilon} : f_n\,\d  B_n),
		\\ (u_{\epsilon}(0),\psi_{\epsilon}(0)) \,=\,&( u_0,\psi_0).
\end{align}
The appearing  measure $q_\epsilon$ on $[0,T]\times \T^3$ is related to $(u_\epsilon,\psi_\epsilon)$ via 
\begin{equation} \label{eq1025}q_\epsilon \,\ge\, 
\frac{1}{\psi_{\epsilon}} \bigl(
|\nabla \psi_{\epsilon}|^2 \,+\, |\nabla_\sym u_{\epsilon}|^2+ \epsilon
\bigr),
\end{equation}
analogously to \eqref{eqn1022}. 
At this point, we obtain that $\psi_\epsilon$ is almost everywhere positive due to the presence of $\epsilon$ on the right-hand side of \eqref{eq1025}, in light of Convention \ref{conv_inverse}. 

For the second passage $\epsilon\to  0$, we do not have the additional bound from the dissipation of $-\epsilon \Delta^2$ in \eqref{eqn1024} at our disposal. Instead, we proceed with the estimate
\begin{equation}
	\E \biggl[
	\int_0^T \int_{\T^3} \frac{1}{
	\psi_{\delta,\epsilon} }\bigl(|\nabla \psi_{\delta,\epsilon}|^2 + |\nabla_{\mathrm{sym}} u_{\delta,\epsilon}|^2 + \epsilon \bigr)
	\,\d x\,\d t 
	\biggr]\,\le\, \E \bigl[
	\|q_\epsilon\|_{\M([0,T]\times \T^3)}
	\bigr]\,\le\, C(u_0,\psi_0,T),
\end{equation}
which follows from \eqref{eq1026} and \eqref{eq1025} by a lower semicontinuity argument. Employing the interpolation with the total energy estimate from Remark \ref{rem_weak_sol}~\eqref{rem_weak_sol_I2}, the above entails uniform moment bounds on 
\begin{align}
	\|\psi_\epsilon\|_{L^{7/5}(0,T;W^{1,7/5} (\T^3))}\,+\, 	\|\psi_\epsilon\|_{L^{7/3}([0,T]\times\T^3)}\,+\,
		\|u_\epsilon\|_{L^{7/5}(0,T;W^{1,7/5} (\T^3;\R^3))}\,+\, 	\|u_\epsilon\|_{L^{7/3}([0,T]\times\T^3;\R^3)}.
\end{align}
The tightness resulting from the latter turns out to be enough to identify the limit as a globally admissible weak martingale solution to \eqref{Eq_new_var}, in the sense of Definition \ref{defi_solution}, as $\epsilon\to  0$. 
In particular, we use that since $\psi_\epsilon$ is almost everywhere positive for each $\epsilon>0$, the limit $\psi$ is non-negative, as desired.

\medskip
\noindent\textbf{\eqref{item_b} Local in time strong existence: Theorem \ref{Prop_ex_strong_solution}.} While weak solutions to \eqref{Eq_new_var} can be constructed for initial values from the energy space, it is often the case  for SPDEs that simultaneously strong solutions  exist uniquely and locally in time for more regular data, see, e.g.,
\cite{AV_CMP, BFH, CFH_Euler} for results of this type regarding fluid equations. Here, our ansatz to construct strong solutions is to employ the Banach fixed-point theorem based on the coercivity of the linear part of \eqref{Eq_new_var} together with local Lipschitz estimates for the nonlinearities in sufficiently regular norms.
The latter is what requires us to work in considerably stronger spaces than the natural energy space $L^2_x$, namely we impose that the initial data has $H^3_x$-regularity. 
To  show higher order coercivity of the linear part by means of a commutator argument, we need to impose also higher regularity of the noise coefficients $f$ and $g$ from Assumption \ref{Ass:noise}, namely that both lie in $C_x^4(\ell^2)$. 
Under these assumptions,  using suitable modifications of the nonlinearities compatible with our Lipschitz estimates, we conclude then the existence of a unique solution to the modified equation. By choosing modifications  that coincide with the original nonlinearities close to the initial data, we ensure that the solution to the modified equation solves also \eqref{Eq_new_var} up to some stopping time $\tau>0$. In other words, we obtain  a local strong solution in the sense of Definition \ref{defi_strong_sol}, as desired.

We point out that 
this idea is very natural and  was recently explored systematically and in abstract settings  in \cite{AV_PTRF}, see \cite{rock_fully_monotone} for a similar result for nonlinear leading order operators. For the sake of accessibility however, we  give here a self-contained proof independent of these works. 
We  stress that also our proof of strong existence is significantly simplified by passing to the new variables $(u,\psi)$. 
Indeed, the linear gradient noise terms in \eqref{Eq_new_var} have the same order of scaling as the leading order heat operator due to the characteristic scaling law of Brownian motion. Thereby, we cannot treat them as a lower order perturbation as, e.g., in the semigroup approach to SPDEs \cite{da2014stochastic}, but instead need to obtain coercivity estimates for the deterministic and stochastic linear part as a whole. 
As a consequence, since these gradient noise terms are  also in \eqref{Eq_intro} of leading order, the additional nonlinearities there would require to treat the system as a quasilinear stochastic evolution equation. While the latter is possible and was done in systematically in \cite{AV19_QSEE_1,AV19_QSEE_2,hornung2019quasilinear}, we may after the variable transformation $(u,\vt)\to(u,\psi)$ proceed in a simpler semilinear setting.

\medskip
\noindent\textbf{\eqref{item_c} The weak-strong uniqueness principle: Theorem \ref{thm_WS_main}.} To prove weak-strong uniqueness we employ a relative energy argument. For more information on the latter we refer to the survey \cite{wiedemann} and the references therein and in the context of stochastic fluid equations to \cite{BFH}. We stress that also here we profit from passing to the system for $(u,\psi)$ as the energy functional \eqref{eqn333} is strictly convex in these variables. Therefore, the unique minimizer of the natural relative energy \eqref{Eq_rel_energy_intro}
 is given by $(u,\psi) = (V,\Phi)$. 
Assuming that $(u,\psi)$ 
and $(V,\Phi)$ are a weak and a strong solution to \eqref{Eq_new_var}, respectively, we may decompose 
\begin{align}\begin{split}\label{eqn111}&
		\EE( u_t,\psi_t | V_t,\Phi_t )\,-\, 	\EE( u_0,\psi_0  | V_0,\Phi_0 )
		\,=\, \frac{1}{2}\int_{\T^3} |u_t|^2 + \psi_t^2- |u_0|^2 - \psi_0^2\,\d x 
		\\& \qquad - \,\int_{\T^3} u_t \cdot V_t -  u_0 \cdot V_0 \,\d x \, -\, \int_{\T^3} \psi_t \Phi_t -  \psi_0  \Phi_0 \,\d x,
		\end{split}
\end{align}
since, as we prove below, strong solutions conserve energy. Assuming that $(u,\psi)$ is globally admissible, i.e., that we have the bound \eqref{Eq:glob_Ad}, we may further estimate the above  by 
\begin{align}\label{eqn112}-\,\int_{\T^3} u_t \cdot V_t
	 -  u_0 \cdot V_0 \,\d x \, -\, \int_{\T^3} \psi_t \Phi_t -  \psi_0  \Phi_0 \,\d x.
\end{align}
To proceed one needs to compute the latter by means of It\^o's product rule, at which point one is faced with the problem that $\psi$ only exists as a random distribution in $L^\infty(0,T;L^2(\T^3))$. To deal with the latter,  inspired by  \cite{DV10}, we show that because $(u,\psi)$ is a weak solution, $\psi$ admits a version $\psi^+$ which is c\`adl\`ag  in the weak topology of $L^2(\T^3)$. 
Moreover, we prove that $\psi^+$ satisfies a version of \eqref{Eq_psi_weak} where one only tests in the space variable, identifying in particular the semimartingale decomposition  of the c\`adl\`ag  process $\psi_{\epsilon,t}^+(x)= \eta_\epsilon * \psi^+_t(x)$ for a convolution kernel $\eta_\epsilon$. Convolving also $u$, $V$ and $\Phi$, we pointwisely justify It\^o's product rule using  stochastic analysis tools for  c\`adl\`ag  semimartingales in order  to deal with the temporal  jumps of $\psi_{\epsilon,t}^+$, cf.\  \cite{protter}. Based on the regularity of $(u,\psi)$, $(V,\Phi)$ and the noise coefficients imposed in Assumption \ref{Ass:noise}, we further justify integration in the space variable and subsequently letting  $\epsilon\to 0$ in the resulting formula.
All in all, this yields  the following pathwise relative energy estimate:
\begin{align}\begin{split}\label{eqn1050}
&
\EE( u_t,\psi^+_t | V_t,\Phi_t ) \,-\, \EE( u_0,\psi_0 | V_0,\Phi_0 )\\&\quad\le\,  \int_0^t \Bigl(2 \|\nabla_\sym V\|_{C(\T^3; \R^{3\times 3})} \,+\,   \| \nabla \Phi \|_{C([0,t]\times \T^3; \R^{3})} \Bigr)\EE( u_s ,\psi_s  | V_s ,\Phi_s ) \,\d s .
	\end{split}
\end{align}
Somewhat surprisingly all contributions of the noise have vanished, which is linked to the  special structure of the thermal noise term, and in contrast to, e.g.,  relative energy estimates resulting from multiplicative It\^o noise \cite{BFH}. In any case, by an application of Gr\"onwall's lemma, \eqref{eqn1050} implies the weak-strong uniqueness property.

\subsection{Comparison to the literature}
There is an extensive list of literature addressing the physical  and mathematical analysis of the Navier--Stokes--Fourier system and related models. Therefore, we only mention the most closely related works here, and refer for more information to the references therein.

\medskip
\noindent\textbf{Physical background of the Navier--Stokes--Fourier system.}
The original formulation of the Navier--Stokes--Fourier system  traces back to the works of Navier and Stokes, who addressed momentum conservation incorporating viscosity, and Fourier, who established the theory of heat conduction \cite{Fourier1822, Navier1822, Stokes1845}. For a modern treatment focusing  on the incompressible variant, we refer to the book of  Landau and Lifshitz  \cite[Section 50]{LL87}. The inclusion of the effects of thermal noise in the model is attributed to these authors as well, and is detailed  in \cite[Chapter IX]{lifshitz}  by Lifshitz and Pitaevskii. For a numerical study of the resulting full Landau--Lifshitz--Navier--Stokes equation, we refer to \cite{BGW07} by Bell, Garcia and Williams. More information on the  GENERIC framework, from which the Landau--Lifshitz--Navier--Stokes system may be derived similarly to Subsection \ref{der:gen}, we refer to the devoted monograph \cite{Ottinger} by \"Ottinger.

\medskip
\noindent\textbf{Selected mathematical research on  deterministic fluid equations.}
As pointed out below \eqref{Eq_Noise_Correlation}, the equation for the velocity of an incompressible fluid can be solved independently of the temperature in the deterministic case, cf.\ \eqref{eqn1037}, following the work of   Leray \cite{Leray}. A corresponding analysis is carried out in the monograph of Lions \cite[Section 3.4]{lions_book_1}, where he alternatively proposes to model both variables as a system, which allows to prove the conservation of energy. Generalizations of the latter proof of existence to  state-dependent  material coefficients coupling both variables are obtained by Naumann \cite{naumann} and Bul\'{\i}\v{c}ek, Feireisl, and M\'{a}lek \cite{BFM09}.

For compressible Navier-Stokes equations, the equation for the mass density is in contrast naturally coupled to the velocity. For an analysis of the latter we refer to \cite{lions_book_2}. The first existence proof for the full Navier--Stokes--Fourier system, modeling simultaneously  the temperature, was obtained in \cite{feireisl_book} and lateron extended also to state-dependent material coefficients  by Bresch and Desjardins \cite{Bresch_Desj} and then by Feireisl and  Novotn\'y \cite{feireisl}. A corresponding result on weak-strong uniqueness based on a relative entropy argument, which goes back to Dafermos \cite{dafermos}, is the content of \cite{FN12}. This improved the previous line of work on weak-strong uniqueness for fluid equations based on relative energies initiated by  Ladyzhenskaya  \cite{lad}, Prodi \cite{prodi} and Serrin \cite{serrin} for incompressible Navier--Stokes. For more information on the latter we refer to the survey by Wiedemann \cite{wiedemann}.

\medskip
\noindent\textbf{Selected mathematical works on  stochastic fluid equations.}
The existence of weak martingale solutions to the incompressible Navier--Stokes equation goes back to Flandoli and Gatarek \cite{FG95}. Brze\'zniak and Motyl \cite{brzesniak_moytl} extended these results to equations posed on unbounded domains and Mikulevicius and Rozovskii \cite{MR2118862} to Navier--Stokes equations with  transport noise. We also mention the recent contribution by \cite{gess2026probabilisticallystrongsolutionsstochastic} by Lasarzik and the first author, in which it is pointed out that one can also obtain probabilistically strong solutions based on the approximations of \cite{FG95}, at the expense of allowing them to be measure-valued.   Research works by  Breit, Feireisl and Hofmanov{\'a}, in which global weak existence, local strong existence and weak-strong uniqueness for compressible Navier--Stokes equations are obtained are moreover detailed in their monograph \cite{BFH}. A suitable modification of the convective term leads in the contrary to tamed Navier--Stokes equations, for which well-posedness in $d=3$ was shown by R\"ockner and Zhang \cite{rockner_tamed}. For local in time strong solution theories we refer   to the work  \cite{CFH_Euler} by Crisan, Flandoli and Holm regarding Euler equations 
and to  \cite{AV_CMP} by Agresti and Veraar regarding Navier--Stokes equations posed in scaling-critical spaces. The latter builds on ideas of Krylov's $L^p$-theory \cite{krylov_Lp}. Probably most close to the current work is the article \cite{BF20}, in which the compressible Navier--Stokes--Fourier system is considered and the velocity equation is perturbed by a multiplicative It\^o  noise term. Then the existence of martingale solutions is shown, extending the  existence result from \cite{feireisl} to the stochastic setting.

The situation that the stochastic force is more rough in space, so that the Navier--Stokes equation becomes a singular stochastic PDE,  was firstly addressed in the work \cite{DP_D} by Da Prato and Debussche. While they treated the less singular, two-dimensional regime, the three-dimensional setting was subsequently explored  by  combining  paracontrolled   calculus and    convex integration techniques by Hofmanov\'a, Zhu and Zhu \cite{HZZ_NS_3d}. Also based on paracontrolled calculus, the aforementioned work \cite{DP_D} was further improved in the two-dimensional setting by Hairer and Rosati \cite{hairer_rosati}. In the full sub-critical regime of noises a treatment using energy solutions was recently completed by  Gubinelli and Turra \cite{gubinelli_turra}. In the critical and supercritical regime on the contrary, a treatment of the Navier--Stokes equation requires the introduction of an artificial smallness along a regularization procedure. The effective limiting dynamics  were investigated by  
Cannizzaro and Kiedrowski \cite{cannizzaro}, by
Jin and Perkowski in  \cite{jin2024fractionalstochasticlandaulifshitznavierstokes} and by Kotitsas, Romito, Yang and Zhu in  \cite{kotitsas2025gaussianfluctuationsstochasticlandaulifshitz}.

Besides its solution theory, a natural question regarding the Landau--Lifshitz--Navier--Stokes equation is  to make the connection to  underlying mesoscopic  particle systems mathematically precise. The latter was recently addressed by Heydecker together with the first and last author in \cite{GHW24}, where a restricted large deviation principle  was shown, which is  consistent with the   lattice particle system model proposed by Quastel and Yau \cite{QY98}.

\subsection{Notation}Throughout,  $(\Omega,\A,\mathbb{P})$ is a potentially varying probability space   endowed with a filtration $\F$ satisfying the usual conditions. Expectations with respect to $\mathbb{P}$ are denoted by $\mathbb{E}$. We use standard notation regarding function spaces, and remark only that $\M_+$ denotes the non-negative cone in the space of Radon  measures $\M$. Weak and weak-${}^*$ topologies are moreover indicated by an additional $w$ and $w^*$, respectively, so that $(L^2([0,T];L^2(\mathbb{T}^3 )),w)$, e.g., stands for $L^2([0,T];L^2(\mathbb{T}^3 ))$ equipped with its weak topology. We denote by $\Pi$ the Helmholtz projection onto divergence-free vector fields and by $\nabla_{\sym}$ the symmetrized gradient operator.
We write moreover $\{e_k\}_{k\in\mathbb{Z}^3}$ for the Fourier basis on $\mathbb{T}^3$, so that for each $m\in\mathbb{N}$, the Galerkin projection $P_m$ onto the Fourier modes with frequency $|k|\le m$ is defined by
\begin{equation}
P_m f := \sum_{|k|\le m} \langle f, e_k \rangle e_k .\label{Fourier_proj}
\end{equation}
Similarly, let $\{e_\zeta\}_{\zeta\in\mathcal{B}}$ be the eigenfunctions of the  Stokes operator $-\Pi \Delta$, indexed by the set
\begin{equation}\label{Stokes_proj}
\mathcal{B} := \bigl((\mathbb{Z}^3\setminus\{0\}) \times \{\pm 1\}\bigr) \cup \bigl(\{0\} \times \{\pm 1,0\}\bigr).
\end{equation}
Each eigenfunction $e_\zeta$, with $\zeta=(k,\theta)$, is a normalized plane wave with frequency $2\pi k$, and the corresponding eigenvalue is given by $\lambda_\zeta = 4\pi^2 |k|^2$. We define accordingly  Galerkin projections $\Pi_m$ onto the span of $\{e_\zeta : \zeta \in \mathcal{B}_m\}$ by
\begin{equation}\label{eq: Pm}
\Pi_m f := \sum_{\zeta\in \mathcal{B}_m} \langle f, e_\zeta \rangle e_\zeta ,
\end{equation}
 where $\mathcal{B}_m$ consists of all indices with $|k|\le m$.
Finally, for each $\delta>0$, $h_\delta(r)$ is a regularization of $1/r$ in accordance with Assumption \ref{Ass:nu_N}.

\section{Properties of weak and strong solutions}\label{sec-2}

In this section, we prove the properties stated in Subsection \ref{SS:defi_sol} of weak and strong solutions to \eqref{Eq_new_var}. Besides their interest on their own, they are also useful in later proofs. We start with the fact stated in Remark \ref{rem_weak_sol}~\eqref{rem_weak_sol_I3}, that the $\psi$-variable of a weak solution, which is a priori only defined as a random element of $L^\infty(0,T; L^2(\T^3))$, admits a c\`adl\`ag version in the weak topology of $L^2(\T^3)$. The proof is inspired by \cite[Proposition 8]{DV10}, where a similar statement is shown for solutions to scalar stochastic conservation laws.
\begin{lemma}[C\`adl\`ag version of $\psi$ for weak solutions]\label{Lemma_cadlag}Let Assumption \ref{Ass:noise} hold and $(\Omega,\A,\P)$, $\F$, $(W_n)_{n\in \Z^3}$, $(B_n)_{n\in \Z^3}$, 
	$(u,\psi ,q)  $ be a weak martingale solution to \eqref{Eq_new_var} with divergence-free initial value 
	$u_0 \in L^2(\T^3; \R^3)$ and $\psi_0\in L^2(\T^3)$ non-negative. Then there exists a $\P\otimes \d t $ version $\psi^+$ of $\psi$ which is c\`adl\`ag in the weak topology of $L^2(\T^3)$. Denoting the process of left limits by 
	\begin{equation}\label{Eq_left_lim}\psi^-(t)\,=\,
		\begin{cases}
			\psi_0,& t = 0 ,\\
			\lim_{s\nearrow t}  \psi^+(s), & t\in (0,T],
		\end{cases}
	\end{equation}
	this version satisfies moreover, $\P$-a.s.,
	\begin{align}\label{eqn3}
		\int_{\T^3}	\eta \bigl(\psi_{t}^{+}   - \psi_{t}^{-} \bigr) \,\d x \,=\, \int_{\{t\}\times \T^3} \eta  \,\d q,
	\end{align} and 
	\begin{align} \begin{split}\label{eqn4}&
			\int_{\T^3} \psi^+_t \eta   \, \d x \,-\, \int_{\T^3}
			\psi_0 \eta 
			\,\d x 
			\\&\quad=\, \int_0^t \int_{\T^3} -  \bigl(1+ {F_1}/2 +G_1/2\bigr) \nabla \psi \cdot \nabla  \eta \,+\,\psi u\cdot \nabla \eta   \,-\, (F_2/2+G_2/8) \psi \eta \,\d x\,\d s \,+\, \int_{[0,t]\times\T^3} \eta \,\d q
			\\&\qquad + \sum_{i=1}^3
			\sum_{n \in \Z^3}\int_0^t \int_{\T^3}\psi g_n \partial_i \eta  + \frac{1}{2} \psi  \eta \partial_ig_n  \,\d x\,  \d W_{n,s}^i
			\, +\, \frac{1}{\sqrt{2}}\sum_{i,j =1}^3\sum_{n \in \Z^3}\int_0^t \int_{\T^3}u_j\partial_i (\eta f_n) \,\d x \,\d B_{n,s}^{ji},
		\end{split}
	\end{align}
	for all $\eta \in C^{\reg}(\T^3)$ and  $t\in [0,T)$.
\end{lemma}
\begin{proof}
	We start by obtaining c\`adl\`ag versions of $\int_{\T^3} \psi \eta \d x$ for each $\eta\in C^{\reg}(\T^3)$. To this end, we insert   $\zeta =  \vp\otimes\eta$ for $\vp \in C_c^{\reg}([0,T))$  in the temporally weak formulation \eqref{Eq_psi_weak}, which results in the identity
	\begin{align}&
		-\,
		\int_0^T \vp'(t)\int_{\T^3} \psi \eta \, \d x \,  \d t \,-\,  \vp(0)\int_{\T^3}
		\psi_0 \eta
		\,\d x
		\\&\quad=\,  \int_0^T \vp (t)\int_{\T^3} - \bigl(1+ {F_1}/2 +G_1/2\bigr) \nabla \psi\cdot \nabla  \eta \,+\,\psi u\cdot \nabla \eta   - (F_2/2+G_2/8) \psi \eta \,\d x  \,\d t \,+\, \int_{[0,T]\times\T^3}   \vp \otimes\eta\,\d q
		\\&\qquad + \sum_{i=1}^3
		\sum_{n \in \Z^3}\int_0^T \vp(t) \int_{\T^3}\psi g_n \partial_i \eta  + \frac{1}{2} \psi  \eta \partial_ig_n  \,\d x\,  \d W_{n,t}^i
		\,+\, \frac{1}{\sqrt{2}}\sum_{i,j =1}^3\sum_{n \in \Z^3}\int_0^T \vp(t)\int_{\T^3}u_j\partial_i (\eta f_n) \,\d x \,\d B_{n,t}^{ji}.
	\end{align}
	After rewriting  $\vp(t) = -\int_0^T \mathbf{1}_{[t, T]}(s)\vp'(s) $, we apply the deterministic and stochastic Fubini theorem, see, e.g., \cite{stoch_fubini} for the latter,  to  find that
	\begin{align} \label{eqn_1}
		-\,
		\int_0^T \vp'(t) \biggl(\int_{\T^3} \psi \eta \,\d x - \Theta_\eta (t)\biggr)  \,\d t \,-\,  \vp(0)\int_{\T^3}
		\psi_0 \eta
		\,\d x\,
		=\,  \int_{[0,T]} \vp(t)  \int_{\T^3} \eta \,  \d K_t \, \d \nu,
	\end{align}
	where $ q = \nu \otimes K$ is the disintegration\footnote{We take $\nu\in \M_+([0,T])$ and $K(t,\cdot)$ to be a probability measure on $\T^3$, for $t\in [0,T]$.} of $q$ and
	\begin{align}&
		\Theta_\eta(t) \,=\,  \int_0^t \int_{\T^3} - \bigl(1+ {F_1}/2 +G_1/2\bigr) \nabla \psi\cdot \nabla  \eta \,+\,\psi u\cdot \nabla \eta   \,-\, (F_2/2+G_2/8) \psi \eta \,\d x  \,\d s
		\\&\qquad + \sum_{i=1}^3
		\sum_{n \in \Z^3}\int_0^t  \int_{\T^3}\psi g_n \partial_i \eta  + \frac{1}{2} \psi  \eta \partial_ig_n  \,\d x  \,\d W_{n,s}^i
		\,+\, \frac{1}{\sqrt{2}}\sum_{i,j =1}^3\sum_{n \in \Z^3}\int_0^t \int_{\T^3}u_j\partial_i (\eta f_n) \,\d x \,\d B_{n,s}^{ji}
	\end{align}
	is a continuous stochastic process. The identity \eqref{eqn_1}, implies that, $\P$-a.s., $\int_{\T^3} \psi \eta \,\d x - \Theta_\eta \in BV(0,T)$  and therefore there exists a representative in the  $\d t$-equivalence class of the latter which is c\`adl\`ag. By the continuity of $\Theta_\eta$, the same holds for $\int_{\T^3} \psi \eta \,\d x $.
	
	Next, we leverage the above to obtain a c\`adl\`ag representative of $\psi$ itself. For this purpose, we fix a countable dense subset $\D\subset C^{\reg}(\T^3)$, and  let $ (\psi_\eta^+)_{\eta\in \D}$ be the family of  c\`adl\`ag versions of $\int_{\T^3} \psi \eta \,\d x $, which exists $\P$-almost surely. We also record that due to   $\psi \in L^\infty (0,T; L^2(\T^3))$, $\P$-a.s., and the  right-continuity of $\psi_\eta^+$, we have the bound 
	\begin{equation}\label{eqn2}
		|\psi_\eta^+(t)| \,\le \,\| \eta \|_{L^2(\T^3)} \|\psi\|_{L^\infty (0,T; L^2(\T^3))} ,\qquad 
		\eta\in \D,\,	t\in [0,T),
	\end{equation}
	$\P$-almost surely. The density of $\D\subset L^2(\T^3)$ allows us to construct, $\P$-a.s., for every $t\in [0,T)$ a linear functional $ \eta \mapsto \psi_\eta^+(t)$, that can by the Riesz representation theorem be represented by some  $\psi^+_t\in L^2(\T^3)$. Since then, $\P$-a.s.,
	\[
	\int \psi^+_t \eta\,\d x \,=\, \psi_\eta^+(t)\,=\, \int \psi_t \eta\,\d x,
	\]	
	for every $\eta\in \D$ and $\d t$-a.e.\ $t\in [0,T)$, we find that $\psi^+$ is a $\P\otimes \d t$-version of $\psi$. Additionally, using  \eqref{eqn2}, we see that, $\P$-a.s., the constructed process $\psi^+$ obeys the bound 
	\[
	\| \psi^+_t\|_{L^2(\T)}\, \le\,  \|\psi\|_{L^\infty (0,T; L^2(\T^3))} , \qquad t\in [0,T),
	\]
	which implies that the limit $\lim_{s\searrow t}\psi^+(s) = \psi^+(t)$ is attained in the weak topology of $L^2(\T)$ by a subsequence-subsequence argument.
	Since we also have $\P$-a.s.\ the estimate
	\begin{equation}
		\bigl|\lim_{s\nearrow t} \psi_\eta^+(s)\bigr| \,\le \,\| \eta \|_{L^2(\T^3)} \|\psi\|_{L^\infty (0,T; L^2(\T^3))} ,	\qquad  \eta\in \D,\, t\in (0,T] ,
	\end{equation}
	it follows analogously that $\psi^+$ is actually c\`adl\`ag in the weak toplogy of $L^2(\T)$ and consequently there exists the process of left limits \eqref{Eq_left_lim}.
	
	It remains to verify the identities \eqref{eqn3}--\eqref{eqn4}, for which we observe from \eqref{eqn_1} that the c\`adl\`ag representative of $\int_{\T^3} \psi \eta \,\d x  $  for $\eta\in \D$ is $\P$-a.s.\ given by the formula 
	\begin{equation}
		\psi_\eta^+(t) \,=\, 
		\int \psi_0 \eta \,\d x \,+\, \int_{[0,t]}  \int_{\T^3} \eta \,  \d K_s\,  \d \nu  \,+\, \Theta_\eta(t), \qquad t\in [0,T),
	\end{equation}
	allowing us to also calculate its left limits as
	\begin{equation}
		\lim_{s\nearrow t }\psi_\eta^+(s) \,=\,
		\int \psi_0 \eta \,\d x \,+\, \int_{[0,t)}  \int_{\T^3} \eta \,  \d K_s\,  \d \nu  \,+\, \Theta_\eta(t), \qquad t\in (0,T].
	\end{equation}
	Since $ \int \psi^+\eta \d x = \psi_\eta^+ $, $\P$-a.s., for  $\eta\in \D$, the desired formulae \eqref{eqn3} and \eqref{eqn4} follow by a density argument.
\end{proof}

Next, we turn our attention to strong solutions. As announced in Remark \ref{rem_strong_sol}~\eqref{Item_Rem_Strong_sol2}, the additional regularity allows us to characterize strong solutions by the It\^o formulation of \eqref{Eq_intro}. In particular, all results obtained in this manuscript concerning strong solutions to \eqref{Eq_new_var} can be equivalently rephrased in terms of strong solutions to \eqref{Eq_intro}, defined according to the following proposition.

\begin{proposition}[Characterization of strong solutions by the temperature formulation]\label{prop_prop_SL}
	Let Assumption \ref{Ass:noise} hold, $(\Omega,\A,\P)$ be a probability space equipped with a filtration $\F$ satisfying the usual conditions and $3$-dimensional standard $\F$-Brownian motions $(W_n)_{n\in \Z^3}$ and $3\times 3$-matrix valued $\F$-Brownian motions $(B_n)_{n\in \Z^3}$ with covariance structure \eqref{Eq_noise_matrix_covariation} be independent. Moreover, we let  $U\colon [0,\tau ) \to C^1(\T^3;\R^3)$ and $\Psi\colon [0,\tau )\to  C^1(\T^3)$ be continuous, adapted processes, defined up to a stopping time $\tau>0$, such that for a strictly increasing localizing sequence $0\le \tau_j\nearrow \tau $ properties \eqref{Item_pos_Phi} and \eqref{Item_U_eq} of Definition \ref{defi_strong_sol} holds. 
	Then, $(U,\Psi)$ is a local strong solution to \eqref{Eq_new_var}, 
	iff the temperature $\Theta = \Psi^2/2$ satisfies the following for any $j\in \mathbb N$:
	\begin{enumerate}[(i)']
		\setcounter{enumi}{2}
		\item \label{Item_Theta_eq}
		For every $\eta\in C^{\reg}( \T^3)$ it holds
		\begin{align}\begin{split}\label{Eq_Phi_strong_prime}
				&
				\int_{\T^3} \Theta_t \eta \, \d x \,-\, \int_{\T^3} \Theta_0 \eta \, \d x \,=\, \int_0^t \int_{\T^3}\eta \bigl((1+F_1/2)
				|\nabla_\sym U|^2 - F_1|\nabla \sqrt{\Theta}|^2\bigr) \, \d x \,\d s
				\\&\qquad+\int_0^t \int_{\T^3} - \bigl(1+ {F_1}/2 +G_1/2\bigr)  \nabla \Theta \cdot \nabla  \eta \,+\,\Theta U\cdot \nabla \eta \,-\, F_2 \Theta \eta \,\d x\,\d s
				\\&\qquad + \sum_{i=1}^3
				\sum_{n \in \Z^3}\int_0^t \int_{\T^3} 
				\Theta g_n \partial_i \eta  
				\,\d x \, \d W_{n,s}^i
				\,-\, \sum_{n \in \Z^3}\int_0^t\int_{\T^3}\eta f_n \sqrt{\Theta} \nabla U\,\d x :\d B_{n,s},
			\end{split}
		\end{align} for all $t\in [0,\tau_j ]$, $\P$-almost surely.
	\end{enumerate}
\end{proposition}
\begin{proof}
	We have to argue that, given \eqref{Item_pos_Phi} and \eqref{Item_U_eq} of Definition \ref{defi_strong_sol}, \eqref{Item_Psi_eq} of said definition implies \eqref{Item_Theta_eq}' above  and vice versa. Due to the regularity assumptions on $U$ and $\Psi$ together with Assumption \ref{Ass:noise}, both implications can be obtained by means of It\^o's formula as stated in  \cite[Proposition A.1]{DHV_16}. Indeed, 
	to go from \eqref{Item_Psi_eq} to \eqref{Item_Theta_eq}', we need to compute the evolution of 
	\[
	\int_{\T^3} \Theta_t \eta \, \d x \,=\, \frac{1}{2}\int_{\T^3} \Psi_t^2 \eta \, \d x,
	\]
	for $\Psi$ satisfying \eqref{Eq_phi_strong}. 
	Then, together with a stopping time argument to restore moments, the above mentioned version of It\^o's formula yields that
	\begin{align}&
		\frac{1}{2}\int_{\T^3} \Psi_t^2 \eta \, \d x \,-\, \frac{1}{2}\int_{\T^3} \Psi_0^2 \eta \, \d x \,=\, \int_0^t \int_{\T^3}\eta \bigl(
		|\nabla_\sym U|^2 - (F_1/2 +G_1/2)|\nabla \Psi|^2\bigr) \, \d x \,\d s
		\\&\qquad+\int_0^t \int_{\T^3} - \bigl(1+ {F_1}/2 +G_1/2\bigr) \Psi \nabla \Psi \cdot \nabla  \eta \,+\,\Psi^2 U\cdot \nabla \eta   \,+\, \eta\Psi U\cdot \nabla \Psi  \,-\, (F_2/2+G_2/8) \Psi^2 \eta \,\d x\,\d s
		\\& \qquad +  \frac{1}{2}\sum_{i=1}^3
		\sum_{n \in \Z^3}\int_0^t \int_{\T^3}
		\Bigl( -  \partial_i (\Psi g_n)   + \frac{1}{2} \Psi (\partial_ig_n)\Bigr)^2 \eta 
		\,\d x \, \d s
		\\&\qquad + \frac{F_1}{4}\sum_{i,j=1, \,i\ne j}^3
		\int_0^t \int_{\T^3}
		\bigl( (\partial_i U_j)^2 + \partial_i U_j\partial_j U_i\bigr) \eta 
		\,\d x \, \d s\,+\,\frac{ F_1 }{2}\sum_{i=1}^3 \int_0^t \int_{\T^3}
		(\partial_i U_i)^2  \eta 
		\,\d x \, \d s
		\\&\qquad + \sum_{i=1}^3
		\sum_{n \in \Z^3}\int_0^t \int_{\T^3} - \Psi \partial_i (\Psi g_n) \eta  + \frac{1}{2} \Psi^2  \eta \partial_ig_n  \,\d x  \, \d W_{n,s}^i
		\,-\, \frac{1}{\sqrt{2}}\sum_{n \in \Z^3}\int_0^t\int_{\T^3}\eta f_n \Psi \nabla U\,\d x :\d B_{n,s},
	\end{align}
	for all $t\in [0,\tau_j ]$, $\P$-almost surely.
	Here, we made use of the specific correlation structure \eqref{Eq_noise_matrix_covariation} of the matrix valued noise. To further rewrite the correction terms, we notice that 
	\begin{align}&
		\frac{1}{2}\sum_{i=1}^3
		\sum_{n \in \Z^3}
		\Bigl( -  \partial_i (\Psi g_n)   + \frac{1}{2} \Psi \partial_ig_n\Bigr)^2\\
		&\quad 
		=\,  \frac{1}{2}\sum_{i=1}^3
		\sum_{n \in \Z^3}
		\Bigl(    g_n\partial_i \Psi   + \frac{1}{2} \Psi \partial_ig_n \Bigr)^2\,=\, \frac{G_1}{2} |\nabla \Psi |^2 \,+\, \frac{G_2}{8} \Psi^2,
		\\&\frac{1}{4} \sum_{i,j=1, \,i\ne j}^3 \bigl((\partial_i U_j)^2 + \partial_i U_j\partial_j U_i\bigr) \,+\, \frac{1}{2}\sum_{i=1}^3 (\partial_iU_i)^2 \\&\quad =\,\frac{1}{8} \sum_{i,j=1, \,i\ne j}^3 \bigl(\partial_i U_j+ \partial_j U_i\bigr)^2 \,+\, \frac{1}{2}\sum_{i=1}^3 (\partial_iU_i)^2\, =\,  \frac{1}{2}|\nabla_\sym U|^2,
	\end{align}
	in light of the cancellation imposed in \eqref{Eq_decay_condition_FG}. 
	Inserting this in the above and using that $\nabla \cdot U  = 0 $ results in 
	\begin{align}&
		\frac{1}{2}\int_{\T^3} \Psi_t^2 \eta \, \d x \,-\, \frac{1}{2}\int_{\T^3} \Psi_0^2 \eta \, \d x \,=\, \int_0^t \int_{\T^3}\eta \bigl((1+F_1/2)
		|\nabla_\sym U|^2 - F_1/2 |\nabla \Psi|^2\bigr) \, \d x \,\d s
		\\&\qquad+\int_0^t \int_{\T^3} - \bigl(1+ {F_1}/2 +G_1/2\bigr) \Psi \nabla \Psi \cdot \nabla  \eta \,+\,\frac{1}{2}\Psi^2 U\cdot \nabla \eta \,-\, F_2/2 \Psi^2 \eta \,\d x\,\d s
		\\&\qquad + \sum_{i=1}^3
		\sum_{n \in \Z^3}\int_0^t \int_{\T^3} - \Psi  g_n \eta \partial_i \Psi  - \frac{1}{2} \Psi^2 \eta \partial_ig_n  \,\d x \, \d W_{n,s}^i
		\,-\, \frac{1}{\sqrt{2}}\sum_{n \in \Z^3}\int_0^t\int_{\T^3}\eta f_n \Psi \nabla U\,\d x :\d B_{n,s},
	\end{align}
	for  $t\in [0,\tau_j ]$, $\P$-a.s., 
	by further integrating by parts. 
	Finally, we use that 
	\[
	\Psi^2/2 \,=\, \Theta,\qquad  \nabla \Psi \,=\, {\sqrt{2} \nabla \sqrt{\Theta}},\qquad \Psi\nabla \Psi \,=\, \nabla \Theta,
	\]
	to express the above in terms of $\Theta$, which yields the desired \eqref{Eq_Phi_strong_prime}.	
	For the reverse implication it suffices to undo the above based on
	\[
	\int_{\T^3} \Psi_t \eta \, \d x \,=\, \int_{\T^3} \sqrt{2\Theta_t} \eta \, \d x,
	\]
	 starting from \eqref{Eq_Phi_strong_prime}.
\end{proof}

Lastly, we rigorously  state in the following lemma that strong solutions conserve  \eqref{eqn333} and are therefore  automatically globally admissible, as announced in Remark \ref{rem_strong_sol}~\eqref{Item_Rem_Strong_sol2}. We omit the proof here since, like the proof of the preceding proposition, it relies on an application of It\^o's formula, which can be justified by  \cite[Proposition A.1]{DHV_16} or in this case also \cite[Theorem 4.2.5]{LR_book}. Then, the cancellations in the It\^o expansion of the total energy are the same as needed for the proof of weak--strong uniqueness given below. More precisely, the total energy conservation of strong solutions can be seen by setting $(u,\psi) = (V,\Phi)$ in \eqref{Eq_Ito_product}, so that the right-hand side of said equation cancels. The result is the following.

\begin{lemma}[Total energy conservation of strong solutions]\label{prop_cons_of_energy}
	Let Assumption \ref{Ass:noise} hold, $(\Omega,\A,\P)$ be a probability space equipped with a filtration $\F$ satisfying the usual conditions and $3$-dimensional standard $\F$-Brownian motions $(W_n)_{n\in \Z^3}$ and $3\times 3$-matrix valued $\F$-Brownian motions $(B_n)_{n\in \Z^3}$ with covariance structure \eqref{Eq_noise_matrix_covariation} be independent. Then, any local strong solution $(U,\Psi)$ to \eqref{Eq_new_var}, defined up to a stopping time $\tau>0$, satisfies 
	\begin{equation}\label{eqn_strong_energy_cons}
		\frac{1}{2}\int_{\T^3}   |U_t|^2 + |\Psi_t |^2\,\d x \,=\, 	\frac{1}{2}\int_{\T^3}   |U_0|^2 + |\Psi_0 |^2\,\d x,
	\end{equation}
	for all $t\in [0,\tau)$, $\P$-almost surely.
\end{lemma}

\section{Weak-strong uniqueness}\label{Sec:WS_uniqueness}

It is the aim of this section to verify a weak-strong uniqueness principle for \eqref{Eq_new_var}, i.e., that given the existence of a strong solution it coincides with any globally admissible  weak solutions. More precisely, we show the following.

\begin{theorem}[Weak-strong uniqueness]\label{thm_WS_main}
	Let Assumption \ref{Ass:noise} hold and $(\Omega,\A,\P)$, $\F$, $(W_n)_{n\in \Z^3}$, $(B_n)_{n\in \Z^3}$, 
	$(u,\psi ,q)  $ be a globally admissible weak martingale solution to \eqref{Eq_new_var} with initial value 
	$u_0 \in L^2(\T^3; \R^3)$ divergence-free and $\psi_0\in L^2(\T^3)$ non-negative. Assume additionally that, on the same stochastic basis, there exists a local strong solution $(V,\Phi)$ to \eqref{Eq_new_var} until a stopping time $\tau>0$. Then it holds
	\begin{equation} \label{Eq_WS_uniqueness} (u,\psi) = (V,\Phi), \qquad \text{on}\quad   \{(u_0,\psi_0) = (V_0,\Phi_0)\} \times [0,T\wedge \tau) , 
	\end{equation}
	$\P\otimes \d t$-almost everywhere.
\end{theorem}
The above is achieved by an a priori estimate on the relative energy of $(u,\psi)$ with respect to $(V,\Phi)$, defined in \eqref{Eq_rel_energy_intro}, which we recall to be
\begin{equation}\label{Eq_rel_energy_sec}
	\EE( u_t,\psi_t^+ | V,\Phi ) \,=\, \frac{1}{2}\int_{\T^3}   |u_t-V_t|^2 + |\psi_t^+ - \Phi_t |^2\,\d x.
\end{equation}
Here, for the relative energy to be meaningful at any time instance $t\in [0,T)$, we inserted the c\`adl\`ag version $\psi^+$ of $\psi$, provided by Lemma \ref{Lemma_cadlag}. In fact, this particular choice allows us to employ stochastic analysis methods for c\`adl\`ag semimartingales in order to obtain the following relative energy estimate.

\begin{proposition}[Relative energy estimate]\label{thm_rel_en_bound}
	Under the assumptions of Theorem \ref{thm_WS_main}, we have, $\P$-a.s.,
	\begin{equation}\label{Eq_rel_energy_bound}
		\EE( u_t,\psi^+_t | V_t,\Phi_t ) \,\le\, \exp \Bigl( \bigl( 2 \| \nabla_\sym V \|_{C([0,t]\times \T^3; \R^{3\times 3})} \,+\,   \| \nabla \Phi \|_{C([0,t]\times \T^3; \R^{3})}  \bigr)\cdot  t \Bigr) \EE( u_0,\psi_0 | V_0,\Phi_0 ),
	\end{equation}
	for $t\in [0,T\wedge \tau)$.
\end{proposition}
A remarkable fact about the relative energy estimate \eqref{Eq_rel_energy_bound} is that no stochastic integrals appear on the right-hand side. In particular, there is no need to take moments. Since
Proposition \ref{thm_rel_en_bound} immediately implies Theorem \ref{thm_WS_main}, the remainder of this section is dedicated to the proof of the preceding proposition. 
\subsection{Proof of the relative energy estimate}
In order to prove Proposition \ref{thm_rel_en_bound}, we essentially need to justify It\^o's product rule to obtain the It\^o expansion of \eqref{Eq_rel_energy_sec}. Then, due to the aforementioned cancellation of stochastic integrals, obtaining \eqref{Eq_rel_energy_bound} is rather direct. Consequently, the key step in the proof of Proposition \ref{thm_rel_en_bound} is the following technical lemma, which we formulate also for weak solutions, which are not globally admissible.

\begin{lemma}[It\^o expansion of the relative energy]\label{thm_rel_energy_est}Let Assumption \ref{Ass:noise} hold and $(\Omega,\A,\P)$, $\F$, $(W_n)_{n\in \Z^3}$, $(B_n)_{n\in \Z^3}$, 
	$(u,\psi ,q)  $ be a weak  martingale solution to \eqref{Eq_new_var} with initial value 
	$u_0 \in L^2(\T^3; \R^3)$ divergence-free and $\psi_0\in L^2(\T^3)$ non-negative. Assume additionally that, on the same stochastic basis, there exists a local strong solution $(V,\Phi)$ to \eqref{Eq_new_var} until a stopping time $\tau>0$. Then it holds
	\begin{align}\begin{split}\label{Eq_rel_energy_expansion}&
			\EE( u_t,\psi^+_t | V_t,\Phi_t ) \,=\, 	\EE( u_0,\psi_0  | V_0,\Phi_0 ) \,+\, \frac{1}{2}\int_{\T^3} |u_t|^2 + (\psi_t^+)^2- |u_0|^2 - \psi_0^2\,\d x 
			\\&\qquad  +\, 2 \int_0^t  
			\int_{\T^3} \nabla_{\sym} V:  \nabla_{\sym}  u + \nabla  \psi \cdot  \nabla \Phi 	\,\d x\, \d s  \,-\,
			\int_0^t \int_{\T^3}  \frac{\psi}{\Phi} \bigl(
			|\nabla \Phi|^2+|\nabla_\sym V|^2
			\bigr)\,\d x \,\d s \,-\, \int_{[0,t] \times \T^3}     \Phi   \,\d q
			\\&\qquad + \,  \int_0^t \int_{\T^3}   u \cdot      (( V \cdot \nabla)  V )-  \nabla   V:  u\otimes u \,\d x\,\d s \,  + \, \int_0^t \int_{\T^3}\psi V\cdot \nabla \Phi -\nabla \Phi \cdot (\psi u ) \,\d x \,\d s,
	\end{split}\end{align}
	for all $t\in [0,T \wedge \tau)$,
	$\P$-almost surely.
\end{lemma}
\begin{remark}\label{f_in_Cl2}In the proof of Lemma \ref{thm_rel_energy_est} we use that Assumption \ref{Ass:noise} implies that $f_n, g_n\in C^1(\T^3; \ell^2(\Z^3))$. Indeed, to show that $f(x_l) \to f(x)$ in $\ell^2(\Z^3)$ for some $x_l\to x $ in $\T^3$, we observe firstly that for any subsequence, there exists yet another subsequence (denoted again by $x_l$) satisfying $f(x_l)\to f(x)$ weakly in $\ell^2(\Z^3)$ by \eqref{Eq_decay_condition_FG} and the continuity of the individual $f_n$. Since the norms $\|f(x_l)\|_{\ell^2(\Z^3)} =\sqrt{F_1} = \|f(x)\|_{\ell^2(\Z^3)}$ coincide, we find that the convergence is strong. This yields that $f\in C(\T^3; \ell^2(\Z^3))$ and  applying the same argument to $\nabla f$, $g$ and $\nabla g$, we obtain the claimed regularity.
\end{remark}

\begin{proof}[Proof of Lemma \ref{thm_rel_energy_est}]
	We start by setting up the proof and observe for the increments of the relative energy that 
	\begin{align}\begin{split}\label{eqn5}&
			\EE( u_t,\psi^+_t | V_t,\Phi_t )\,-\, 	\EE( u_0,\psi_0  | V_0,\Phi_0 )
			\,=\, \frac{1}{2}\int_{\T^3} |u_t|^2 + (\psi_t^+)^2- |u_0|^2 - \psi_0^2\,\d x 
			\\& \qquad - \,\int_{\T^3} u_t \cdot V_t -  u_0 \cdot V_0 \,\d x \, -\, \int_{\T^3} \psi^+_t \Phi_t -  \psi_0  \Phi_0 \,\d x,
		\end{split}
	\end{align}
	by the conservation of energy of strong solutions stated in  Lemma \ref{prop_cons_of_energy}. To calculate the latter terms, we let $\tau_j$ be a localizing sequence as in Definition \ref{defi_strong_sol}, so that it suffices to prove the assertion for $t\in [0, T \wedge \tau_j)$, for fixed $j$.	
	We proceed now in three steps, in which we justify an application of It\^o's product rule to obtain the semimartingale decomposition of \eqref{eqn5}. Before that, we fix some notation and let $\eta \in C^\infty_c(B_{1/2}^{\R^3}(0))$ be a standard mollifier, i.e., $\eta \ge  0$ is radially symmetric and $\int_{\R^3 }\eta = 1$, for which we define an approximate identity by setting $\eta_\epsilon = \epsilon^{-3}\eta (\cdot /\epsilon)$ for $\epsilon \in (0,1)$. Accordingly, we observe that $u_\epsilon = \eta_\epsilon * u $, $\psi_\epsilon^+ = \eta_\epsilon * \psi^+ $, $V_\epsilon = \eta_\epsilon * V $ and $\Phi_\epsilon = \eta_\epsilon * \Phi$ satisfy analogous equations to \eqref{eqn9} and \eqref{eqn4} as well as \eqref{Eq_V_strong} and \eqref{Eq_phi_strong}, respectively, just with a mollified righ-hand side. The latter can be seen by inserting mollified test functions in these formulations. But since the mollified integrands converge  pointwisely, we find for any $x\in \T^3$, $\P$-a.s., the identities
	\begin{align}
		&  u_{\epsilon,t}(x) \,-\, 
		u_{\epsilon,0}(x) \,=\, - \, \frac{1}{\sqrt{2}} \Pi \nabla \cdot \biggl(\sum_{n \in \Z^3}\int_0^t \eta_\epsilon *( \psi f_n) \,	\d B_{n,s} \biggr) (x)
		\\&\qquad +\,  \int_{0}^t  \Bigl(\frac{1}{2}+ \frac{F_1}{4}\Bigr) \Delta(  \eta_\epsilon*u)(x) \,-\, \Pi \bigl( \nabla \cdot  (\eta_\epsilon * (u\otimes u) ) \bigr)(x) \,\d s  ,\\
		&\label{Eq_Psi_reg}
		\\&
		\psi^+_{\epsilon,t}(x)  \,-\, 
		\psi_{\epsilon,0}(x)  \,=\,   \int_{[0,t]}       \eta_\epsilon*K_s (x)\,\d\nu_s 
		\\&\qquad +\, \int_0^t \bigl(1+ {F_1}/2 +G_1/2\bigr) \Delta ( \eta_\epsilon* \psi)(x)  \,-\, \nabla \cdot  (\eta_\epsilon *(\psi u) )(x)   \,-\, (F_2/2+G_2/8) \eta_\epsilon*\psi (x)  \,\d s
		\\&\qquad + \sum_{i=1}^3
		\sum_{n \in \Z^3}\int_0^t - \partial_i (\eta_\epsilon* (\psi g_n) )(x)  + \frac{1}{2} \eta_\epsilon* ( \psi \partial_ig_n)  (x) \,  \d W_{n,s}^i
		\,-\, \frac{1}{\sqrt{2}}\sum_{n \in \Z^3}\int_0^t\eta_\epsilon*( f_n \nabla u)(x):\d B_{n,s},\\&
		V_{\epsilon,t}(x) \,-\, 
		V_{\epsilon,0}(x) \,=\, - \, \frac{1}{\sqrt{2}} \Pi \nabla \cdot  \biggl(\sum_{n \in \Z^3}\int_0^t \eta_\epsilon *( \Phi f_n) \,\d B_{n,s} \biggr) (x)
		\\&\qquad +\,  \int_{0}^t  \Bigl(\frac{1}{2}+ \frac{F_1}{4}\Bigr)\Delta (  \eta_\epsilon*V)(x) \,-\, \Pi  \bigl( \nabla \cdot \bigl(\eta_\epsilon * (  V\otimes V) \bigr)\bigr) (x) \,\d s  ,\\		
		&\label{Eq_Phi_reg}\\&\Phi_{\epsilon, t} (x )\,-\, 		\Phi_{\epsilon,0 }(x)\,=\,  \int_0^t\eta_\epsilon*\Bigl( \frac{1}{\Phi} \bigl(
		|\nabla \Phi|^2+|\nabla_\sym V|^2
		\bigr)\Bigr)(x) \,\d s
		\\&\qquad+\int_0^t \bigl(1+ {F_1}/2 +G_1/2\bigr) \Delta( \eta_\epsilon*   \Phi)(x) \,-\, \nabla \cdot  (\eta_\epsilon*  (\Phi V) )(x) \,-\, (F_2/2+G_2/8) \eta_\epsilon*\Phi (x) \,\d s 
		\\&\qquad + \sum_{i=1}^3
		\sum_{n \in \Z^3}\int_0^t  - \partial_i(\eta_\epsilon*  (\Phi  g_n) ) (x) + \frac{1}{2} \eta_\epsilon* ( \Phi \partial_ig_n) (x)  \, \d W_{n,s}^i
		\,-\, \frac{1}{\sqrt{2}}\sum_{n \in \Z^3}\int_0^t\eta_\epsilon * ( f_n \nabla V)(x) :\d B_{n,s},
	\end{align}
	for $t\in [0,T\wedge \tau_j)$. As in the proof of Lemma \ref{Lemma_cadlag}, we denote here the disintegration  of the measure $q$ by $\nu \otimes K$. Moreover, we used that the Helmholtz projection commutes with a scalar convolution as do differential operators.

	Let us also make a technical observation at this point. While the adaptedness of the processes  $u_\epsilon(x)$, $V_\epsilon(x)$ and $\Phi_\epsilon(x)$ is clear, whether $\psi^+_\epsilon(x)$ is adapted is more subtle and we use that $\F$ satisfies the usual conditions to confirm the latter. Indeed, since $\psi$ is progressively measurable, cf.\ Remark \ref{rem_weak_sol}~\eqref{rem_weak_sol_I1}, we conclude that for any non-negative $\vp \in C_c^\infty((0,1))$ with $\int_0^1\vp(t)\d t =1$ and $\vp_\delta(t) = \delta^{-1}\vp(t/\delta )$, the limit
	\[
	\psi_{\epsilon,t}^+(x) \,=\, \lim_{\delta\to  0} \int_{(t,t+\delta)\times \T^3} \vp_\delta(s-t)  \eta_\epsilon (x-y) \psi(y) \,\d y \,\d s,
	\]
	which is attained $\P$-a.s.,
	is $\F_{s}$-measurable for any $s>t$.  Because  $\F_t = \cap_{s>t} \F_s$, we conclude that  $\psi^+$ is adapted. As it moreover decomposes into a finite variation process and a local martingale, we can apply stochastic analysis methods for c\`adl\`ag semimartingales, cf.\ \cite[Chapter III]{protter}.
	
	\textit{Step 1 (It\^o's product rule for the regularized  velocities).}
	The velocity processes $u_{\epsilon}$ and $V_\epsilon$ have sufficient regularity to apply It\^o's formula for the squared norm as in \cite[Theorem 4.2.5]{LR_book}, and by the polarization identity
	\[
	\int_{\T^3}  u_{\epsilon}\cdot  V_\epsilon \,\d x \,=\, \frac{1}{4}\|u_{\epsilon}(x) + V_\epsilon(x) \|_{L^2(\T^3; \R^3)}^2 \,-\, \frac{1}{4}\|u_{\epsilon}(x) - V_\epsilon(x) \|_{L^2(\T^3; \R^3)}^2 ,
	\]
	also for their pairing in $L^2(\T^3;\R^3)$. To identify the correction term in the latter, we rewrite the noise term in the equation for $u_\epsilon^j$ as
	\[
	- \, \frac{1}{\sqrt{2}}  \sum_{k=1}^3 \sum_{n \in \Z^3}\int_0^t \partial_k (\eta_\epsilon *( \psi f_n) )(x)\,	\d B_{n,s}^{jk} \,
	+ \, \frac{1}{\sqrt{2}}  \sum_{k,l=1}^3 \sum_{n \in \Z^3}\int_0^t \partial_j \Delta^{-1}\partial_{lk} (\eta_\epsilon *( \psi f_n)) (x)\,	\d B_{n,s}^{lk} ,
	\]
	where an analogous formula holds for the equation for $V_\epsilon$. This, together with the covariance structure \eqref{Eq_noise_matrix_covariation} of the noise $B$, yields
	\begin{align}&
		\int_{\T^3}  u_{\epsilon,t} \cdot  V_{\epsilon,t} \,\d x \,-\, \int_{\T^3}  u_{\epsilon,0} \cdot  V_{\epsilon,0} \,\d x 
		\\&\quad= \,\int_0^t  
		\int_{\T^3} - \Bigl(\frac{1}{2}+ \frac{F_1}{4}\Bigr)\nabla V_\epsilon :  \nabla  u_\epsilon   \,-\, u_\epsilon   \eta_\epsilon * ( ( V\cdot \nabla ) V)  \,\d x \,\d s  \\&\qquad +\, \int_{0}^t \int_{\T^3}  -  \Bigl(\frac{1}{2}+ \frac{F_1}{4}\Bigr) \nabla V_\epsilon :  \nabla  u_\epsilon   \,+ \, \nabla   V_\epsilon :  \eta_\epsilon * (u\otimes u)  \,\d x\,\d s 
		\\&\qquad +\,\frac{1}{2} \sum_{j,k=1}^3 \sum_{n\in \Z^3} (1+\delta_{jk}) \int_0^t  
		\int_{\T^3}  (\partial_ k \eta_\epsilon * ( \psi f_n ) )  (\partial_ k \eta_\epsilon * ( \Phi f_n ) ) \,\d x \,\d s 
		\\&\qquad +\ \sum_{j,k,l=1}^3 \sum_{n\in \Z^3} \int_0^t  
		\int_{\T^3} (\partial_j \Delta^{-1} \partial_{lk} \eta_\epsilon * ( \psi f_n ) )  (\partial_j \Delta^{-1} \partial_{lk} \eta_\epsilon * ( \Phi f_n ) ) \,\d x \,\d s 
		\\&\qquad -\,2\sum_{k=1}^3 \sum_{n\in \Z^3} \int_0^t  
		\int_{\T^3}  (\partial_ k\eta_\epsilon * ( \psi f_n ) )  (\partial_ k \eta_\epsilon * ( \Phi f_n ) ) \,\d x \,\d s 
		\\&\qquad - \, \frac{1}{\sqrt{2}} \sum_{n \in \Z^3} \int_0^t\int_{\T^3} u \otimes (\nabla  \eta_\epsilon *( \Phi f_n)) \,\d x :  \d B_{n,s}
		\\&\qquad+ \, \frac{1}{\sqrt{2}}\sum_{n \in \Z^3}\int_0^t \eta_\epsilon *( \psi f_n) \nabla V_\epsilon \,\d x :	\d B_{n,s} 
		,
	\end{align}
	after suitable  integration by parts. Integrating by parts also in the correction term simplifies the above to 
	\begin{align}\begin{split}
			\label{eqn7}&
			\int_{\T^3}  u_{\epsilon,t} \cdot  V_{\epsilon,t} \,\d x \,-\, \int_{\T^3}  u_{\epsilon,0} \cdot  V_{\epsilon,0} \,\d x 
			\\&\quad= \,\int_0^t  
			\int_{\T^3} - \Bigl(1+ \frac{F_1}{2}\Bigr)\nabla V_\epsilon :  \nabla  u_\epsilon   \,-\, u_\epsilon\cdot    \eta_\epsilon * ( ( V\cdot \nabla ) V)    \,+ \, \nabla   V_\epsilon :  \eta_\epsilon * (u\otimes u)  \,\d x\,\d s 
			\\&\qquad +\, \sum_{n\in \Z^3}  \int_0^t  
			\int_{\T^3}  \nabla  \eta_\epsilon * ( \psi f_n ) \cdot \nabla \eta_\epsilon * ( \Phi f_n )  \,\d x \,\d s 
			\\&\qquad +  \, \frac{1}{\sqrt{2}} \sum_{n \in \Z^3}  \int_0^t \biggl( \int_{\T^3} \eta_\epsilon *( \psi f_n) \nabla V_\epsilon   \,-\,  u_\epsilon \otimes (\nabla  \eta_\epsilon *( \Phi f_n)) \,\d x \biggr):  \d B_{n,s}\,=:\, I_1^{(\epsilon)} \,+\, K_1^{(\epsilon)}\,+\, S_1^{(\epsilon)}.
		\end{split}
	\end{align}
	
	\textit{Step 2 (It\^o's product rule for the regularized temperature square roots).} We use the  integration by parts formula \cite[Corollary 2, p.68]{protter} for  c\`adl\`ag semimartingales and the description of their covariation process from \cite[p.70]{protter}, to obtain for any $x\in \T^3$, $\P$-a.s., that
	\begin{align}\begin{split}\label{eqn74375}
			&
			\psi^+_{\epsilon,t}(x)\Phi_{\epsilon, t} (x ) - \psi_{\epsilon,0}^+(x)\Phi_{\epsilon, 0} (x ) \,=\,  \int_0^t 	\psi_{\epsilon,s}^+(x) \,\d \Phi_{\epsilon, s} (x ) \,+\,
			\int_{(0,t]} \Phi_{\epsilon, s} (x )  \,\d 	\psi^+_{\epsilon,s}(x) \\&\qquad +\,
			\sum_{0<s\le t} (\psi^+_{\epsilon,s}(x) - \psi^-_{\epsilon,s}(x))(\Phi_{\epsilon,s}^+ (x ) - \Phi_{\epsilon,s}^- (x )   )
			\,+\, [	\psi^+_{\epsilon}(x), \Phi_{\epsilon} (x )   ]_t^c,
		\end{split}
	\end{align}
	for all $t\in [0,T\wedge \tau_j)$. 
	We comment on the above formula: Inspecting \cite[Corollary 2, p.68]{protter} one sees that there the integrands in the above formula are replaced by their left limits. For the first integral in \eqref{eqn74375} however, this doesn't play a role, since  $	\Phi_{\epsilon,t}(x)$ doesn't have any jumps and $\psi_{\epsilon, t}^+(x)$ differs from its left limit,  $\P$-a.s., only at most at countably many time instances by \eqref{eqn3}. For the same reason, we can  replace $\psi_{\epsilon, t}^+(x)$ in this term in the following by $\psi_{\epsilon, t} = \eta_\epsilon*\psi_t  $. In the second integral, we could insert $\Phi_{\epsilon,s}(x)$, since this process is continuous. Also by the continuity of $\Phi$, the right and left limits $\Phi^\pm_{\epsilon,s}(x)$ in the above formula coincide and therefore the sum on the right-hand side of \eqref{eqn74375} vanishes.
	Following the comment on \cite[p.70]{protter}, the continuous part of the quadratic variation process $ [	\psi^+_{\epsilon}(x), \Phi_{\epsilon} (x )   ]_t^c$ coincides with the quadratic variation of the continuous local martingale  part of the involved processes. The latter can then be evaluated using the usual rules of It\^o calculus as
	\begin{align}&
		\sum_{n\in \Z^3} \int_0^t 
		\Bigl[
		- \nabla (\eta_\epsilon* (\psi g_n) ) + \frac{1}{2}\eta_\epsilon* ( \psi \nabla g_n) 
		\Bigr]\cdot  \Bigl[ - \nabla(\eta_\epsilon*   (\Phi  g_n))  + \frac{1}{2} \eta_\epsilon* ( \Phi \nabla g_n)  \Bigr](x)\,\d s
		\\& \quad +\, \sum_{n\in \Z^3}\int_0^t \eta_\epsilon * ( f_n \nabla_{\sym} u)(x)  : \eta_\epsilon * ( f_n \nabla_{\sym} V)(x)   \,\d s,
	\end{align}
	by invoking the assumption \eqref{Eq_noise_matrix_covariation} on the covariation structure of $B$. Using also that
	\[
	\psi^+_{\epsilon,0}(x)  \,=\,
	\psi_{\epsilon,0}(x)  \,+\,   \int_{\{0\}}       \eta_\epsilon*K_s (x)\,\d\nu_s ,
	\]
	as follows by taking $t= 0$ in \eqref{Eq_Psi_reg}, we process \eqref{eqn74375} to the form
	\begin{align}&\psi^+_{\epsilon,t}(x)\Phi_{\epsilon, t} (x ) - \psi_{\epsilon,0}(x)\Phi_{\epsilon, 0} (x ) \,=\,  \int_0^t 	\psi_{\epsilon,s}^+(x) \,\d \Phi_{\epsilon, s} (x ) \,+\,
		\int_{[0,t]} \Phi_{\epsilon, s} (x )  \,\d 	\psi^+_{\epsilon,s}(x) \\&\qquad +\, \sum_{n\in \Z^3} \int_0^t 
		\Bigl[
		- \nabla (\eta_\epsilon* (\psi g_n) ) + \frac{1}{2}\eta_\epsilon* ( \psi \nabla g_n) 
		\Bigr]\cdot \Bigl[ - \nabla(\eta_\epsilon*   (\Phi  g_n))  + \frac{1}{2} \eta_\epsilon* ( \Phi \nabla g_n)  \Bigr](x)\,\d s
		\\& \qquad +\, \sum_{n\in \Z^3}\int_0^t \eta_\epsilon * ( f_n \nabla_{\sym} u)(x)  : \eta_\epsilon * ( f_n \nabla_{\sym} V)(x)   \,\d s,
	\end{align}
	for $t\in [0,T\wedge \tau_j)$, $\P$-almost surely. 
	Then, by the semimartingale decompositions \eqref{Eq_Psi_reg} and \eqref{Eq_Phi_reg} of the integrators $	\psi^+_{\epsilon}(x)$ and  $\Phi_{\epsilon} (x )$, we obtain that for every $x\in \T^3$, $\P$-a.s.,
	\begin{align}\begin{split}\label{eqn6}&
			\psi^+_{\epsilon,t}(x)\Phi_{\epsilon, t} (x ) - \psi_{\epsilon,0}(x)\Phi_{\epsilon, 0} (x )
			\\& \quad =\, \int_0^t 	\psi_{\epsilon}(x) \eta_\epsilon*\Bigl( \frac{1}{\Phi} \bigl(
			|\nabla \Phi|^2+|\nabla_\sym V|^2
			\bigr)\Bigr)(x) \,\d s \,+\, \int_{[0,t]}    \Phi_{\epsilon}(x)   \eta_{\epsilon}*K_s (x)\,\d\nu_s
			\\&\qquad+\int_0^t  \psi_{\epsilon}(x) \bigl(1+ {F_1}/2 +G_1/2\bigr) \Delta  \Phi_\epsilon(x) \,-\, \psi_{\epsilon}(x) \nabla \cdot (\eta_\epsilon*  (\Phi V) )(x)\,\d s
			\\&\qquad +\, \int_0^t  \Phi_{\epsilon}(x)  \bigl(1+ {F_1}/2 +G_1/2\bigr) \Delta   \psi_\epsilon(x)  \,-\,  \Phi_{\epsilon}(x)  \nabla \cdot (\eta_\epsilon *(\psi u) )(x)  \,\d s
			\\&\qquad -\,\int_0^t \psi_{\epsilon}(x) (F_2/2+G_2/8) \Phi_\epsilon (x) \,\d s  \,-\, \int_0^t  \Phi_{\epsilon}(x)   (F_2/2+G_2/8) \psi_\epsilon (x)  \,\d s
			\\&\qquad +
			\sum_{n\in \Z^3} \int_0^t 
			\Bigl[
			- \nabla (\eta_\epsilon* (\psi g_n) )  + \frac{1}{2} \eta_\epsilon* ( \psi \nabla g_n) 
			\Bigr]\cdot  \Bigl[ - \nabla(\eta_\epsilon*   (\Phi  g_n))  + \frac{1}{2} \eta_\epsilon* ( \Phi \nabla g_n)   \Bigr](x)\,\d s
			\\& \qquad +\, \sum_{n\in \Z^3}\int_0^t  \eta_\epsilon * ( f_n \nabla_{\sym} u)(x)  : \eta_\epsilon * ( f_n \nabla_{\sym} V)(x)   \,\d s 
			\\&\qquad +
			\sum_{n \in \Z^3}\int_0^t  	\psi_{\epsilon}(x)\Bigl[ - \nabla (\eta_\epsilon*  (\Phi  g_n) ) (x) + \frac{1}{2} \eta_\epsilon* ( \Phi \nabla g_n) (x) \Bigr] \cdot \d W_{n,s}
			\\&\qquad + 
			\sum_{n \in \Z^3}\int_0^t \Phi_\epsilon (x)\Bigl[- \nabla (\eta_\epsilon* (\psi g_n) )(x)  + \frac{1}{2} \eta_\epsilon* ( \psi \nabla g_n)  (x)\Bigr] \cdot   \d W_{n,s}
			\\&\qquad -\, \frac{1}{\sqrt{2}}\sum_{n \in \Z^3}\int_0^t 	\psi_{\epsilon}(x) \eta_\epsilon * ( f_n \nabla V)(x) :\d B_{n,s} \,-\, \frac{1}{\sqrt{2}}\sum_{n \in \Z^3}\int_0^t  \Phi_\epsilon (x) \eta_\epsilon*( f_n \nabla u)(x):\d B_{n,s},
	\end{split}	\end{align}
	for all $t\in [0,T\wedge \tau_j)$.

	We claim that the above remains true  when integrating in space. To this end, we average over  equidistant points   $(x_m)_{m=1,\dots M}$, where we take $M$ to be a cubic number, and show that one can take the limit $M\to\infty$ in probability. Indeed, for the temporally discontinuous term, we can observe for example that 
	\begin{align}
		\int_{[0,\tau_j]}\biggl| \frac{1}{M}\sum_{m=1}^M     \Phi_{\epsilon}(x_m)   \eta_{\epsilon}*K_s (x_m) \,-\, \int_{\T^3} \Phi_{\epsilon}   \eta_{\epsilon}*K_s  \d x \biggr| \,\d\nu_s \,\to \,0, \qquad \P\text{-a.s.},
	\end{align}
	by dominated convergence, since 
	\[\sup_{(s,x)\in[0,\tau_j]\times \T^3}
	|\Phi_{\epsilon,s}(x)  | \,\lesssim_\epsilon\, \|\Phi\|_{C([0,\tau_j]; L^2(\T^3))},\quad \quad
	\sup_{(s,x)\in [0,\tau_j]\times \T^3}
	| \eta_{\epsilon}*K_s(x) |\,\le\, 1,
	\]
	and  $x\mapsto \Phi_{\epsilon,s}(x)   \eta_{\epsilon}*K_s (x)$ is  continuous in $x$, for any $s\in [0,\tau_j]$. We expressed the above estimate  purposely in terms of $\|\Phi\|_{C([0,\tau_j]; L^2(\T^3))}$, even though $\Phi\in {C([0,\tau_j]; C^1(\T^3))}$, in order to see how to deal with terms involving $u$ or $\psi$, which are of lower regularity.  For the other deterministic integrals one can thereby proceed analogously, and additional care must only be taken in the correction terms due to the summation over $n\in \Z^3$. In order to show that,  e.g.,
	\begin{equation}\label{eqn373737372949832942}	\sum_{n \in \Z^3} 
		\int_0^{\tau_j}  \biggl|	\frac{1}{M} \sum_{m=1}^M  \eta_\epsilon * ( f_n \nabla_{\sym} u)(x_m)  : \eta_\epsilon * ( f_n \nabla_{\sym} V)(x_m)\,-\, \int_{\T^3}  \eta_\epsilon * ( f_n \nabla_{\sym} u) : \eta_\epsilon * ( f_n \nabla_{\sym} V) \,\d x \biggr|\,  \d s,
	\end{equation}
	vanishes
	$\P$-a.s.\ as $M\to\infty$, one can use that $\P\otimes \d t$-a.e., for all $n\in \Z^3$,
	\[
	\biggl|\frac{1}{M} \sum_{m=1}^M  \eta_\epsilon * ( f_n \nabla_{\sym} u)(x_m)  : \eta_\epsilon * ( f_n \nabla_{\sym} V)(x_m) \,-\, \int_{\T^3}  \eta_\epsilon * ( f_n \nabla_{\sym} u) : \eta_\epsilon * ( f_n \nabla_{\sym} V) \,\d x  \biggr|\, \to \,0.
	\]
	Then it suffices to obtain $\P$-a.s.\ an integrable upper bound on the above, uniformly in $M$,  for which we observe that 
	\begin{align}\begin{split}\label{eqn48485}&	\mathrm{ess sup}_{t \in [0,\tau_j]} \sum_{n\in \Z^3} \sup_{x\in \T^3} \bigl| \eta_\epsilon * ( f_n \nabla_{\sym} u)(x) : \eta_\epsilon * ( f_n \nabla_{\sym} V)(x)\bigr| \\&\quad \lesssim_\epsilon \, 	\mathrm{ess sup}_{t \in [0,\tau_j]} \sum_{n\in \Z^3} \|f_n \nabla_{\sym} u \|_{H^{-2}(\T^3;\R^{3\times 3})} \|  f_n \nabla_{\sym} V\|_{H^{-2}(\T^3;\R^{3\times 3})}
			\\&\quad 	\lesssim \, 	\mathrm{ess sup}_{t \in [0,\tau_j]} \sum_{n\in \Z^3} \|f_n\|_{H^1(\T^3)} \| u \|_{L^2(\T^3;\R^{3})} \|  f_n \|_{H^1(\T^3)} \| V\|_{L^2(\T^3;\R^{ 3})}
			\\&\quad 	\le  \|f \|_{\ell^2(\Z^3; H^1(\T^3))}^2 \|u\|_{L^\infty(0,T;L^2(\T^3;\R^3))}\| V\|_{C([0,\tau_j];L^2(\T^3;\R^3))}
			,\end{split}
	\end{align}
	$\P$-a.s., by the elementary
	\[
	\biggl| \int f_n \nabla_{\sym} u:\vp \,\d x \biggr| \,\lesssim  \, \|u \|_{L^2(\T^3; \R^3)}\|f_n \vp \|_{H^1(\T^3 ;\R^{3\times 3})} \,\lesssim  \,
	\|u \|_{L^2(\T^3; \R^3)}\|f_n \|_{H^1(\T^3)} \|\vp \|_{H^2(\T^3 ;\R^{3\times 3})}.
	\]
	The right-hand side of \eqref{eqn48485} is  $\P$-a.s.\ finite by 
	\[
	\|f \|_{\ell^2(\Z^3; H^1(\T^3))} \,=\, 
	\|f \|_{H^1(\T^3;\ell^2(\Z^3))}\,\le\, \|f \|_{C^1(\T^3;\ell^2(\Z^3))},
	\]
	and the regularity of $f$ observed in  Remark \ref{f_in_Cl2}. For the stochastic integrals we argue based on that, e.g., 
	\[
	\sup_{t\in [0,\tau_j ]} \biggl|
	\sum_{n \in \Z^3}\int_0^t  \biggl[	\frac{1}{M} \sum_{m=1}^M \psi_{\epsilon}(x_m) \nabla (\eta_\epsilon*  (\Phi  g_n) ) (x_m )\,-\, \int_{\T^3} \psi_{\epsilon} \nabla (\eta_\epsilon*  (\Phi  g_n) )  \,\d x \biggr] \cdot \d W_{n,s}
	\biggr| \,\to \,0,
	\]
	in probability, iff 
	\[
	\int_0^{\tau_j} 	\sum_{n \in \Z^3}  \biggl|	\frac{1}{M} \sum_{m=1}^M \psi_{\epsilon}(x_m) \nabla (\eta_\epsilon*  (\Phi  g_n) ) (x_m )\,-\, \int_{\T^3} \psi_{\epsilon} \nabla (\eta_\epsilon*  (\Phi  g_n) )  \,\d x \biggr|^2  \d s\,\to\, 0,
	\]
	in probability. From this point one can proceed as for \eqref{eqn373737372949832942} in order to pass $M\to\infty$ also in these terms.
	
	Having completed the limiting procedure, we find the spatially integrated version of \eqref{eqn6}
	\begin{align}\begin{split}\label{eqn8}&
			\int_{\T^3} \psi^+_{\epsilon,t}\Phi_{\epsilon, t} \, \d x \,-\, \int_{\T^3}  \psi_{\epsilon,0}\Phi_{\epsilon, 0} \, \d x
			\\& \quad =\, \int_0^t \int_{\T^3} 	\psi_{\epsilon}  \eta_\epsilon*\Bigl( \frac{1}{\Phi} \bigl(
			|\nabla \Phi|^2+|\nabla_\sym V|^2
			\bigr)\Bigr)\,\d x \, \d s \,+\, \int_{[0,t]} \int_{\T^3}     \Phi_{\epsilon}(x) \eta_{\epsilon}*K_s   \d x \,\d\nu_s
			\\&\qquad- \, \int_0^t \int_{\T^3} \bigl(2+ {F_1} +G_1\bigr) \nabla  \psi_{\epsilon} \cdot  \nabla \Phi_\epsilon 
			\,+\,  (F_2+G_2/4)  \psi_{\epsilon}\Phi_\epsilon 
			\,\d x\,\d s
			\\&\qquad +\, \int_0^t \int_{\T^3} -\, \psi_{\epsilon} \nabla \cdot (\eta_\epsilon*  (\Phi V) ) \,+ \,  \nabla \Phi_{\epsilon}\cdot (\eta_\epsilon *(\psi u) ) \,\d x \,\d s
			\\&\qquad +
			\sum_{n\in \Z^3} \int_0^t \int_{\T^3}
			\Bigl[
			- \nabla (\eta_\epsilon* (\psi g_n) )  + \frac{1}{2} \eta_\epsilon* ( \psi \nabla g_n) 
			\Bigr] \cdot \Bigl[ - \nabla(\eta_\epsilon*   (\Phi  g_n))  + \frac{1}{2} \eta_\epsilon* ( \Phi \nabla g_n)   \Bigr]\,\d x \,\d s,
			\\& \qquad +\, \sum_{n\in \Z^3}\int_0^t \int_{\T^3} \eta_\epsilon * ( f_n \nabla_{\sym} u)  : \eta_\epsilon * ( f_n \nabla_{\sym} V)  \,\d x\, \d s 
			\\&\qquad +
			\sum_{n \in \Z^3}\int_0^t \int_{\T^3}  	\psi_{\epsilon}\Bigl[ - \nabla (\eta_\epsilon*  (\Phi  g_n) )  + \frac{1}{2} \eta_\epsilon* ( \Phi \nabla g_n) \Bigr] \d x  \cdot \d W_{n,s}
			\\&\qquad + 
			\sum_{n \in \Z^3}\int_0^t \int_{\T^3}   \nabla  \Phi_\epsilon   \eta_\epsilon* (\psi g_n)  + \frac{1}{2} \Phi_\epsilon\eta_\epsilon* ( \psi \nabla g_n) \, \d x \cdot   \d W_{n,s}
			\\&\qquad -\, \frac{1}{\sqrt{2}}\sum_{n \in \Z^3}\int_0^t \int_{\T^3} 	\psi_{\epsilon} \eta_\epsilon * ( f_n \nabla V) \,\d x :\d B_{n,s} \,-\, \frac{1}{\sqrt{2}}\sum_{n \in \Z^3}\int_0^t\int_{\T^3}  \Phi_\epsilon  \eta_\epsilon*( f_n \nabla u)\,\d x:\d B_{n,s}
			\\&\quad =:\, I_2^{(\epsilon)} \,+\,\dots \,+\, I_5^{(\epsilon)} \,+\, K_2^{(\epsilon)}\,+\,  K_3^{(\epsilon)}\,+\,  S_2^{(\epsilon)} \,+\,\dots \, +\,S_5^{(\epsilon)},
	\end{split}	\end{align}
	for $t\in [0,T\wedge \tau_j)$, $\P$-a.s.,
	where we integrated additionally by parts and  grouped  similar terms.
	
	\textit{Step 3 (Taking out the regularization).} We remark that we strategically  grouped the terms appearing in the expansions \eqref{eqn7} and \eqref{eqn8} into deterministic integrals $(I)$, correction terms ($K$), and stochastic integrals ($S$), to discuss specific aspects when taking $\epsilon \to 0$ separately. Before addressing them, we firstly observe that the left-hand sides converges for fixed $t \in [0,T\wedge \tau_j)$, to the desired 
	\begin{equation}\label{eqn19}\int_{\T^3} u_t : V_t -  u_0 : V_0 \,\d x \qquad 
		\text{and}\qquad \int_{\T^3} \psi^+_t \Phi_t -  \psi_0  \Phi_0 \,\d x,
	\end{equation}
	$\P$-a.s., respectively.

	{Ad $I_1^{(\epsilon)},\dots, I_5^{(\epsilon)}$.} We recall the additional regularity of $u$ and $\psi$ discussed in Remark \ref{rem_weak_sol}~\eqref{rem_weak_sol_I2}. Then,
	regarding  $I_1^{(\epsilon)}$, it suffices to use that, $\P$-a.s., 
	\begin{align}\label{eqn17}
		u_\epsilon \,\to\, u ,\qquad  &\text{in ${L^{7/5}} (0,T; W^{1,7/5}(\T^3 ;\R^3)) \cap L^{7/3}([0,T]\times \T^3;\R^3)$}, \\
		\eta_\epsilon * (u\otimes u) \,\to\, u\otimes u   ,\qquad &\text{in $L^{7/6}([0,T]\times \T^3;\R^{3\times 3})$},
		\\ \label{eqn16}\\
		\nabla V_\epsilon \to \nabla V,\, \eta_\epsilon * ((V\cdot \nabla )V ) \,\to \,(V\cdot \nabla) V , \qquad &\text{in $C ([0,\tau_j]\times \T^3;\R^{3\times 3})$ and $C ([0,\tau_j]\times \T^3;\R^3)$}, 
	\end{align}
	respectively,
	to deduce that it converges uniformly on $[0,\tau_j\wedge T)$. Regarding $I^{(\epsilon)}_2$, one can use similarly that 
	\begin{align}\label{eqn10}
		\psi_\epsilon \,\to\, \psi ,\qquad &\text{in ${L^{7/5}} (0,T; W^{1,7/5}(\T^3 )) \cap L^{7/3}([0,T]\times \T^3)$},
	\end{align}
	and \begin{align}
		\eta_\epsilon*\Bigl( \frac{1}{\Phi} \bigl(
		|\nabla \Phi|^2+|\nabla_\sym V|^2
		\bigr)\Bigr) \,\to\, \frac{1}{\Phi} \bigl(
		|\nabla \Phi|^2+|\nabla_\sym V|^2 \bigr),\qquad &\text{in $C ([0,\tau_j]\times \T^3)$},
	\end{align}
	while for $I_3^{(\epsilon)}$, we employ
	\begin{align}\label{eqn11}
		\Phi_\epsilon \to \Phi, \qquad &\text{in $C ([0,\tau_j]\times \T^3)$},
		\\
		\forall t \in [0,T]: \eta_\epsilon*K_t \rightharpoonup  K_t, \qquad & \text{in $\M(\T^3)$},
	\end{align}
	$\P$-a.s., together with dominated convergence. For $I_4^{(\epsilon)}$ and $I^{(\epsilon)}_5$ we combine \eqref{eqn10} and \eqref{eqn11} with 
	\begin{align}\label{eqn18}
		\nabla \Phi_\epsilon\,\to\, \nabla \Phi , \, \eta_\epsilon*  (\Phi V) \to \Phi V, \qquad &\text{in $C([0,\tau_j]\times \T^3;\R^3)$} ,\\
		\eta_\epsilon *(\psi u) \,\to \, \psi u ,\qquad & \text{in $L^{7/6}([0,T]\times \T^3;\R^3)$} ,
	\end{align}
	so that all these integrals are seen to converge uniformly on $[0,T\wedge \tau_j]$.
	
	{Ad $K_1^{(\epsilon)}, K_2^{(\epsilon)}, K_3^{(\epsilon)}$.} For the former of these terms, we employ that 
	\begin{align}\label{eqn14}
		\eta_\epsilon *\nabla  ( \psi f_n ) \,\to \, \nabla  ( \psi f_n ) , \qquad &\text{in $ L^{7/5}([0,T]\times \T^3;\ell^2(\Z^3;\R^3) )$},
		\\ \label{eqn15}
		\eta_\epsilon *\nabla  ( \Phi f_n )  \, \to  \,\nabla  ( \Phi f_n ) ,\qquad& \text{in $ C([0,\tau_j]\times \T^3 ; \ell^2(\Z^3; \R^3)) $},
	\end{align}
	due to Assumption \ref{Ass:noise}. Indeed, by a repeated application of H\"older's inequality we have 
	\begin{align}&
		\|\nabla  ( \psi f_n )\|_{L^{7/5}([0,T]\times \T^3;\ell^2(\Z^3;\R^3) )} \,\le \,	\|f_n\nabla  \psi  \|_{L^{7/5}([0,T]\times \T^3;\ell^2(\Z^3;\R^3) )} \,+\, 	\|\psi \nabla   f_n \|_{L^{7/5}([0,T]\times \T^3;\ell^2(\Z^3;\R^3) )}
		\\&\quad \le \, \sqrt{F_1} \|\nabla \psi \|_{L^{7/5}([0,T]\times \T^3;\R^3)} \,+\, \sqrt{F_2} \| \psi \|_{L^{7/5}([0,T]\times \T^3)}\,<\,\infty  ,
	\end{align}
	$\P$-a.s.,
	and by the properties of vector-valued  convolution, cf.\ \cite[Section 1.2]{Analysis1}, it follows that  
	\begin{align}&
		\| \eta_\epsilon *\nabla  ( \psi f_n )\|_{L^{7/5}(\T^3;\ell^2(\Z^3;\R^3))} \,\le\, 	\| \nabla  ( \psi f_n )\|_{L^{7/5}(\T^3;\ell^2(\Z^3;\R^3))} ,
		\\&
		\eta_\epsilon *\nabla  ( \psi f_n ) \,\to\, \nabla  ( \psi f_n ) ,\qquad \text{in $L^{7/5}( \T^3 ; \ell^2(\Z^3;\R^3) )$},
	\end{align}
	$\P\otimes \d t$-almost everywhere. Therefore, the convergence \eqref{eqn14} follows by the dominated convergence theorem, wheres \eqref{eqn15} can be obtained from $\nabla  ( \Phi f_n ) \in C([0,\tau_j]\times \T^3 ; \ell^2(\Z^3; \R^3)) $. This implies the convergence of $K_1^{(\epsilon)}$ to its unregularized counterpart, uniformly on $[0,T\wedge \tau_j]$ and the convergence of  $K_2^{(\epsilon)}$ and $ K_3^{(\epsilon)}$ is seen in exactly the same manner.
	
	{Ad $S_1^{(\epsilon)},\dots, S_5^{(\epsilon)}$.} Concerning $S_1^{(\epsilon)}$, similarly to the previous step, we make use of the fact that uniform convergence on $[0,T\wedge \tau_j)$ in probability of 
	\begin{align}&
		\frac{1}{\sqrt{2}}	\sum_{n \in \Z^3}  \int_0^t  \int_{\T^3} \eta_\epsilon *( \psi f_n) \nabla V_\epsilon - \psi f_n\nabla V   \,\d x :  \d B_{n,s},
		\\&
		\frac{1}{\sqrt{2}}	\sum_{n \in \Z^3}  \int_0^t \int_{\T^3} u_\epsilon \otimes \nabla  \eta_\epsilon *( \Phi f_n) \,-\,u \otimes \nabla  ( \Phi f_n) \,\d x:  \d B_{n,s},
	\end{align}
	to $0$, 
	i.e., the individual components of $S_1^{(\epsilon)}$, is equivalent to the vanishing in probability of the corresponding quadratic variation processes
	\begin{align}&
		\sum_{n \in \Z^3}  \int_0^{T\wedge \tau_j}\biggl|  \int_{\T^3} \eta_\epsilon *( \psi f_n) \nabla_{\sym} V_\epsilon - \psi f_n\nabla_{\sym} V   \,\d x \biggr|^2 \,  \d s,
		\\&
		\sum_{n \in \Z^3}  \int_0^{T\wedge \tau_j} \biggl|  \int_{\T^3} u_\epsilon \otimes_\sym  \eta_\epsilon*\nabla  (\Phi f_n ) -  u \otimes_\sym  \nabla  (\Phi f_n )   \,\d x \biggr|^2 \,  \d s,
	\end{align}
	where $a \otimes_\sym b = \frac{1}{2} (a\otimes b +b\otimes a)$ denotes the symmetrized tensor product. 
	The desired convergences  follow however from \eqref{eqn17}, \eqref{eqn16} and \eqref{eqn15}, together with 
	\[
	\eta_\epsilon *  ( \psi  f_n )  \, \to  \,   \psi  f_n, \qquad \text{in $ L^{7/3}([0,T]\times \T^3;\ell^2(\Z^3) )$} .\]
	For $S_2^{(\epsilon)},S_3^{(\epsilon)}, S_4^{(\epsilon)}$ one can use precisely the same line of arguments, while for $S_5^{(\epsilon)}$ we use the symmetry of the operation $\eta_\epsilon*$ and integration by parts to deduce that
	\begin{align}
		S_5^{(\epsilon)} \,=\, \frac{1}{\sqrt{2}}\sum_{n \in \Z^3}\int_0^t\int_{\T^3}  \nabla \Phi_\epsilon \otimes  \eta_\epsilon*( f_n  u) +  \Phi_\epsilon   \eta_\epsilon*( \nabla f_n \otimes u)\,\d x:\d B_{n,s}.
	\end{align}
	The limit there can be taken now using \eqref{eqn11}, \eqref{eqn18} and 
	\begin{align}
		\eta_\epsilon *  (   f_n u  )  \to     f_n u ,\qquad &\text{in $ L^{7/3}([0,T]\times \T^3;\ell^2(\Z^3;\R^3) )$} ,\\
		\eta_\epsilon *  (  \nabla  f_n \otimes u  )  \to    \nabla  f_n \otimes u ,\qquad &\text{in $ L^{7/3}([0,T]\times \T^3;\ell^2(\Z^3;\R^{3\times 3}) )$} .
	\end{align}
	
	We summarize that the terms in the It\^o expansions \eqref{eqn7} and \eqref{eqn8} converge in probability, for all times $t\in [0,T\wedge \tau_j)$. Since the resulting left-hand sides \eqref{eqn19} as well as  uniform limits of right-continuous right-hand sides are right-continuous, we conclude that, $\P$-a.s., 
	\begin{align}\begin{split}\label{Eq_Ito_product}&
			\int_{\T^3} u_t : V_t -  u_0 : V_0 \,\d x \,+\,  \int_{\T^3} \psi^+_t \Phi_t -  \psi_0  \Phi_0 \,\d x
			\\&\quad =\,  -2 \int_0^t  
			\int_{\T^3} \nabla_{\sym} V:  \nabla_{\sym}  u \,+\, \nabla  \psi \cdot  \nabla \Phi 	\,\d x \,\d s  \,+
			\int_0^t \int_{\T^3}  \frac{\psi}{\Phi} \bigl(
			|\nabla \Phi|^2+|\nabla_\sym V|^2
			\bigr)\,\d x\, \d s \,+\, \int_{[0,t] \times \T^3}     \Phi   \,\d q
			\\&\qquad +\,  \int_0^t \int_{\T^3}  \,-\, u \cdot      ( (V \cdot \nabla ) V )   \,+ \, \nabla   V:  u\otimes u  \,\d x\,\d s  +\, \int_0^t \int_{\T^3} -\, \psi V\cdot \nabla \Phi \,+ \,  \nabla \Phi \cdot (\psi u )\,\d x \,\d s
			\\&\qquad +
			\sum_{n \in \Z^3}\int_0^t\biggl( \int_{\T^3}  	\psi \Bigl[ - \nabla  (\Phi  g_n)   + \frac{1}{2}  \Phi \nabla g_n \Bigr]
			\,+\,  \psi g_n \nabla  \Phi  \,   +\, \frac{1}{2} \Phi  \psi \nabla g_n  \, \d x \biggr) \cdot   \d W_{n,s}
			\\&\qquad +  \, \frac{1}{\sqrt{2}} \sum_{n \in \Z^3}  \int_0^t \biggl( \int_{\T^3}  \psi f_n \nabla V  \,-\, u \otimes \nabla  ( \Phi f_n) \,-\,  	\psi  f_n \nabla V \, +\, f_n \nabla \Phi \otimes   u \,  +\,  \Phi     \nabla f_n \otimes u \,\d x \biggr):  \d B_{n,s},
		\end{split}
	\end{align}
	for all $t\in [0,T\wedge \tau_j)$. Here, we accounted already for cancellations in the deterministic integrals for which we used that
	\begin{equation}
		\int_{\T^3} \nabla_{\sym }  u : \nabla_{\sym }  V \,\d  x = \frac{1}{2}  \int_{\T^3} \nabla u : \nabla V \,\d  x ,
	\end{equation}
	due to $\nabla \cdot u  = \nabla \cdot V =0$. Inspecting the stochastic integrands shows that they cancel with each other, so that inserting the above in \eqref{eqn5} results in \eqref{Eq_rel_energy_expansion}.
\end{proof}
We turn our attention to the proof of Proposition \ref{thm_rel_en_bound}. We remark that there, in contrast to the preceding lemma, we assume the weak solution to be globally admissible. 
\begin{proof}[Proof of Proposition \ref{thm_rel_en_bound}]
	We start from \eqref{Eq_rel_energy_expansion} and observe that  global admissibility entails
	\begin{equation}\label{eqn20}
		\frac{1}{2}\int_{\T^3} |u_t|^2 + (\psi_t^+)^2- |u_0|^2 - \psi_0^2\,\d x  \,\le\, 0,
	\end{equation}
	for $t\in [0,T)$, $\P$-a.s., by the right-continuity of $\psi^+$. Regarding the second row of \eqref{Eq_rel_energy_expansion}, we use that, $\P$-a.s.,  $q \ge \frac{1}{\psi}(|\nabla \psi|^2 + |\nabla_\sym u|^2) $ with the Convention \ref{conv_inverse} and Young's inequality to  deduce that
	\begin{equation}\label{eqn21}
		2 \int_0^t  
		\int_{\T^3} \nabla_{\sym} V:  \nabla_{\sym}  u + \nabla  \psi \cdot  \nabla \Phi 	\,\d x\, \d s  \,-\,
		\int_0^t \int_{\T^3}  \frac{\psi}{\Phi} \bigl(
		|\nabla \Phi|^2+|\nabla_\sym V|^2
		\bigr)\,\d x\, \d s \,-\, \int_{[0,t] \times \T^3}     \Phi   \,\d q \,\le\, 0,
	\end{equation}
	for $t\in [0,T\wedge \tau )$. 
	Regarding the first term in the third row, we use instead that 
	\begin{align}&  \int_{\T^3}     u \cdot      ( (V \cdot \nabla ) V ) \,-\, \nabla   V:  u\otimes u  \,\d x
		\, =\, \int_{\T^3}       ( (V \cdot \nabla)  V )  \cdot  u\,-\,  ( (u \cdot   \nabla  )  V) \cdot   u\,\d x
		\\&\quad =\,  \int_{\T^3} \bigl( ((V-u) \cdot   \nabla )   V\bigr) \cdot   u  \,\d x \,= \,  \int_{\T^3} \bigl(( (V-u) \cdot   \nabla )   V\bigr) \cdot   (u-V)  \,\d x   ,
	\end{align}
	where the in the last equality we used that 
	\[
	\int_{\T^3} \bigl(( (V-u) \cdot   \nabla )   V\bigr) \cdot   V \,\d x
	\,=\, 
	\frac{1}{2}\int_{\T^3} (u -V) \cdot   \nabla   (|V|^2 )\,\d x \,=\, 0,
	\]
	since $\nabla \cdot u = \nabla \cdot V= 0$. Consequently, we have, $\P\otimes \d t$-a.e.,
	\begin{align}\label{eqn22}
		\int_{\T^3}  u \cdot      (( V \cdot \nabla)  V )\,-\,   \nabla   V:  u\otimes u   \,\d x \,\le\, 2 \|\nabla_\sym V\|_{C(\T^3; \R^{3\times 3})} \EE( u ,\psi^+  | V ,\Phi ).
	\end{align}
	Regarding the second term in the third row of \eqref{Eq_rel_energy_expansion}, we proceed similarly and estimate, $\P\otimes \d t$-a.e.,
	\begin{align}\begin{split}\label{eqn23}&
			\int_{\T^3} \psi V\cdot \nabla \Phi \,-\, \nabla \Phi \cdot (\psi u ) \,\d x \,=\,  \int_{\T^3} ((V-u)\cdot \nabla \Phi) \psi  \,\d x
			\\&\quad =\,  \int_{\T^3} ((V-u)\cdot \nabla \Phi) (\psi -\Phi) \,\d x\,\le\, 
			\|\nabla \Phi \|_{C(\T^3;\R^3)} \EE( u ,\psi^+  | V ,\Phi ).
		\end{split}
	\end{align}
	Inserting \eqref{eqn20}--\eqref{eqn23}\noeqref{eqn21}\noeqref{eqn22} in \eqref{Eq_rel_energy_expansion} yields the desired bound \eqref{Eq_rel_energy_bound} by Gr\"onwall's lemma.
\end{proof}
\section{Global existence of weak martingale solutions}\label{Sec:weak_ex}

In this section, we establish the global in time existence of globally admissible weak martingale solutions to \eqref{Eq_new_var} for initial data from the energy space. More precisely, we will show the following result. 
\begin{theorem}[Global in time weak existence]\label{thm-weak-existence}
		Let Assumption \ref{Ass:noise} hold,
	 $u_0\in L^2(\T^3;\R^3)$ be divergence-free and $\psi_0 \in  L^2(\T^3)$ be non-negative. Then there exists a globally admissible weak martingale solution to  \eqref{Eq_new_var} with initial data $(u_0,\psi_0)$.  
\end{theorem}
For the rest of this section we impose Assumption \ref{Ass:noise} on the noise coefficients, without further mentioning it, and fix initial data $(u_0,\psi_0)$ as in Theorem \ref{thm-weak-existence}.
We recall also the approximation
$ h_\delta(r)$ of $1/r$  defined in Assumption \ref{Ass:nu_N}. In order to prove Theorem \ref{thm-weak-existence}, we start by considering the following approximation of \eqref{Eq_new_var}:
\begin{align}\begin{split}\label{Eq_approx}	\d u_{\delta,\epsilon} \, =\, &  \nabla\cdot(\nabla_{\sym}u_{\delta,\epsilon})\d t \,-\, \epsilon \Delta^2 u_{\delta,\epsilon} \d t-
		\Pi \div   (u_{\delta,\epsilon} \otimes   u_{\delta,\epsilon})  \d t\,+\, \frac{F_1}{4}\Delta u_{\delta,\epsilon} \d t \\& - \frac{1}{\sqrt{2}}\sum_{n \in \Z^3}\Pi \nabla\cdot(\psi_{\delta,\epsilon} f_n   \d B_n)
		,\qquad \nabla\cdot u_{\delta,\epsilon}=0,
		\\
		\d \psi_{\delta,\epsilon}  = &(1+F_1 /2+G_1/2)\Delta \psi_{\delta,\epsilon} \d t -
		\epsilon\Delta^2\psi_{\delta,\epsilon}\d t
		+ h_\delta(\psi_{\delta,\epsilon})\bigl(
		|\nabla \psi_{\delta,\epsilon}|^2 \,+\, |\nabla_\sym u_{\delta,\epsilon}|^2+ \epsilon
		\bigr) \d t
		\\&- \nabla \cdot( u_{\delta,\epsilon}\psi_{\delta,\epsilon} )\d t - (F_2/2+ G_2/8) \psi_{\delta,\epsilon} \d t \\&- \sum_{n\in \Z^3} \nabla \cdot ( \psi_{\delta,\epsilon} g_n \d W_n ) \,
		+\, \frac{1}{2} \sum_{n \in \Z^3} \psi_{\delta,\epsilon} \nabla \cdot ( g_n  \d  W_n) \,-\,  \frac{1}{\sqrt{2}}\sum_{n\in \Z^3}(\nabla  u_{\delta,\epsilon} : f_n\d  B_n),
		\\ (u_{\delta,\epsilon}(0),\psi_{\delta,\epsilon}(0)) =&( u_0,\psi_0).
\end{split}\end{align}
Here  $\delta,\epsilon \in(0,1)$ are  regularization parameters and we proceed in three steps organized as follows:
	In Subsection~\ref{subsec-4-1} we establish the existence of weak martingale solutions to \eqref{Eq_approx} by means of a Galerkin approximation.  Owing to the presence of the higher-order regularization term $-\epsilon \Delta^2$ and $h_\delta$, the energy estimate is  sufficient to yield Sobolev regularity and hence tightness in strong topologies at this stage.
	When passing to the limit $\delta \to 0$ in Subsection~\ref{subsec-4-2}, however, a refined analysis is required for the quadratic gradient term  $h_\delta(\psi_{\delta,\epsilon})\bigl(
	|\nabla \psi_{\delta,\epsilon}|^2 \,+\, |\nabla_\sym u_{\delta,\epsilon}|^2 + \epsilon
	\bigr)$, whose limit gives rise to a dominating measure
	\begin{equation}\label{eqn163} q_\epsilon \,\ge\, 
			\frac{1}{\psi_{\epsilon}} \bigl(
			|\nabla \psi_{\epsilon}|^2 \,+\, |\nabla_\sym u_{\epsilon}|^2+ \epsilon
			\bigr),
	\end{equation}
	due to the  singular limit $h_\delta(r) \to 1/r$. To this end, we close also an estimate on the dissipation of $-\int_{\T^3} \psi \d x$, which refer to as $L^1$-estimate from here on.  At the same time, we also take advantage of \eqref{eqn163} since it implies that $\epsilon/\psi_\epsilon\in L^1([0,T]\times \T^3)$, a.s., so that $\psi_\epsilon$ is  positive a.e.\ in light of Convention \ref{conv_inverse}.
	In Subsection~\ref{subsec-4-3}, we then take also $\epsilon \rightarrow 0$. There, the higher-order regularization term $-\epsilon \Delta^2$ no longer contributes to strong compactness and we instead rely on suitable interpolations of the energy and $L^1$-estimate to obtain uniform gradient regularity of $u_\epsilon$ and $\psi_\epsilon$.
		A detailed analysis of the dominating measure terms is again necessary, and,  
		by passing to the limit, we show that a subsequential limit $(u,\psi)$ of $(u_{\epsilon},\psi_{\epsilon})$ is a  globally admissible weak martingale solution to~\eqref{Eq_new_var}.

\subsection{Construction of weak solutions to the ($\delta,\epsilon)$-regularization. }\label{subsec-4-1}
In this subsection, we show the existence of weak martingale solutions to  \eqref{Eq_approx}, defined as follows. 
\begin{definition}[Weak martingale solution to \eqref{Eq_approx}]\label{defi_solution_approx}
	A  \emph{weak martingale solution 
		to \eqref{Eq_approx}} 
	consists of a probability space $(\Omega,\A,\P)$, a filtration $\F$ satisfying the usual conditions, independent $3$-dimensional standard $\F$-Brownian motions $(W_n)_{n\in \Z^3}$ and $3\times 3$-matrix valued $\F$-Brownian motions $(B_n)_{n\in \Z^3}$ with covariance structure \eqref{Eq_noise_matrix_covariation}, a weakly continuous and adapted  $L^2(\T^3;\R^3)$-valued process
	$u_{\delta,\epsilon}$, and $L^2(\T^3)$-valued  $\psi_{\delta,\epsilon}$, such that the following holds:
	\begin{enumerate}[(i)]
		\item  We have the additional regularity  $ u_{\delta,\epsilon}\in L^{2}(0,T;H^2( \T^3 ;\R^{3}))$ and $\psi_{\delta,\epsilon} \in  L^{2}(0,T;H^2( \T^3))$, $\P$-almost surely.
		\item  For every divergence-free $\vp\in C^{\reg}(\T^3;\R^3)$ it holds
		\begin{align}\begin{split}
				& \int_{\T^3}  u_{\delta,\epsilon}(t)\cdot \vp\,\d x \,-\, 
				\int_{\T^3} u_0 \cdot \vp\,\d x\,+\,\epsilon\int_0^t\int_{\T^3} u_{\delta,\epsilon}\Delta^2 \vp\,\d x\,\d s
				\\&\quad=\,  \int_0^t \int_{\T^3} -  \Bigl(\frac{1}{2}+ \frac{F_1}{4}\Bigr) \nabla u_{\delta,\epsilon}: \nabla  \vp \,+\, u_{\delta,\epsilon}\otimes u_{\delta,\epsilon}:\nabla \vp \,\d x\, \d s \,+\, \frac{1}{\sqrt{2}}\sum_{n \in \Z^3}\int_0^t \int_{\T^3}\psi_{\delta,\epsilon} f_n \nabla \vp\,\d x :\d B_{n,s},
			\end{split}
		\end{align}
		and $\nabla \cdot u_t = 0$ for all $t\in [0,T]$, $\P$-almost surely.
		\item  
		For every $\zeta\in C_c^{\reg}([0,T) \times \T^3)$ it holds
		\begin{align}\begin{split}\label{Eq_psi_weak_new}
				&-\,
				\int_0^T \int_{\T^3} \psi_{\delta,\epsilon}  \partial_t \zeta  \, \d x\, \d t \,-\, \int_{\T^3}
				\psi_0 \zeta
				\,\d x+\,\epsilon\int_0^T\int_{\T^3}\psi_{\delta,\epsilon}\Delta^2\zeta\,\d x\,\d s 
				\\&\quad=\,  \int_0^T \int_{\T^3} - \bigl(1+ {F_1}/2 +G_1/2\bigr) \nabla \psi_{\delta,\epsilon}\cdot \nabla  \zeta \,+\,\psi_{\delta,\epsilon} u_{\delta,\epsilon}\cdot \nabla \zeta   \,-\, (F_2/2+G_2/8) \psi_{\delta,\epsilon} \zeta \,\d x\,\d t \\
				&\qquad+\, \int_0^T\int_{\T^3} \zeta\, h_\delta(\psi_{\delta,\epsilon})\bigl(
				|\nabla \psi_{\delta,\epsilon}|^2 \,+\, |\nabla_\sym u_{\delta,\epsilon}|^2 + \epsilon
				\bigr)\,\d x\,\d t
				\\&\qquad + \sum_{i=1}^3
				\sum_{n \in \Z^3}\int_0^T \int_{\T^3} \psi_{\delta,\epsilon} g_n \partial_i \zeta  + \frac{1}{2} \psi_{\delta,\epsilon}  \zeta \partial_ig_n \,\d x \, \d W_{n,t}^i
				\, +\, \frac{1}{\sqrt{2}}\sum_{i,j =1}^3\sum_{n \in \Z^3}\int_0^T \int_{\T^3}u_{\delta,\epsilon,j}\partial_i (\zeta f_n) \,\d x \,\d B_{n,t}^{ji},
			\end{split}
		\end{align}  $\P$-almost surely.	
	\end{enumerate}
\end{definition}
\begin{remark}
	We remark that the regularity of $\psi_{\delta,\epsilon}$ in Definition \ref{defi_solution_approx} suffices for a for a $t$-wise interpretation of  \eqref{Eq_psi_weak_new}, i.e., one could alternatively test only in space. Consequently, it is not necessary to test against smooth functions with compact support in $[0,T)$.  Nevertheless we adopt the above formulation, because it is convenient for the subsequent limiting procedures. In any case, a formulation without testing in the time variable can be recovered by following the steps of Lemma \ref{Lemma_cadlag}, which yields that 
		\begin{align}\label{modification-eq} \begin{split}&
			\int_{\T^3} \psi_{\delta,\epsilon} (t) \eta   \, \d x \,-\, \int_{\T^3}
			\psi_0 \eta 
			\,\d x \, +\,\epsilon \int_0^t \int_{\T^3} \psi_{\delta,\epsilon}
			\Delta^2 \eta  \,
			\d x \,\d s
			\\&\quad=\, \int_0^t \int_{\T^3} -  \bigl(1+ {F_1}/2 +G_1/2\bigr) \nabla \psi_{\delta,\epsilon}  \cdot \nabla  \eta \,+\,\psi u_{\delta,\epsilon} \cdot \nabla \eta   \,-\, (F_2/2+G_2/8) \psi_{\delta,\epsilon}  \eta \,\d x\,\d s \\&\qquad \,+ \, \int_0^t \int_{\T^3} \eta  h_\delta(\psi_{\delta,\epsilon})\bigl(
			|\nabla \psi_{\delta,\epsilon}|^2 \,+\, |\nabla_\sym u_{\delta,\epsilon}|^2 + \epsilon
			\bigr)\,\d x \,\d s
			\\&\qquad + \sum_{i=1}^3
			\sum_{n \in \Z^3}\int_0^t \int_{\T^3}\psi_{\delta,\epsilon}  g_n \partial_i \eta  + \frac{1}{2} \psi_{\delta,\epsilon}   \eta \partial_ig_n  \,\d x\,  \d W_{n,s}^i
			\, +\, \frac{1}{\sqrt{2}}\sum_{i,j =1}^3\sum_{n \in \Z^3}\int_0^t \int_{\T^3}u_{{\delta,\epsilon}, j}\partial_i (\eta f_n) \,\d x \,\d B_{n,s}^{ji},
		\end{split}
	\end{align}
for $t\in [0,T]$, $\P$-a.s.,
	for any $\eta \in C^{\reg}(\T^3)$.
\end{remark}

\begin{proposition}\label{prop_1}
		For every $\delta,\epsilon\in(0,1)$, there exists a weak martingale solutions to  \eqref{Eq_approx} with initial data $(u_0,\psi_0)$, which satisfies
	\begin{equation}\label{total-energy-es}
		\frac{1}{2}\int_{\T^3}|u_{\delta,\epsilon}(t)|^2 +|\psi_{\delta,\epsilon}(t)|^2 \,\d x \,+\, \epsilon \int_0^t \int_{\T^3} 
		|\Delta u_{\delta,\epsilon}  |^2 +
		|\Delta\psi_{\delta,\epsilon} |^2\,\d x \,\d s
		\,\leq\, \frac{1}{2}\int_{\T^3}|u_0|^2 +|\psi_0|^2 \,\d x\,+\,\epsilon t,
	\end{equation}
	for all $ t\in [0,T]$, $\P$-almost surely.
\end{proposition}
\begin{proof}
	For each $m \in \mathbb{N}$ we recall the Fourier projection $P_m$ and the projection $\Pi_m$ on the span of eigenvectors of the Stokes operator introduced in \eqref{Fourier_proj} and \eqref{Stokes_proj}, respectively. Then, we consider the following Galerkin  scheme for \eqref{Eq_approx}:
	\begin{align}\begin{split}\label{Eq_approx_m}
			\d u_{\delta,\epsilon,m} \, =\, &  \nabla\cdot(\nabla_{\sym}u_{\delta,\epsilon,m})\,\d t 
			\,-\, \epsilon \Delta^2 u_{\delta,\epsilon,m}\,\d t
			\,- \,\Pi_m \div (u_{\delta,\epsilon,m} \otimes u_{\delta,\epsilon,m})\,\d t  \\
			&+\, \frac{F_1}{4}\Delta u_{\delta,\epsilon,m}\,\d t
			\,-\, \frac{1}{\sqrt{2}}\sum_{n \in \Z^3} \Pi_m \nabla\cdot(\psi_{\delta,\epsilon,m} f_n\, \d B_n),
			\qquad \nabla\cdot u_{\delta,\epsilon,m}=0,
			\\
			\d \psi_{\delta,\epsilon,m}  =\,& (1+F_1 /2+G_1/2)\Delta \psi_{\delta,\epsilon,m}\,\d t
			- \epsilon\Delta^2\psi_{\delta,\epsilon,m}  \,\d t \\
			&+ P_m\Big[h_\delta(\psi_{\delta,\epsilon,m})\bigl(
			|\nabla \psi_{\delta,\epsilon,m}|^2 \,+\, |\nabla_\sym u_{\delta,\epsilon,m}|^2 + \epsilon
			\bigr)\Big]\,\d t \\
			&- \,P_m\nabla \cdot( u_{\delta,\epsilon,m}\psi_{\delta,\epsilon,m} )\,\d t 
			\,- \,(F_2/2+ G_2/8)\psi_{\delta,\epsilon,m}\,\d t \\
			&- \,\sum_{n\in \Z^3} P_m\nabla \cdot ( \psi_{\delta,\epsilon,m} g_n\, \d W_n )
			\,+ \,\frac{1}{2} \sum_{n \in \Z^3} P_m[\psi_{\delta,\epsilon,m} \nabla \cdot ( g_n\, \d W_n)] \\
			&- \,\frac{1}{\sqrt{2}}\sum_{n\in \Z^3} P_m(\nabla u_{\delta,\epsilon} : f_n\, \d B_n),
			\\
			(u_{\delta,\epsilon,m}(0),\psi_{\delta,\epsilon,m}(0))\, =\,& (P_m u_0, P_m \psi_0).
	\end{split}\end{align}
	
	Applying It\^o's formula to 
	$\frac{1}{2}\int_{\T^3}|u_{\delta,\epsilon,m}(t)|^2\,\d x$
	and 
	$\frac{1}{2}\int_{\T^3}|\psi_{\delta,\epsilon,m}(t)|^2\,\d x$,
	and  arguing as for the energy conservation of strong solutions to \eqref{Eq_new_var}, stated in Lemma \ref{prop_cons_of_energy}, 
	one obtains the following total energy estimate: $\mathbb{P}$-almost surely, for every $t \in [0,T]$, we have
	\begin{align}\begin{split}\label{total-energy-es-m}&
		\frac{1}{2}\int_{\T^3}|u_{\delta,\epsilon,m}(t)|^2 +|\psi_{\delta,\epsilon,m}(t)|^2 \,\d x
		+ \epsilon \int_0^t \int_{\T^3} 
		|\Delta u_{\delta,\epsilon,m}|^2
		+ |\Delta\psi_{\delta,\epsilon,m}|^2\,\d x\,\d s
		\\&\quad \leq\, \frac{1}{2}\int_{\T^3}|P_m u_0|^2 +|P_m \psi_0|^2 \,\d x\,+\,\epsilon t,\end{split}
	\end{align}
and in particular the solutions to the locally Lipschitz SDEs \eqref{Eq_approx_m} exists globally in time. 
 Indeed, when computing the It\^o correction terms from the stochastic integrals, the presence of the Galerkin projections leads only to a diminished $L_x^2$-norm. In the terms where the right-hand sides of \eqref{Eq_approx_m} are tested against $u_{\delta,\epsilon,m}$ and $\psi_{\delta,\epsilon,m}$, the projections vanish entirely since they are orthogonal in $L_x^2$. Thereby the only necessary adaption is regarding the emerging term 
	\begin{equation}\label{eqn181}
\int_0^t  \int_{\mathbb{T}^3} \psi_{\delta,\epsilon,m} h_{\delta}(\psi_{\delta,\epsilon,m}) \big(|\nabla \psi_{\delta,\epsilon,m}|^2 + |\nabla_{\mathrm{sym}} u_{\delta,\epsilon,m}|^2 + \epsilon \big) \,\mathrm{d}x\,\mathrm{d}s,
\end{equation}
for which we use the pointwise bound that
$$
r h_{\delta}(r) \leq 1 \qquad \text{for all } r \in \mathbb{R},
$$
resulting from Assumption \ref{Ass:nu_N}. This leads to the upper bound 
$$\int_0^t \int_{\mathbb{T}^3} \big(|\nabla  \psi_{\delta,\epsilon,m}|^2 + |\nabla_{\mathrm{sym}} u_{\delta,\epsilon,m}|^2 + \epsilon \big) \,\mathrm{d}x\,\mathrm{d}s
$$
on \eqref{eqn181}. The extra $\epsilon$ on the right-hand side is responsible for the additional term $\epsilon t$ in \eqref{total-energy-es-m} compared to \eqref{eqn_strong_energy_cons}, while the extra dissipation on the left-hand side of \eqref{total-energy-es-m} results from the regularizations $-\epsilon\Delta^2 u_{\delta,\epsilon,m}$ and  $-\epsilon\Delta^2 \psi_{\delta,\epsilon,m}$.

	From here on, we proceed analogously to the more delicate limit $\delta\to 0$ detailed in Subsection \ref{subsec-4-2}, to which we refer for more details: Following the line of arguments in Lemma \ref{lem-timeregularity} below, we see that for every $\alpha \in (0,1/2)$ and $l > 5/2$, 
	there exists a constant $C =C(\alpha, u_0,\psi_0, F_1,F_2, G_1,G_2,T,\epsilon,\delta) > 0$ such that
	$$
	\mathbb{E}\|u_{\delta,\epsilon,m}\|_{C^{\alpha}([0,T];H^{-l}(\T^3;\R^3))}
	\,+\, \mathbb{E}\|\psi_{\delta,\epsilon,m}\|_{C^{\alpha}([0,T];H^{-l}(\T^3))}\,
	\leq \,C.
	$$
	Combining this with \eqref{total-energy-es-m} and the Aubin--Lions--Simon compactness criterion yields  uniform in $m$ tightness of 
	$(u_{\delta,\epsilon,m},\psi_{\delta,\epsilon,m})$ in the  strong topology of
	$L^2([0,T];H^1(\T^3;\R^3)) \times L^2([0,T];H^1(\T^3))$, which allows for an application of the Skorokhod--Jakubowski theorem to extract an equidistributed convergent subsequence.
	Identifying its limit as a weak martingale solution to \eqref{Eq_approx} is analogous and in fact easier than the limiting procedure in Subsection \ref{subsec-4-2}, so we only provide some comments on the passage to the limit for the nonlinearity 
	$$
	P_m\Big[h_\delta(\psi_{\delta,\epsilon,m})\bigl(
	|\nabla \psi_{\delta,\epsilon,m}|^2 \,+\, |\nabla_\sym u_{\delta,\epsilon,m}|^2 + \epsilon
	\bigr)\Big],
	$$
which does not collapse into a dominating measure at this stage.
	We denote the new probability space and  converging subsequence again by $(\Omega,\A,\P)$ and  $(u_{\delta,\epsilon,m},\psi_{\delta,\epsilon,m})$, such that $\mathbb{P}$-almost surely,
	$$
	(u_{\delta,\epsilon,m},\psi_{\delta,\epsilon,m})
	\to (u_{\delta,\epsilon},\psi_{\delta,\epsilon})
	\quad \text{in }
	L^2([0,T];H^1(\T^3;\R^3)) \times L^2([0,T];H^1(\T^3)).
	$$
	We claim that $\mathbb{P}$-almost surely, for every $\zeta \in C_c^\infty([0,T)\times \T^3)$,
	\begin{align*}
		&\int_0^T\int_{\T^3}h_\delta(\psi_{\delta,\epsilon,m})\bigl(
		|\nabla \psi_{\delta,\epsilon,m}|^2 + |\nabla_\sym u_{\delta,\epsilon,m}|^2 + \epsilon
		\bigr)P_m\zeta\,\d x\,\d t \\
		&\qquad \to
		\int_0^T\int_{\T^3}h_\delta(\psi_{\delta,\epsilon})\bigl(
		|\nabla \psi_{\delta,\epsilon}|^2 + |\nabla_\sym u_{\delta,\epsilon}|^2 + \epsilon
		\bigr)\zeta\,\d x\,\d t.
	\end{align*}
	To see this, we decompose the difference into
	\begin{align*}
		&\int_0^T\int_{\T^3}h_\delta(\psi_{\delta,\epsilon,m})\bigl(
		|\nabla \psi_{\delta,\epsilon,m}|^2 + |\nabla_\sym u_{\delta,\epsilon,m}|^2 + \epsilon
		\bigr)P_m\zeta\,\d x\,\d t \\& - \int_0^T\int_{\T^3}h_\delta(\psi_{\delta,\epsilon})\bigl(
		|\nabla \psi_{\delta,\epsilon}|^2 + |\nabla_\sym u_{\delta,\epsilon}|^2 + \epsilon
		\bigr)\zeta\,\d x\,\d t\\
		=\;
		&\int_0^T\int_{\T^3}h_\delta(\psi_{\delta,\epsilon})\bigl(
		|\nabla \psi_{\delta,\epsilon,m}|^2 + |\nabla_\sym u_{\delta,\epsilon,m}|^2 - |\nabla \psi_{\delta,\epsilon}|^2 - |\nabla_\sym u_{\delta,\epsilon}|^2
		\bigr)P_m\zeta\,\d x\,\d t\\&+\int_0^T\int_{\T^3}(h_\delta(\psi_{\delta,\epsilon,m})-h_\delta(\psi_{\delta,\epsilon}))\bigl(
		|\nabla \psi_{\delta,\epsilon,m}|^2 + |\nabla_\sym u_{\delta,\epsilon,m}|^2 + \epsilon
		\bigr)P_m\zeta\,\d x\,\d t\\
		&+\int_0^T\int_{\T^3}h_\delta(\psi_{\delta,\epsilon})\bigl(
		|\nabla \psi_{\delta,\epsilon}|^2 + |\nabla_\sym u_{\delta,\epsilon}|^2 + \epsilon
		\bigr)(P_m\zeta-\zeta)\,\d x\,\d t\,=\, I_1 + I_2 + I_3,
	\end{align*}
	where $I_1$, $I_2$, and $I_3$ are defined accordingly.
	For $I_1$, the uniform boundedness of $h_\delta(\psi_{\delta,\epsilon,m})$ and $P_m\zeta$, together with the 
	$L^2([0,T];H^1(\T^3;\R^3)) \times L^2([0,T];H^1(\T^3))$ strong convergence, implies that $I_1 \to 0$.
	Moreover, the uniform boundedness of $P_m\zeta$ and the total energy estimate~\eqref{total-energy-es-m} yield $I_2 \to 0$ and $I_3 \to 0$.
\end{proof}

\subsection{Construction of weak solutions to the $\epsilon$-regularization. } \label{subsec-4-2}In
this subsection, we  show the existence of weak martingale solutions to 
\begin{align}\begin{split}\label{ep_approx}	\d u_{\epsilon} \, =\, &  \nabla\cdot(\nabla_{\sym}u_{\epsilon})\d t \,-\, \epsilon \Delta^2 u_{\epsilon} \d t-
		\Pi \div   (u_{\epsilon} \otimes   u_{\epsilon})  \d t\,+\, \frac{F_1}{4}\Delta u_{\epsilon} \d t \\& - \frac{1}{\sqrt{2}}\sum_{n \in \Z^3}\Pi \nabla\cdot(\psi_{\epsilon} f_n   \d B_n)
		, 	\qquad \nabla\cdot u_{\epsilon}=0,
		\\
		\d \psi_{\epsilon}  = &(1+F_1 /2+G_1/2)\Delta \psi_{\epsilon} \d t -
		\epsilon\Delta^2\psi_{\epsilon}\d t
		\, +\,  \d q_\epsilon 
		- \nabla \cdot( u_{\epsilon}\psi_{\epsilon} )\d t - (F_2/2+ G_2/8) \psi_{\epsilon} \d t \\&- \sum_{n\in \Z^3} \nabla \cdot ( \psi_{\epsilon} g_n \d W_n ) \,
		+\, \frac{1}{2} \sum_{n \in \Z^3} \psi_{\epsilon} \nabla \cdot ( g_n  \d  W_n) \,-\,  \frac{1}{\sqrt{2}}\sum_{n\in \Z^3}(\nabla  u_{\epsilon} : f_n\d  B_n),
		\\ (u_{\epsilon}(0),\psi_{\epsilon}(0)) =&( u_0,\psi_0),
\end{split}\end{align}
where 
\begin{equation}\label{eqn354865865}
	q_\epsilon \,\ge\, \frac{1}{\psi_{\epsilon}}\bigl(
	|\nabla \psi_{\epsilon}|^2 \,+\, |\nabla_\sym u_{\epsilon}|^2+ \epsilon
	\bigr) ,
\end{equation}
defined as follows.
\begin{definition}[Weak martingale solution to \eqref{ep_approx}--\eqref{eqn354865865}]\label{defi_solution_ep}
	A  \emph{weak martingale solution 
		to \eqref{ep_approx}--\eqref{eqn354865865}} 
	consists of a probability space $(\Omega,\A,\P)$, a filtration $\F$ satisfying the usual conditions, independent $3$-dimensional standard $\F$-Brownian motions $(W_n)_{n\in \Z^3}$ and $3\times 3$-matrix valued $\F$-Brownian motions $(B_n)_{n\in \Z^3}$ with covariance structure \eqref{Eq_noise_matrix_covariation}, a weakly continuous and adapted  $L^2(\T^3 ;\R^3)$-valued process
	$u_{\epsilon}$, as well as   progressively measurable   $\psi_{\epsilon} \in L^\infty(0,T;L^2( \T^3))$ and  $q_{\epsilon}\in \M_+([0,T]\times \T^3)$, such that the following holds:
	\begin{enumerate}[(i)]
		\item \label{Item_defi_sol_1_ep} We have the additional regularity {$u_{\epsilon}\in L^{2}(0,T;H^2(\T^3 ;\R^{3}))$ and $\psi_{\epsilon} \in  L^{2}(0,T;H^2(\T^3))$}, $\P$-almost surely.
		\item  \label{Item_defi_sol_2_ep} It holds $\psi_{\epsilon}>0$, $\P\otimes \d t \otimes \d x$-a.e., and  $q_{\epsilon} \ge \frac{1}{\psi_{\epsilon}}(|\nabla \psi_{\epsilon}|^2 + |\nabla_\sym u_{\epsilon}|^2+\epsilon)$,  as measures  on $[0,T]\times \T^3$, $\P$-almost surely. 
		\item  \label{Item_defi_sol_3_ep} For every divergence-free $\vp\in C^{\infty}(\T^3;\R^3)$ it holds
		\begin{align}\begin{split}
				\label{eqn9_ep}& \int_{\T^3}  u_{\epsilon,t}\cdot \vp\,\d x \,-\, 
				\int_{\T^3} u_0 \cdot \vp\,\d x\,+\,\epsilon\int_0^t\int_{\T^3} u_{\epsilon}\Delta^2\vp\,\d x\,\d s
				\\&\quad=\,  \int_0^t \int_{\T^3} -  \Bigl(\frac{1}{2}+ \frac{F_1}{4}\Bigr) \nabla u_{\epsilon}: \nabla  \vp \,+\, u_{\epsilon}\otimes u_{\epsilon}:\nabla \vp \,\d x\, \d s \,+\, \frac{1}{\sqrt{2}}\sum_{n \in \Z^3}\int_0^t \int_{\T^3}\psi_{\epsilon} f_n \nabla \vp\,\d x :\d B_{n,s},
			\end{split}
		\end{align}
		and $\nabla \cdot u_{\epsilon,t} = 0$ for all $t\in [0,T]$, $\P$-almost surely.
		\item  \label{Item_defi_sol_4_ep}
		For every $\zeta\in C_c^{\infty}([0,T) \times \T^3)$ it holds
		\begin{align}\begin{split}\label{Eq_psi_weak_ep}
				&-\,
				\int_0^T \int_{\T^3} \psi_{\epsilon}  \partial_t \zeta  \, \d x\, \d t \,-\, \int_{\T^3}
				\psi_0 \zeta
				\,\d x\,+\,\epsilon\int_0^T\int_{\T^3} \psi_{\epsilon}\Delta^2\zeta\,\d x\,\d s 
				\\&\quad=\,  \int_0^T \int_{\T^3} - \bigl(1+ {F_1}/2 +G_1/2\bigr) \nabla \psi_{\epsilon}\cdot \nabla  \zeta \,+\,\psi_{\epsilon} u_{\epsilon}\cdot \nabla \zeta   \,-\, (F_2/2+G_2/8) \psi_{\epsilon} \zeta \,\d x\,\d t \,+\, \int_{[0,T]\times\T^3} \zeta \,\d q
				\\&\qquad + \sum_{i=1}^3
				\sum_{n \in \Z^3}\int_0^T \int_{\T^3} \psi_{\epsilon} g_n \partial_i \zeta  + \frac{1}{2} \psi_{\epsilon}  \zeta \partial_ig_n \,\d x \, \d W_{n,t}^i
				\, +\, \frac{1}{\sqrt{2}}\sum_{i,j =1}^3\sum_{n \in \Z^3}\int_0^T \int_{\T^3}u_{\epsilon,j}\partial_i (\zeta f_n) \,\d x \,\d B_{n,t}^{ji},
			\end{split}
		\end{align}  $\P$-almost surely.
	\end{enumerate}
\end{definition}

The starting point of our construction are the solutions $(u_{\delta,\epsilon} ,u_{\delta,\epsilon})$ to \eqref{Eq_approx} provided by Proposition \ref{prop_1}. For notational convenience, we assume all of them to be defined on the same stochastic basis. The latter does not effect the validity of our proof, since we eventually pass to an equidistributed subsequence, anyways.
\subsubsection{Uniform estimates and tightness}
Here, we collect uniform estimates on $(u_{\delta,\epsilon} ,u_{\delta,\epsilon})$ and recall the energy estimate \eqref{total-energy-es}, which holds uniformly in $(\delta,\epsilon)\in (0,1)$. We also establish the announced $L^1$-estimate.
	\begin{lemma}\label{lem-L1-es}
		Let 
		 $(u_{\delta,\epsilon} ,\psi_{\delta,\epsilon})$ be the weak martingale solutions to  \eqref{Eq_approx} with initial data $(u_0,\psi_0)$ constructed in Proposition \ref{prop_1}. Then
		we have 
			\begin{align}\label{eqn151}&
			\mathbb{E}\int_0^T\int_{\T^3}h_\delta(\psi_{\delta,\epsilon})\bigl(
			|\nabla \psi_{\delta,\epsilon}|^2 + |\nabla_\sym u_{\delta,\epsilon}|^2 + \epsilon
			\bigr)\,\d x\,\d t \\&\quad \leq  \, 	\biggl(1 \,+\, \frac{4F_2 +G_2 }{8}T \biggr)  \biggl(
			\frac{1}{2}\int_{\T^3}|u_0|^2 +|\psi_0|^2 \,\d x\,+\,2 \epsilon T
			\biggr)^{1/2},
		\end{align}
	for all $\delta,\epsilon\in (0,1)$.		
	\end{lemma}
	\begin{proof}
		Taking $\eta =\mathbf{1}_{\T^3}$ in \eqref{modification-eq} all terms in divergence form vanish and after taking  the expectation also the martingales, so that 			
		\begin{align*}&\mathbb{E}\|\psi_0\|_{L^1(\T^3)} \,+\, 
			\mathbb{E}\int_0^T\int_{\T^3}h_\delta(\psi_{\delta,\epsilon})\bigl(
			|\nabla \psi_{\delta,\epsilon}|^2 + |\nabla_\sym u_{\delta,\epsilon}|^2 + \epsilon
			\bigr)\,\d x\,\d t \\& \quad  
			\leq\,
			\mathbb{E}\|\psi_{\delta,\epsilon}(T)\|_{L^1(\T^3)}
			\,+\,(F_2/2 +G_2/8) \mathbb{E}\int_0^T\|\psi_{\delta,\epsilon}(t)\|_{L^1(\T^3)}\,\d t .
		\end{align*}
	Using the total energy estimate~\eqref{total-energy-es}, we can further estimate this by
	\begin{align}
		\biggl(1 \,+\, \frac{4F_2 +G_2 }{8}T \biggr)  \biggl(
		\frac{1}{2}\int_{\T^3}|u_0|^2 +|\psi_0|^2 \,\d x\,+\,2\epsilon T
		\biggr)^{1/2},
	\end{align}
confirming \eqref{eqn151}.
	\end{proof}
		
	To obtain  tightness in the strong topology of $L^2([0,T]; L^2(\T^3))$, it  suffices to establish suitable time-regularity estimates. More precisely, we seek uniform bounds in H\"older spaces $C^{\alpha}([0,T];H^{-l}(\T^3;\R^3))$ and fractional Sobolev spaces $W^{\alpha,2}([0,T];H^{-l}(\T^3))$, which, combined with the compact embedding $H^2_x \subset L^2_x$,  implies uniform tightness in $\delta$ via a variant of the Aubin--Lion--Simon lemma, by \eqref{total-energy-es}. We stress that, while the latter argument needs to be adapted for subsequent passage $\epsilon\to 0$, at least the below estimates are uniform $\epsilon$ as well.
	
	\begin{lemma}\label{lem-timeregularity}
		 Let 
		 $(u_{\delta,\epsilon} ,\psi_{\delta,\epsilon})$ be the weak martingale solutions to  \eqref{Eq_approx} with initial data $(u_0,\psi_0)$ constructed in Proposition \ref{prop_1}. Then for every $\alpha \in (0,1/2)$ and $l > 5/2$, there exists a constant \[C = C(\alpha,l, u_0,\psi_0, F_1,F_2, G_1,G_2,T) <\infty, \]
		 such that
		\begin{align}\label{eqn872374}
			\mathbb{E}\|u_{\delta,\epsilon}\|_{C^{\alpha}([0,T];H^{-l}(\mathbb{T}^3;\R^3))}+\mathbb{E}\|\psi_{\delta,\epsilon}\|_{W^{\alpha,2}([0,T];H^{-l}(\mathbb{T}^3))} \,\leq\, C,
		\end{align}
		for all $\delta,\epsilon\in (0,1)$.
	\end{lemma}
	
	\begin{proof}
		We first focus on the velocity component. Applying the norm $\|\cdot\|_{C^{\alpha}([0,T];H^{-l}(\mathbb{T}^3;\R^3))}$ to the velocity equation in \eqref{Eq_approx}, and using the triangle inequality, we obtain:
		\begin{align*}
			\|u_{\delta,\epsilon}\|_{C^{\alpha}([0,T];H^{-l}(\mathbb{T}^3;\R^3))}
			&\leq\, C(u_0,T)\\
			& + \left\|\int^{\cdot}_0\nabla\cdot(\nabla_{\mathrm{sym}} u)\,\d s\right\|_{C^{\alpha}([0,T];H^{-l}(\mathbb{T}^3;\R^3))} 
			+ \epsilon \left\|\int^{\cdot}_0 \Delta^2 u_{\delta,\epsilon}\,\d s\right\|_{C^{\alpha}([0,T];H^{-l}(\mathbb{T}^3;\R^3))} \\
			&+ \left\|\int^{\cdot}_0\Pi \div(u_{\delta,\epsilon} \otimes u_{\delta,\epsilon})\,\d s\right\|_{C^{\alpha}([0,T];H^{-l}(\mathbb{T}^3;\R^3))} 
			+ \frac{F_1}{4} \left\|\int^{\cdot}_0\Delta u_{\delta,\epsilon}\,\d s\right\|_{C^{\alpha}([0,T];H^{-l}(\mathbb{T}^3;\R^3))} \\
			&+ \left\|\frac{1}{\sqrt{2}} \sum_{n \in \mathbb{Z}^3} \int^{\cdot}_0\Pi \nabla \cdot (\psi_{\delta,\epsilon} f_n \,\d B_n) \right\|_{C^{\alpha}([0,T];H^{-l}(\mathbb{T}^3;\R^3))}.
		\end{align*}
		To estimate the convective term, we use embedding $L^1(\T^3;\R^3) \hookrightarrow H^{1-l}(\T^3;\R^3)$ to deduce that
		$$
		W^{1,\infty}([0,T];L^1(\T^3;\R^3)) \hookrightarrow C^{\alpha}([0,T];H^{1-l}(\T^3;\R^3)).
		$$
		 Using similarly 
		\begin{align}
			& 	W^{1,\infty}([0,T];H^{-2}(\T^3;\R^3)) \hookrightarrow C^{\alpha}([0,T];H^{-l}(\T^3;\R^3)),
			\\
			& W^{1,2}([0,T];H^{-2}(\T^3;\R^3)) \hookrightarrow C^{\alpha}([0,T];H^{-l}(\T^3;\R^3)),
		\end{align}
		for the other deterministic integrals, we obtain that
		\begin{align*}
			\|u_{\delta,\epsilon}&\|_{C^{\alpha}([0,T];H^{-l}(\mathbb{T}^3;\R^3))}
			\le\, C(T,u_0)\\
			& + {C(\alpha,l)}\sqrt{\epsilon}   \left\|\int^{\cdot}_0\Delta u_{\delta,\epsilon}\,\d s\right\|_{W^{1,2}([0,T];L^2(\mathbb{T}^3;\R^3))}+ {C(\alpha,l)} \left(\frac{1}{2}+\frac{F_1}{4}\right)\left\|\int^{\cdot}_0 u_{\delta,\epsilon}\, \d s\right\|_{W^{1,\infty}([0,T];L^2(\T^3;\R^3))} \\
			&+{C(\alpha,l)}  \left\|\int^{\cdot}_0u_{\delta,\epsilon} \otimes u_{\delta,\epsilon}\,\d s\right\|_{W^{1,\infty}([0,T];L^1(\mathbb{T}^3;\R^{3\times3}))} 
			+ \left\| \frac{1}{\sqrt{2}} \sum_{n \in \mathbb{Z}^3} \int^{\cdot}_0\Pi \nabla \cdot (\psi_{\delta,\epsilon} f_n \,\d B_n) \right\|_{C^{\alpha}([0,T];H^{-l}(\mathbb{T}^3;\R^3))}.
		\end{align*}
		Then, using H\"older's inequality and the energy estimate \eqref{total-energy-es}, we find
		\begin{align}\label{eqn3737373}
			&\sqrt{\epsilon} \left\|\int^{\cdot}_0\Delta u_{\delta,\epsilon}\,\d s\right\|_{W^{1,2}([0,T];L^2(\mathbb{T}^3;\R^3))}+\left(\frac{1}{2}+\frac{F_1}{4}\right)\left\|\int^{\cdot}_0 u_{\delta,\epsilon} \,\d s\right\|_{W^{1,\infty}([0,T];L^2(\T^3;\R^3))}\\
			& + \left\|\int^{\cdot}_0u_{\delta,\epsilon} \otimes u_{\delta,\epsilon}\,\d s\right\|_{W^{1,\infty}([0,T];L^1(\mathbb{T}^3;\R^{3\times3}))}\notag\\
			\leq\,& C(u_0,\psi_0,T),
		\end{align}
		where the constant depends on the initial values via the initial  energy.
		
		Therefore,
		\begin{align*}
			\|u_{\delta,\epsilon}\|_{C^{\alpha}([0,T];H^{-l}(\mathbb{T}^3;\R^3))}
			\leq C(l,u_0,\psi_0,F_1,T) 
			+ \left\| \frac{1}{\sqrt{2}} \sum_{n \in \mathbb{Z}^3} \int^{\cdot}_0\Pi \nabla \cdot (\psi_{\delta,\epsilon} f_n \,\d B_n) \right\|_{C^{\alpha}([0,T];H^{-l}(\mathbb{T}^3;\R^3))}
		\end{align*}
		and it remains to bound the stochastic integral. For this we use that
		for every $p\in[2,\infty)$ and $0\leq t_1\leq t_2\leq T$, by the  Burkholder--Davis--Gundy inequality, Assumption \ref{Ass:noise} and \eqref{total-energy-es}, 
		\begin{align}\label{eqn_BDG_tightness}
			\mathbb{E}\left\|\frac{1}{\sqrt{2}} \sum_{n \in \mathbb{Z}^3} \int^{t_2}_{t_1}\Pi \nabla \cdot (\psi_{\delta,\epsilon} f_n \,\d B_n)\right\|_{H^{-l}(\mathbb{T}^3;\mathbb{R}^3)}^p\leq&C(p,F_1)\mathbb{E}\biggl|\int^{t_2}_{t_1}\|\psi_{\delta,\epsilon}\|_{L^2(\mathbb{T}^3)}^2\,\d r	\biggr|^{\frac{p}{2}}\\
			\leq&C(p,u_0,\psi_0,F_1)|t_2-t_1|^{\frac{p}{2}}. 
		\end{align}
		Choosing $p>2$ such that $\alpha<\frac{1}{2}-\frac{1}{p}$, by the Kolmogorov's continuity criterion \cite[Theorem 1.1]{KU23}, we have 
		$$
		\mathbb{E}\Big\|\frac{1}{\sqrt{2}} \sum_{n \in \mathbb{Z}^3} \int^{\cdot}_0\Pi \nabla \cdot (\psi_{\delta,\epsilon} f_n \,\d B_n)\Big\|_{C^{\alpha}([0,T];H^{-l}(\mathbb{T}^3{ ;\R^3}))}^p\leq C(\alpha,p, u_0,\psi_0,F_1,T).
		$$ 
Using Jensen's inequality, this yields that 
		\begin{align*}
			\mathbb{E}\|u_{\delta,\epsilon}\|_{C^{\alpha}([0,T];H^{-l}(\mathbb{T}^3;\R^3))}
			\leq&\, C(\alpha,l, u_0,\psi_0,F_1, T).
		\end{align*}

		Next, we focus on the equation of $\psi_{\delta,\epsilon}$. Taking the $\|\cdot\|_{W^{\alpha,2}([0,T];H^{-l}(\mathbb{T}^3))}$-norm in the equation \eqref{modification-eq} satisfied by $\psi_{\delta,\epsilon}$, we get 
		\begin{align*}
			\|\psi_{\delta,\epsilon}&\|_{W^{\alpha,2}([0,T];H^{-l}(\mathbb{T}^3))}\leq C(\psi_0,T)+(1+F_1/2+G_1/2)\left\|\int^{\cdot}_0\Delta\psi_{\delta,\epsilon}\,\d s\right\|_{W^{\alpha,2}([0,T];H^{-l}(\mathbb{T}^3))}\\
			&+\epsilon\left\|\int^{\cdot}_0\Delta^2\psi_{\delta,\epsilon}\,\d s\right\|_{W^{\alpha,2}([0,T];H^{-l}(\mathbb{T}^3))}+\left\|\int^{\cdot}_0h_\delta(\psi_{\delta,\epsilon})\left(|\nabla\psi_{\delta,\epsilon}|^2+|\nabla u_{\delta,\epsilon}|^2+\epsilon\right)\,\d s\right\|_{W^{\alpha,2}([0,T];H^{-l}(\mathbb{T}^3))}\\
			&+\left\|\int^{\cdot}_0\nabla\cdot(u_{\delta,\epsilon}\psi_{\delta,\epsilon})\,\d s\right\|_{W^{\alpha,2}([0,T];H^{-l}(\mathbb{T}^3))}+(F_2/2+G_2/8)\left\|\int^{\cdot}_0\psi_{\delta,\epsilon}\,\d s\right\|_{W^{\alpha,2}([0,T];H^{-l}(\mathbb{T}^3))}\\
			&+\left\|\int^{\cdot}_0\sum_{n\in\mathbb{Z}^3}\nabla\cdot(\psi_{\delta,\epsilon}g_n\,\d W_n)\right\|_{W^{\alpha,2}([0,T];H^{-l}(\mathbb{T}^3))}+\left\|\frac{1}{2}\int^{\cdot}_0\sum_{n\in\mathbb{Z}^3}\psi_{\delta,\epsilon}\nabla\cdot(g_n\,\d W_n)\right\|_{W^{\alpha,2}([0,T];H^{-l}(\mathbb{T}^3))}\\
			&+\left\|\int^{\cdot}_0\frac{1}{\sqrt{2}}\sum_{n\in\mathbb{Z}^3}(\nabla u_{\delta,\epsilon}:f_n\,\d B_n)\right\|_{W^{\alpha,2}([0,T];H^{-l}(\mathbb{T}^3))}. 	
		\end{align*}
		As above, by the choice of $l>5/2$, and with the help of the energy estimate \eqref{total-energy-es} as well as H\"older's inequality, we get 
		\begin{align*}&
			(1+F_1/2+G_1/2)\left\|\int^{\cdot}_0\Delta\psi_{\delta,\epsilon}\,\d s\right\|_{W^{\alpha,2}([0,T];H^{-l}(\mathbb{T}^3))}+\epsilon\left\|\int^{\cdot}_0\Delta^2\psi_{\delta,\epsilon}\,\d s\right\|_{W^{\alpha,2}([0,T];H^{-l}(\mathbb{T}^3))}\\& \quad \leq C(\alpha, l,u_0,\psi_0,T),\\& 
			\left\|\int^{\cdot}_0\nabla\cdot(u_{\delta,\epsilon}\psi_{\delta,\epsilon})\,\d s\right\|_{W^{\alpha,2}([0,T];H^{-l}(\mathbb{T}^3))}+(F_2/2+G_2/8)\left\|\int^{\cdot}_0\psi_{\delta,\epsilon}\,\d s\right\|_{W^{\alpha,2}([0,T];H^{-l}(\mathbb{T}^3))}\\&\quad \leq C(\alpha,l,u_0,\psi_0,T). 
		\end{align*}
		By the $L^1$-estimate \eqref{eqn151}, we observe moreover 
		\begin{align*}
			&\E \left\|\int^{\cdot}_0h_\delta(\psi_{\delta,\epsilon})\left(|\nabla\psi_{\delta,\epsilon}|^2+|\nabla u_{\delta,\epsilon}|^2+\epsilon\right)\,\d s\right\|_{W^{\alpha,2}([0,T];H^{-l}(\mathbb{T}^3))}\\
			\le \, & {C(\alpha,l)}\mathbb{E}\int^T_0\int_{\mathbb{T}^3}h_\delta(\psi_{\delta,\epsilon})\left(|\nabla\psi_{\delta,\epsilon}|^2+|\nabla u_{\delta,\epsilon}|^2+\epsilon\right)\,\d x\,\d t\\
			\leq\,&C(\alpha,l,\psi_0,u_0,F_2,G_2,T). 
		\end{align*}
		To estimate the contribution of  the noise term, we  use here \cite[Lemma 2.1]{FG95} and Assumption \ref{Ass:noise}, to deduce that
		\begin{align}\label{eqn_martingales_psi1}
			&\,\mathbb{E}\left\|\int^{\cdot}_0\sum_{n\in\mathbb{Z}^3}\nabla\cdot(\psi_{\delta,\epsilon}g_n\,\d W_n)\right\|_{W^{\alpha,2}([0,T];H^{-l}(\mathbb{T}^3))}^{ 2}\,+\,\left\|\frac{1}{2}\int^{\cdot}_0\sum_{n\in\mathbb{Z}^3}\psi_{\delta,\epsilon}\nabla\cdot(g_n\,\d W_n)\right\|_{W^{\alpha,2}([0,T];H^{-l}(\mathbb{T}^3))}^{2}\\
			\leq&\,{C(\alpha,l)} \left(\mathbb{E}\int^T_0\sum_{n\in\mathbb{Z}^3}\|\psi_{\delta,\epsilon}g_n\|_{L^2(\mathbb{T}^3)}^2\,\mathrm{d}t\right)\,+\,{C(\alpha,l)}\left(\mathbb{E}\int^T_0\sum_{n\in\mathbb{Z}^3}\|\psi_{\delta,\epsilon}\nabla g_n\|_{L^2(\mathbb{T}^3 ;\R^3)}^2\,\d t\right)\\
			\leq&\,C(\alpha, G_1,G_2,\psi_0,u_0). 
		\end{align}
		Furthermore, 
		\begin{align}\label{eqn_martingales_psi2}
			&\mathbb{E}\left\|\int^{\cdot}_0\frac{1}{\sqrt{2}}\sum_{n\in\mathbb{Z}^3}(\nabla u_{\delta,\epsilon}:f_n\,\d B_n)\right\|_{W^{\alpha,2}([0,T];H^{-l}(\mathbb{T}^3))}^{2}\\
			=\,&\mathbb{E}\left\|\int^{\cdot}_0\frac{1}{\sqrt{2}}\sum_{n\in\mathbb{Z}^3}[\nabla (u_{\delta,\epsilon}f_n)-u_{\delta,\epsilon}\otimes\nabla f_n]:\,\d B_n\right\|_{W^{\alpha,2}([0,T];H^{-l}(\mathbb{T}^3))}^{2}\\
			\leq\,&C(\alpha,l)\biggl(\mathbb{E}\int^T_0\sum_{n\in\mathbb{Z}^3}\|u_{\delta,\epsilon}f_n\|_{L^2(\mathbb{T}^3; \R^3)}^2dt\biggr)+ C(\alpha,l)\biggl(\mathbb{E}\int^T_0\sum_{n\in\mathbb{Z}^3}\|u_{\delta,\epsilon}\otimes \nabla f_n\|_{L^2(\mathbb{T}^3;\R^{3\times 3})}^2dt\biggr)\\
			\leq\,&C(\alpha,l, F_1,F_2,u_0,\psi_0). 
		\end{align}
	 Combining all preceding estimates, we arrive at \eqref{eqn872374}.
	\end{proof}
	
	In combination with the energy estimate \eqref{total-energy-es}, this yields the following tightness result.

	\begin{lemma}\label{tightness-1}
		Let  
		$(u_{\delta,\epsilon} ,\psi_{\delta,\epsilon})$ be the weak martingale solutions to  \eqref{Eq_approx} with initial data $(u_0,\psi_0)$ constructed in Proposition \ref{prop_1}. Then for fixed $\epsilon>0$, every $\alpha\in(0,1/2)$ and every $l>5/2$, the family \[(u_{\delta,\epsilon},\psi_{\delta,\epsilon},\nabla u_{\delta,\epsilon},\nabla\psi_{\delta,\epsilon})_{\delta\in(0,1)}\]
		is tight in 
		\begin{align*}
		L^2([0,T];L^2(\mathbb{T}^3;\R^3))\cap &C^{\alpha}([0,T];H^{-l}(\T^3;\R^3))\times L^2([0,T];L^2(\mathbb{T}^3))\\
		&\times(L^2([0,T];L^2(\mathbb{T}^3;\R^{3\times 3})),w)\times(L^2([0,T];L^2(\mathbb{T}^3 ;\R^3)),w). 
		\end{align*}
	\end{lemma}
	\begin{proof}
		Applying the Aubin--Lions--Simon \cite{Simon} compactness criterion, the proof is a straightforward consequence of \eqref{total-energy-es} and \eqref{eqn872374} via the Chebyshev--Markov inequality. 
	\end{proof}
	
	In the following, we require independent   tightness of  the martingale terms in \eqref{Eq_approx}. For this we set
	\begin{align}\label{eqn862346d}
		M^1_{\delta,\epsilon}=-\frac{1}{\sqrt{2}}\sum_{n\in\mathbb{Z}^3}\int^t_0\Pi\nabla\cdot(\psi_{\delta,\epsilon} f_n \,\d B_{n}),
	\end{align}
	and 
	\begin{align}\label{eqn862346e}
		M^2_{\delta,\epsilon}=-\sum_{n\in\mathbb{Z}^3}\int^t_0\nabla\cdot(\psi_{\delta,\epsilon} g_n\,\d W_n)+\frac{1}{2}\sum_{n\in\mathbb{Z}^3}\int^t_0\psi_{\delta,\epsilon}\nabla\cdot(g_n\,\d W_n)-\frac{1}{\sqrt{2}}\sum_{n\in\mathbb{Z}^3}\int^t_0(\nabla u_{\delta,\epsilon}:f_n\,\d B_n),
	\end{align}
	for which have the following  result.

	\begin{lemma}\label{lem-tight-martingales}Let 
		$(u_{\delta,\epsilon} ,\psi_{\delta,\epsilon})$ be the weak martingale solutions to  \eqref{Eq_approx} with initial data $(u_0,\psi_0)$ constructed in Proposition \ref{prop_1}. Then
		for  fixed $\epsilon\in(0,1)$, every  $\gamma\in(0,1/2)$ and every  $s>1$,  $(M^1_{\delta,\epsilon})_{\delta\in(0,1)}$ and $(M^2_{\delta,\epsilon})_{\delta\in(0,1)}$ are tight in $C^{\gamma}([0,T];H^{-s}(\mathbb{T}^3;\mathbb{R}^3))$ and $C^{\gamma}([0,T];H^{-s}(\mathbb{T}^3))$, respectively. 
	\end{lemma}
	\begin{proof}
		For every $p\in[2,\infty)$ and $0\leq t_1\leq t_2\leq T$, similarly to the computations \eqref{eqn_BDG_tightness}, \eqref{eqn_martingales_psi1} and \eqref{eqn_martingales_psi2}, we deduce that 
		\begin{align*}
			&\mathbb{E}\|M^1_{\delta,\epsilon}(t_2)-M^1_{\delta,\epsilon}(t_1)\|_{H^{-1}(\mathbb{T}^3;\mathbb{R}^3)}^p+\mathbb{E}\|M^2_{\delta,\epsilon}(t_2)-M^2_{\delta,\epsilon}(t_1)\|_{H^{-1}(\mathbb{T}^3)}^p\\
			\leq&C(p,F_1,F_2,G_1,G_2)\mathbb{E}\Big|\int^{t_2}_{t_1}\|\psi_{\delta,\epsilon}\|_{L^2(\mathbb{T}^3)}^2+\|u_{\delta,\epsilon}\|_{L^2(\mathbb{T}^3; \R^3)}^2\,\d r	\Big|^{\frac{p}{2}}\\
			\leq&C(p,u_0,\psi_0,F_1,F_2,G_1,G_2)|t_2-t_1|^{\frac{p}{2}}, 
		\end{align*}
		by \eqref{total-energy-es}. Now we 
		let $\gamma'\in(\gamma,1/2)$. Choosing $p>2$ such that $\gamma'<\frac{1}{2}-\frac{1}{p}$, by the Kolmogorov's continuity criterion \cite[Theorem 1.1]{KU23}, we have 
		$$
		\mathbb{E}\|M^1_{\delta,\epsilon}\|_{C^{\gamma'}([0,T];H^{-1}(\mathbb{T}^3{ ;\R^3}))}^p+\mathbb{E}\|M^2_{\delta,\epsilon}\|_{C^{\gamma'}([0,T];H^{-1}(\mathbb{T}^3))}^p\leq C(p,\gamma',u_0,\psi_0,F_1,F_2,G_1,G_2, T).
		$$ 
		Then the compact embedding $C^{\gamma'}_tH^{-1}_x\subset C^{\gamma}_tH^{-s}_x$ concludes the proof.
		
	\end{proof}
	
	\subsubsection{Limiting procedure: $\delta\to 0$.}
	
	Here we complete our construction of weak solutions to \eqref{ep_approx}--\eqref{eqn354865865}. 
	
	\begin{proposition}\label{prp-existence-ep} 
		For every $\epsilon\in(0,1)$, there	 exist a   weak martingale solutions to \eqref{ep_approx}--\eqref{eqn354865865} with initial data $(u_0,\psi_0)$, which satisfies 
		\begin{equation}\label{total-energy-es-es}\mathrm{ess \, sup}_{t\in [0,T]}\biggl(
			 \frac{1}{2}\int_{\T^3}
			|u_{\epsilon,t}|^2 +\psi_{\epsilon,t}^2 \,\d x 	+ \epsilon \int_0^t \int_{\T^3} 
			|\Delta  u_{\epsilon}|^2
			+ |\Delta \psi_{\epsilon}|^2\,\d x\,\d s\,-\,\epsilon t\biggr) \,\le\, 
			\frac{1}{2}\int_{\T^3}
			|u_0|^2 +\psi_0^2 \,\d x,
		\end{equation}
		$\P$-almost surely, as well as 
		\begin{equation}\label{eq_meas_bound_eps}
			\E \bigl[\| q_\epsilon \|_{ \M([0,T]\times \T^3) }\bigr]\,\le\, 	\biggl(1 \,+\, \frac{4F_2 +G_2 }{8}T \biggr)  \biggl(
			\frac{1}{2}\int_{\T^3}|u_0|^2 +|\psi_0|^2 \,\d x\,+\,2 \epsilon T
			\biggr)^{1/2}.
		\end{equation} 
	\end{proposition}
	\begin{proof}
		Since $\epsilon\in(0,1)$ is fixed in the whole proof, we do not emphasize the dependency on it and denote the solutions to \eqref{Eq_approx} constructed in Proposition \ref{prop_1} by $(u_{\delta},\psi_{\delta})_{\delta\in(0,1)}$. Moreover, we recall that for the ease of notation we may assume  the latter are defined with respect to the same stochastic basis, see the comments below Definition  \ref{defi_solution_ep}.
		
		\textit{Step 1 (Application of the Skorokhod--Jakubowski theorem).} Throughout this proof, we fix $\alpha\in(0,1/2)$ and $l>5/2$. Then, by
		 Lemma \ref{tightness-1} {the sequence}
		\begin{align*}
			(u_{\delta}, \psi_{\delta}, \nabla u_{\delta}, \nabla \psi_{\delta}) 
		\end{align*}
		is tight in
		\begin{align*}
		L^2([0,T]; L^2(\mathbb{T}^3;\R^3))\cap& C^{\alpha}([0,T];H^{-l}(\T^3;\R^3))\times L^2([0,T]; L^2(\mathbb{T}^3))\\
		&\times (L^2([0,T]; L^2(\mathbb{T}^3;\R^{3\times3})), w)\times (L^2([0,T]; L^2(\mathbb{T}^3;\R^3)), w).
		\end{align*}Fixing also $\gamma\in(0,1/2)$ and $s>1$,  Lemma \ref{lem-tight-martingales} states that
		\begin{align}
			M^1_{\delta}=-\frac{1}{\sqrt{2}}\sum_{n\in\mathbb{Z}^3}\int^t_0\Pi\nabla\cdot(\psi_{\delta} f_n \,\d B_{n})
		\end{align}
		and 
		\begin{align}
			M^2_{\delta}=-\sum_{n\in\mathbb{Z}^3}\int^t_0\nabla\cdot(\psi_{\delta} g_n\,\d W_n)+\frac{1}{2}\sum_{n\in\mathbb{Z}^3}\int^t_0\psi_{\delta}\nabla\cdot(g_n\,\d W_n)-\frac{1}{\sqrt{2}}\sum_{n\in\mathbb{Z}^3}\int^t_0(\nabla u_{\delta}:f_n\,\d B_n) 
		\end{align}
		 are tight in  $C^{\gamma}([0,T];H^{-s}(\mathbb{T}^3;\mathbb{R}^3))$ and $C^{\gamma}([0,T];H^{-s}(\mathbb{T}^3))$, respectively.
		Defining 
		\[
		q_\delta \,=\, h_\delta(\psi_{\delta})\bigl(
		|\nabla \psi_{\delta}|^2 + |\nabla_\sym u_{\delta}|^2 + \epsilon
		\bigr),
		\]	 
 		the latter sequence of measures  is tight in $(\mathcal{M}([0,T]\times\mathbb{T}^3),w^*)$ in light of Lemma \ref{lem-L1-es}.
 		Since the laws of the Brownian motions $B_n$, $W_n$ are independent of $\delta$, we deduce tightness of the sequence 
 		$$
 		(X_{\delta}, (B_n)_{n\in\mathbb{Z}^3}, (W_n)_{n\in\mathbb{Z}^3} )_{\delta\in(0,1)}
 		$$ in the topological product space
 		\begin{equation}\label{eqn_space}
 		\mathbb{X} \times C([0,T];\mathbb{R}^{3\times 3})^{\mathbb{Z}^3} \times C([0,T];\mathbb{R}^{3})^{\mathbb{Z}^3},
 	\end{equation}
 		where 
 		$$X_{\delta} := \Big(u_{\delta}, \psi_{\delta}, \nabla u_{\delta}, \nabla \psi_{\delta},q_{\delta}, M^1_{\delta},M^2_{\delta}\Big),$$
 		and
 		\begin{align*}
 			\mathbb{X}:=&L^2([0,T]; L^2(\mathbb{T}^3;{\R^3}))\cap C^{\alpha}([0,T];H^{-l}(\T^3;\R^3))\times L^2([0,T]; L^2(\mathbb{T}^3)) \times (L^2([0,T]; L^2(\mathbb{T}^3{;\R^{3\times 3}})), w)\\
 			&\times (L^2([0,T]; L^2(\mathbb{T}^3;{\R^3})), w)\times(\mathcal{M}([0,T]\times\mathbb{T}^3),w^*)\times C^{\gamma}([0,T];H^{-s}(\mathbb{T}^3;\mathbb{R}^3))\times C^{\gamma}([0,T];H^{-s}(\mathbb{T}^3)).
 		\end{align*}
 	Applying the Skorokhod--Jakubowski  theorem \cite{Jak97}, we deduce that there exists a new probability space $(\bar{\Omega}, \bar{\mathfrak{A}}, \bar{\mathbb{P}})$ and a sequence of random variables satisfying
		$$
		(\bar{X}_{\delta}, (\bar{B}_{n,\delta})_{n\in\mathbb{Z}^3},(\bar{W}_{n,\delta})_{n\in\mathbb{Z}^3})\,\to\, (\bar{X}, (\bar{B}_{n})_{n\in\mathbb{Z}^3},(\bar{W}_{n})_{n\in\mathbb{Z}^3}) ,
		$$
		$\bar{\P}$-a.s., in \eqref{eqn_space},
		such that the  laws satisfy
		\begin{align}\label{identity-jointlaws}
			(\bar{X}_{\delta}, (\bar{B}_{n,\delta})_{n\in\mathbb{Z}^3},(\bar{W}_{n,\delta})_{n\in\mathbb{Z}^3}) \overset{d}{\,=\,}& (X_{\delta}, (B_{n})_{n\in\mathbb{Z}^3} ,(W_{n})_{n\in\mathbb{Z}^3}), 
		\end{align}
		along a subsequence $\delta\to 0$.
	Breaking the above vectors up into their individual components, we denote in the following $\bar{X}_{\delta} = (\bar{u}_{\delta}, \bar{\psi}_{\delta},\bar{f}_{\delta},\bar{g}_{\delta},\bar{q}_{\delta},\bar{M}^1_{\delta},\bar{M}^2_{\delta})$ and $\bar{X} = (\bar{u}, \bar{\psi},\bar{f},\bar{g},\bar{q},\bar{M}^1,\bar{M}^2)$.

	\textit{Step 2 (Transferring previous relations to the new probability space).} 
The equality in law \eqref{identity-jointlaws} allows us to argue that $
\bar{f}_{\delta}= \nabla \bar{u}_{\delta}$ and $ \bar{g}_{\delta} = \nabla \bar{\psi}_{\delta}$. Indeed, for any test function $\eta \in C^\infty([0,T]; \T^3)$, we have 
\begin{align}&
\bar{\P} \biggl(\int_0^T \int_{\T^3} \bar{u}_{\delta,i} \partial_j  \vp \,\d x\,\d t \,+\, 
\int_0^T  \int_{\T^3} \bar{f}_{\delta, ij} \vp \,\d x\,\d t\,=\, 0 \biggr) \\&\quad =\, 
{\P} \biggl(\int_0^T \int_{\T^3} u_{\delta,i} \partial_j  \vp \,\d x\,\d t \,+\, \int_0^T
\int_{\T^3} \partial_j {u}_{\delta, i} \vp \,\d x\,\d t\,=\, 0 \biggr) \,=\, 1,
\end{align}
so that the above identification follows. Similarly, for any $\vp \in C([0,T]\times \T^3)$, the functional 
\begin{align} 
\Gamma_\vp \colon &	L^2([0,T]; L^2(\mathbb{T}^3)) \times (L^2([0,T]; L^2(\mathbb{T}^3{;\R^{3\times 3}})), w)\\&\quad \times (L^2([0,T]; L^2(\mathbb{T}^3;{\R^3})), w)\times 
	(\mathcal{M}([0,T]\times\mathbb{T}^3),w^*) \,\to\, \R,
	\\& 
	(\psi, f,g,q )\,\mapsto 
	\int_{[0,T]\times \T^3} \vp \,\d q \,-\, \int_0^T \int_{\T^3} \vp h_\delta(\psi  ) \bigl(
	|f_\sym|^2 + |g|^2 +\epsilon
	\bigr)\,\d x \,\d t,
\end{align}
is measurable. Thus,
\begin{align}
	\bar{\P} \bigl(\Gamma_\vp (\bar{\psi}_\delta , \nabla \bar{u}_\delta, \nabla \bar{\psi}_\delta ,\bar{q}_\delta) = 0 \bigr) \,=\, 
	{\P} \bigl(\Gamma_\vp (\psi_\delta , \nabla {u}_\delta, \nabla {\psi}_\delta ,{q}_\delta) = 0  \bigr) \,=\, 1,
\end{align}
identifying also
\begin{equation}\label{eqn3278474328742}
\bar{q}_\delta \,=\, h_\delta(\bar{\psi}_{\delta})\bigl(
|\nabla \bar{\psi}_{\delta}|^2 + |\nabla_\sym \bar{u}_{\delta}|^2 + \epsilon
\bigr),
\end{equation}
$\bar{\P}$-almost surely.

Let us also relate $\bar{M}^1_{\delta}$ and $\bar{M}^2_{\delta}$ to the variables $\bar{u}_\delta$ and $\bar{\psi}_\delta$. To this end, we let  $\varphi \in C^\infty(\mathbb{T}^{3};\mathbb{R}^3)$ be divergence-free, $\zeta \in C_c^\infty([0,T)\times\mathbb{T}^3)$ and introduce the notation 
\begin{align}\begin{split}\label{eqn737232}&
	\bar{M}_{\varphi,\delta}^1=\langle\varphi,\bar{M}^1_{\delta}\rangle,\qquad  \bar{M}_{\zeta,\delta}^2(t)=\int^t_0\langle\zeta,d\bar{M}^2_{\delta}(s)\rangle,\\&
	M_{\varphi,\delta}^1=\langle\varphi,M^1_{\delta}\rangle,\qquad  M_{\zeta,\delta}^2(t)=\int^t_0\langle\zeta,dM^2_{\delta}(s)\rangle . 	
\end{split}
\end{align}
for every $t \in [0,T]$.
Then, similarly to the previous arguments we find that  \eqref{identity-jointlaws} implies 
\begin{align}\begin{split}\label{eqn33}
	\bar{M}_{\varphi,\delta}^1(t)&=\int_{\T^3}  \bar{u}_{\delta}(t)\cdot \vp\,\d x \,-\, 
	\int_{\T^3} u_0 \cdot \vp\,\d x+\epsilon\int_0^t\int_{\T^3}\bar{u}_{\delta}\cdot\Delta^2\vp\,\d x\, \d s
	\\&\quad\,  +\int_0^t \int_{\T^3}   \Bigl(\frac{1}{2}+ \frac{F_1}{4}\Bigr) \nabla \bar{u}_{\delta}: \nabla  \vp \,-\, \bar{u}_{\delta}\otimes \bar{u}_{\delta}:\nabla \vp \,\d x\,\d s,
	\end{split}
\end{align}
for every $t\in[0,T]$,
and
\begin{align}\begin{split}\label{eqn34}
	\bar{M}_{\zeta,\delta}^2(T)\,=\,&-
	\int_0^T \int_{\T^3} \bar{\psi}_{\delta}  \partial_t \zeta  \, \d x \,\d s \, -\, \int_{\T^3}
	\psi_0 \zeta(0)
	\,\d x+\epsilon\int^T_0\int_{\T^3}\bar{\psi}_{\delta}\Delta^2\zeta \,\d x\,\d s
	\\&  +\, \int_0^T \int_{\T^3} \bigl(1+ {F_1}/2 +G_1/2\bigr) \nabla \bar{\psi}_{\delta}\cdot \nabla  \zeta \,-\,\bar{\psi}_{\delta} \bar{u}_{\delta}\cdot \nabla \zeta   \,+\, (F_2/2+G_2/8) \bar{\psi}_{\delta} \zeta \,\d x\,\d s \\& -\, \int_{[0,T]\times\T^3} \zeta \,\d \bar{q}_{\delta},
	\end{split}
\end{align} $\bar{\P}$-almost surely.

		\textit{Step 3 (Taking $\delta \to 0$ in  \eqref{eqn33}--\eqref{eqn34}).}
		Analogous to \eqref{eqn737232}, we introduce also the notation 
		\[ \bar{M}_{\varphi}^1=\langle\varphi,\bar{M}^1\rangle,\qquad  \bar{M}_{\zeta}^2(t)=\int^t_0\langle\zeta,d\bar{M}^2(s)\rangle.\]
		With the help of the convergence in $$L^2([0,T];L^2(\mathbb{T}^3;\R^3))\cap C^{\alpha}([0,T];H^{-l}(\T^3;\R^3))\cap (L^2([0,T];H^1(\mathbb{T}^3;\R^3)),w)$$ of $\bar{u}_{\delta}\rightarrow u$, and $\bar{M}^1_{\varphi,\delta}(t)\rightarrow\bar{M}^1_{\varphi}(t)$, as $\delta\to0$, $\bar{\P}$-a.s., we obtain that, $\bar{\P}$-a.s., for every $t\in[0,T]$, 
		\begin{align}\begin{split}\label{eqn33_lim}
			\bar{M}_{\varphi}^1(t)=&\int_{\T^3}  \bar{u}(t)\cdot \vp\,\d x \,-\, 
			\int_{\T^3} u_0 \cdot \vp\,\d x\,+\,\epsilon\int_0^t\int_{\T^3}\bar{u}\cdot\Delta^2\vp\,\d x\,\d s
			\\&\quad\,  +\int_0^t \int_{\T^3}   \Bigl(\frac{1}{2}+ \frac{F_1}{4}\Bigr) \nabla \bar{u}: \nabla  \vp \,-\, \bar{u}\otimes \bar{u}:\nabla \vp \,\d x\,\d s. \end{split}
		\end{align}
		By further using the convergence in $L^2([0,T];L^2(\mathbb{T}^3))\cap (L^2([0,T];H^1(\mathbb{T}^3)),w)$ of $\bar{\psi}_{\delta}\rightarrow \psi$, $\bar{M}^2_{\zeta,\delta}(T)\rightarrow\bar{M}^2_{\varphi}(T)$, and $\bar{q}_{\delta}\rightarrow\bar{q}$ in $(\mathcal{M}([0,T]\times\mathbb{T}^3),w^*)$, as $\delta\to0$, $\P$-a.s., it follows from H\"older's inequality that 
		\begin{align}\begin{split}\label{eqn34_lim}
			\bar{M}^2_{\zeta}(T)=&-\,\int_0^T \int_{\T^3} \bar{\psi}  \partial_t \zeta  \,\d x\,\d s \, -\, \int_{\T^3}
			\psi_0 \zeta
			\,\d x+\epsilon\int^T_0\int_{\T^3}\bar{\psi}\Delta^2\zeta \,\d x\,\d s
			\\&  +\,\int_0^T \int_{\T^3} \bigl(1+ {F_1}/2 +G_1/2\bigr) \nabla \bar{\psi}\cdot \nabla  \zeta \,-\,\bar{\psi} \bar{u}\cdot \nabla \zeta   \,+\, (F_2/2+G_2/8) \bar{\psi} \zeta\,\d x\,\d s\\& -\,\int_{[0,T]\times\T^3} \zeta \,\d \bar{q}. 
		\end{split}
		\end{align}
		
		 \textit{Step 4 (Identification of $\bar{B}$ and $ \bar{W}$  as  Brownian motions).}
	At this point, in order to define a filtration on $(\bar{\Omega}, \bar{\mathfrak{A}}, \bar{\mathbb{P}})$, we also introduce the 
	space
		\begin{align}
		\mathbb{Y}_t \,=\, L^2([0,t]; L^2(\mathbb{T}^3;\R^3)) \times& L^2([0,t]; L^2(\mathbb{T}^3)) \times \D'( (-\infty,t) \times \T^3 )\\& \times C([0,t];H^{-s}(\mathbb{T}^3;\mathbb{R}^3)) \times C([0,t];H^{-s}(\mathbb{T}^3;\mathbb{R}^3)).
	\end{align}
	 The reason to define the latter  is that 
	 $\bar{Y}_{\delta,t} \to \bar{Y}_{t} $ in $
	 \mathbb{Y}_t$ for
	 \begin{align}
	 	\bar{Y}_{\delta,t}\,=\, (\bar{u}_\delta|_{[0,t]}, \bar{\psi}_\delta|_{[0,t]}, \bar{q}_\delta|_{[0,s)},\bar{M}_\delta^1|_{[0,t]},\bar{M}_\delta^2|_{[0,t]} ) ,\qquad 	\bar{Y}_{t}\,=\, (\bar{u}|_{[0,t]}, \bar{\psi}|_{[0,t]}, \bar{q}|_{[0,t)} ,\bar{M}^1|_{[0,t]},\bar{M}^2|_{[0,t]} ),
	 \end{align}
	 which would not be the case if we considered $ \bar{q}_\delta|_{[0,t)}$
	 and $\bar{q}|_{[0,t)}$  as measures on $[0,t)$. Defining then the filtration
		$$
		\mathscr{G}_t := \sigma\Big(\bar{Y}_{t},\bar{B}|_{[0,t]} , ,\bar{W}|_{[0,t]} \Big),\qquad t\in [0,T],
		$$
		we find that $\bar{u}$, $ \bar{\psi}$, $\bar{q}$, $\bar{M}^1$ and $\bar{M}^2$ are $\mathscr{G}$-adapted in the sense of distributions in the sense of 	\cite[Definition 2.2.13]{BFH}. The latter notion is equivalent to usual adaptedness for those processes which are continuous in time. 
		We let moreover   $\bar{\mathscr{G}}$ be the augmentation of $\mathscr{G}$, so that all the aforementioned processes are $\bar{\mathscr{G}}$-progressively measurable by right-continuity of the filtration, cf.\ Remark \ref{rem_weak_sol}~\eqref{rem_weak_sol_I1}.
		
		Taking  $F:\mathbb{Y}_s\times   C([0,s];\mathbb{R}^{3\times 3})^{\mathbb{Z}^3} \times C([0,s];\mathbb{R}^{3})^{\mathbb{Z}^3} \rightarrow\mathbb{R}$ to
		be a  bounded continuous function, we obtain
		 for every $0 \leq s \leq t \leq T$, every $ 1\le i,j\le 3$ and every $n\in\Z^3$, by the identity in law \eqref{identity-jointlaws}, that
		\begin{align}&\label{eqn_last_one}
			\bar{\mathbb{E}}\Big(F(\bar{Y}_{\delta,s} ,\bar{B}_{\delta}|_{[0,s]},\bar{W}_{\delta}|_{[0,s]})(\bar{B}_{n,\delta}^{ij}(t)-\bar{B}_{n,\delta}^{ij}(s))\Big)\\
			&\quad = \mathbb{E}\Big(F({Y}_{\delta,s},B_{\delta}|_{[0,s]},W_{\delta}|_{[0,s]})(B_{n,\delta}^{ij}(t)-B_{n,\delta}^{ij}(s))\Big) = 0,
		\end{align}
	where 
	\[
		{Y}_{\delta,s}\,=\, ({u}_\delta|_{[0,s]}, {\psi}_\delta|_{[0,s]}, {q}_\delta|_{[0,s)},{M}_\delta^1|_{[0,s]},{M}_\delta^2|_{[0,s]}).
	\]
		By continuity of $F$ and  Vitali's convergence theorem, we are able to pass to the limit  $\delta \to 0$, deducing that 
		\begin{align}\label{eqn_last_two}
			\bar{\mathbb{E}}\Big(F(\bar{Y}_{s},\bar{B}|_{[0,s]},\bar{W}|_{[0,s]})(\bar{B}_n^{ij}(t)-\bar{B}_n^{ij}(s))\Big) = 0.
		\end{align}
		An analogous argument yields  also for all $ 1\le k,l\le 3$ and $m \in\mathbb{Z}^3$, 
		\begin{align*}
			\bar{\mathbb{E}}\Big(F(\bar{Y}_{s},\bar{B}|_{[0,s]},\bar{W}|_{[0,s]})\bigl(\bar{B}_n^{ij}(t)\bar{B}_m^{kl}(t)-\bar{B}_n^{ij}(s)\bar{B}_m^{kl}(s)-(t-s)\delta_{nm}( \delta_{ik} \delta_{jl}+\delta_{il}\delta_{jk}) \bigr)\Big) = 0.
		\end{align*}
		These identities establish that $(\bar{B}_n)_{n\in\mathbb{Z}^3}$ is a sequence of independent $\R^{3\times 3}$-valued $\mathscr{G}$-Brownian motions with covariance \eqref{Eq_noise_matrix_covariation}. Using continuity in time and Doob's maximal inequality yields that $(\bar{B}_n)_{n\in\mathbb{Z}^3}$ remains a sequence of Brownian motions with respect to the augmented filtration $\bar{\mathscr{G}}$. Similarly, one shows that $(\bar{W}_n)_{n\in\mathbb{Z}^3}$ are independent standard $\R^3$-valued $\bar{\mathscr{G}}$-Brownian motions, and independent of $(\bar{B}_n)_{n\in\mathbb{Z}^3}$. 
		
		 \textit{Step 5 (Identification of $\bar{M}_{\varphi}^1$ and $\bar{M}_{\zeta}^2$ as stochastic integrals).} That $\bar{M}_{\varphi}^1$ and $\bar{M}_{\zeta}^2$ are $\bar{\mathscr{G}}$-martingales follows along the lines of \eqref{eqn_last_one} and \eqref{eqn_last_two}. Indeed, the required uniform integrability to apply Vitali's convergence theorem when taking $\delta\to 0$ can be obtained by bounding the quadratic variation of ${M}_{\varphi,\delta}^1$ and ${M}_{\zeta,\delta }^2$ using the energy estimate \eqref{total-energy-es}, cf.\ Lemma \ref{lem-tight-martingales}. 
		 Taking again $F$ as above, we find for every $0 \leq s \leq t \leq T$, and every  $n\in\Z^3$ and $1\leq i,j\leq 3$, using  identity in law, 
		\begin{align*}
			\bar{\mathbb{E}}\biggr(F&(\bar{Y}_{\delta,s} ,\bar{B}_{\delta}|_{[0,s]},\bar{W}_{\delta}|_{[0,s]})\\
			&\quad\quad\cdot\biggr(\bar{M}^1_{\varphi,\delta}(t)\bar{B}^{ij}_{n,\delta}(t)-\bar{M}^1_{\varphi,\delta}(s)\bar{B}^{ij}_{n,\delta}(s)- {\sqrt{2}}\int^t_s\int_{\mathbb{T}^{3}}(\nabla_\sym\varphi)_{ij}\bar{\psi}_{\delta}f_n\,\d x\,\d r\biggr)\biggr)\\
			=\mathbb{E}\biggr(F&({Y}_{\delta,s} ,{B}_{\delta}|_{[0,s]},{W}_{\delta}|_{[0,s]})\\
			&\quad\quad\cdot\biggr(M^1_{\varphi,\delta}(t)B^{ij}_{n}(t)-M^1_{\varphi,\delta}(s)B^{ij}_{n}(s)- \sqrt{2}\int^t_s\int_{\mathbb{T}^{3}}(\nabla_\sym\varphi)_{ij}\psi_{\delta}f_n\,\d x\,\d r\biggr)\biggr)=0.  
		\end{align*}
		By the continuity of $F$ and convergence of the involved processes, we may pass with $\delta \to 0$, yielding that
		\begin{align*}
			\bar{\mathbb{E}}\biggl(F&(\bar{Y}_{s} ,\bar{B}|_{[0,s]},\bar{W}|_{[0,s]})\\
			&\quad\quad\cdot\biggr(\bar{M}^1_{\varphi}(t)\bar{B}^{ij}_{n}(t)-\bar{M}^1_{\varphi}(s)\bar{B}^{ij}_{n}(s)- \sqrt{2}\int^t_s\int_{\mathbb{T}^{3}}(\nabla_\sym\varphi)_{ij}\bar{\psi}f_n\,\d x\,\d r\biggr)\biggr)=0.
		\end{align*}
		This identity implies that, for every $n\in\mathbb{Z}^3$ and $1\leq i,j\leq 3$,
		\begin{align}\label{martingale-identity}
			\bar{M}^1_{\varphi}(t)\bar{B}^{ij}_{n}(t)- \sqrt{2}\int^t_s\int_{\mathbb{T}^{3}}(\nabla_\sym\varphi)_{ij}\bar{\psi}f_n\,\d x\,\d r
		\end{align}
		is a $\mathscr{G}$-martingale  and  also a $\bar{\mathscr{G}}$-martingale by  continuity in time. 
		Analogously, one verifies that the process
		\begin{align*}
			&(\bar{M}^1_{\varphi}(t))^2-  \int_0^t \sum_{n\in\mathbb{Z}^3}\sum_{i,j=1}^3\Big|\int_{\mathbb{T}^{3}}(\nabla_\sym \vp )_{ij} \bar{\psi}f_n\,\d x\Big|^2\,\d r 
		\end{align*}
		is a  $\bar{\mathscr{G}}$-martingale as well. Consequently, we have identified the quadratic variation of $\bar{M}^1_{\varphi}$ and its covariation with the Brownian motions $\bar{B}^{ij}_{n}$. This allows for a direct computation of the quadratic variation of the martingale 
		\[
		\bar{M}^1_{\varphi}-\frac{1}{\sqrt{2}}\sum_{n\in\mathbb{Z}^3}\int_{\mathbb{T}^{3}}\int^\cdot_0\nabla\varphi:\bar{\psi}f_n\,\d \bar{B}_n\,\d x,
		\]
		 which yields that
		\begin{align}\label{martingale-1}
			\bar{\mathbb{E}}\biggl[\bar{M}^1_{\varphi}(t)-\frac{1}{\sqrt{2}}\sum_{n\in\mathbb{Z}^3}\int^t_0\int_{\mathbb{T}^{3}}\nabla\varphi:\bar{\psi}f_n\,\d \bar{B}_n\biggr]^2 = 0, \qquad \text{for every } t \in [0,T]. 
		\end{align}

		Applying the same argument to $\bar{M}^2_{\zeta,\delta}$ yields for every $0 \leq s \leq t \leq T$ and  $n\in\mathbb{Z}^3$ and   $1\leq i \leq 3$ that 
		\begin{align*}
			\bar{\mathbb{E}}\biggr(F&(\bar{Y}_{\delta,s} ,\bar{B}_{\delta}|_{[0,s]},\bar{W}_{\delta}|_{[0,s]})\\
			&\cdot\biggr(\bar{M}^2_{\zeta,\delta}(t)\bar{W}^i_{n,\delta}(t)-\bar{M}^2_{\zeta,\delta}(s)\bar{W}^i_{n,\delta}(s)-\int^t_s\int_{\mathbb{T}^3}(\nabla\zeta)_i\bar{\psi}_{\delta}g_n\,\d x\,\d r-\frac{1}{2}\int^t_s\int_{\mathbb{T}^3}\zeta\bar{\psi}_{\delta}(\nabla g_n)_i \,\d x\,\d r\biggr)\biggr)\\
			=\mathbb{E}\biggr(F&({Y}_{\delta,s} ,{B}_{\delta}|_{[0,s]},{W}_{\delta}|_{[0,s]})\\
			&\cdot\biggr(M^2_{\zeta,\delta}(t)W^i_{n}(t)-M^2_{\zeta,\delta}(s)W^i_{n}(s)-\int^t_s\int_{\mathbb{T}^3}(\nabla\zeta)_i\psi_{\delta}g_n\,\d x\,\d r-\frac{1}{2}\int^t_s\int_{\mathbb{T}^3}\zeta\psi_{\delta}(\nabla g_n)_i \,\d x\,\d r\biggr)\biggr)=0. 
		\end{align*}
	Therefore,
		\begin{align*}
			\bar{\mathbb{E}}\biggr(F&(\bar{Y}_{s} ,\bar{B}|_{[0,s]},\bar{W}|_{[0,s]})\\
			&\cdot\biggr(\bar{M}^2_{\zeta}(t)\bar{W}^i_{n}(t)-\bar{M}^2_{\zeta}(s)\bar{W}^i_{n}(s)-\int^t_s\int_{\mathbb{T}^3}(\nabla\zeta)_i\bar{\psi}g_n\,\d x\,\d r-\frac{1}{2}\int^t_s\int_{\mathbb{T}^3}\zeta\bar{\psi}(\nabla g_n)_i \,\d x\,\d r\biggr)\biggr)=0,
		\end{align*}
		which shows that 
		$$
		\bar{M}^2_{\zeta}(t)\bar{W}^i_{n}(t)-\int^t_0\int_{\mathbb{T}^3}(\nabla\zeta)_i\bar{\psi}g_n\,\d x\,\d r-\frac{1}{2}\int^t_0\int_{\mathbb{T}^3}\zeta\bar{\psi}(\nabla g_n)_i \,\d x\,\d r
		$$
		is a $\mathscr{G}$-martingale, as well as a $\bar{\mathscr{G}}$-martingale due to  time continuity.
		
		Similarly, we conclude that for every $0 \leq s \leq t \leq T$, $n\in\mathbb{Z}^3$ and $1\leq i,j\leq 3$, 
		\begin{align*}
			\bar{\mathbb{E}}\biggr(F&(\bar{Y}_{s} ,\bar{B}|_{[0,s]},\bar{W}|_{[0,s]})\\
			&\quad\quad\cdot\biggr(\bar{M}^2_{\zeta}(t)\bar{B}^{ij}_{n}(t)-\bar{M}^2_{\zeta}(s)\bar{B}^{ij}_{n}(s)+{\sqrt{2}\int^t_s\int_{\mathbb{T}^3}\zeta(\nabla_\sym\bar{u})_{ij}f_n\,\d x\,\d r}\biggr)\biggr)=0.
		\end{align*}
		As above, this shows that for every $n\in\mathbb{Z}^3$ and $1\leq i,j\leq 3$,
		$$
		\bar{M}^2_{\zeta}(t)\bar{B}^{ij}_{n}(t)+{\sqrt{2}}\int^t_0\int_{\mathbb{T}^3}\zeta (\nabla_\sym\bar{u})_{ij}f_n\,\d x\,\d r
		$$
		is $\bar{\mathscr{G}}$-martingale and the same holds for
	 	\begin{align*}
			(\bar{M}^2_{\zeta}(t))^2	-&{2}\sum_{n\in\mathbb{Z}^3}\sum_{i,j=1}^3\int^t_0\bigg|\int_{\mathbb{T}^3}\zeta(\nabla_\sym\bar{u})_{ij}f_n\,\d x\bigg|^2\,\d r\\
			-&\sum_{n\in\mathbb{Z}^3}\sum_{i=1}^3\biggl[\int^t_0\bigg|\int_{\mathbb{T}^3}(\nabla\zeta)_i\bar{\psi}g_n\,\d x\bigg|^2\,\d r+\frac{1}{4}\int^t_0\bigg|\int_{\mathbb{T}^3}\zeta\bar{\psi}(\nabla g_n)_i \,\d x\bigg|^2\,\d r\biggr]\\
			-&\sum_{n\in\mathbb{Z}^3}\sum_{i=1}^3\int^t_0\bigg(\int_{\mathbb{T}^3}(\nabla\zeta)_i\bar{\psi}g_n\,\d x\bigg)\bigg(\int_{\mathbb{T}^3}\zeta\bar{\psi}(\nabla g_n)_i\,\d x\bigg)\,\d r. 
		\end{align*}   By a direct computation of the quadratic variation, we arrive this time at
		\begin{align}\label{martingale-2}
			\bar{\mathbb{E}}\biggl[\bar{M}^2_{\zeta}(t)-\sum_{n\in\mathbb{Z}^3}{ \int_{\mathbb{T}^3}}\int^t_0\nabla\zeta\cdot\bar{\psi}g_n\,\d \bar{W}_{n}\,\d x-&\frac{1}{2}\sum_{n\in\mathbb{Z}^3}{ \int_{\mathbb{T}^3}}\int^t_0\zeta\bar{\psi}\nabla\cdot(g_n\,\d \bar{W}_{n})\,\d x\notag \\
			+&\frac{1}{\sqrt{2}}\sum_{n\in\mathbb{Z}^3}{\int_{\mathbb{T}^3}}\int^t_0\zeta(\nabla\bar{u}:f_n\,\d \bar{B}_{n})\,\d x\biggr]^2 = 0, \qquad \text{for every } t \in [0,T].
		\end{align}
		We finally record that \eqref{martingale-1} and \eqref{martingale-2}  clearly imply   that
		\begin{align}\label{eqn_M_1_vp_SI}
			\bar{M}^1_{\varphi}(t)=\frac{1}{\sqrt{2}}\sum_{n\in\mathbb{Z}^3}\int_{\mathbb{T}^{3}}\int^t_0\nabla\varphi:\bar{\psi}f_n\,\d \bar{B}_n \,\d x,
		\end{align}
		and 
		\begin{align}\label{eqn_M_2_zeta_SI}
			\bar{M}^2_{\zeta}(t)=&\sum_{n\in\mathbb{Z}^3}\int_{\mathbb{T}^3}\int^t_0\nabla\zeta\cdot\bar{\psi}g_n\,\d \bar{W}_{n}\,\d x+\frac{1}{2}\sum_{n\in\mathbb{Z}^3}\int_{\mathbb{T}^3}\int^t_0\zeta\bar{\psi}\nabla\cdot(g_n\,\d \bar{W}_{n})\,\d x\\& -\frac{1}{\sqrt{2}}\sum_{n\in\mathbb{Z}^3}\int_{\mathbb{T}^3}\int^t_0\zeta(\nabla\bar{u}:f_n\,\d \bar{B}_{n})\,\d x, 	
		\end{align}
		for all $t\in [0,T]$, $\P$-almost surely.
		
		 \textit{Step 6 (Relation of the dominating measure).}
		Here we aim to show that, $\bar{\P}$-a.s.,
		\begin{equation}\label{eqn37}
		\bar{q} \,\ge\, \frac{1}{\bar{\psi}} \Big(|\nabla\bar{\psi}|^2 + |\nabla_{\sym} \bar{u}|^2+\epsilon\Big),
	\end{equation}
	as measures on $[0,T]\times \T^3$.  Since we use Convention \ref{conv_inverse} that $1/\bar{\psi} =\infty$ on $\{\bar{\psi}\le 0\}$, the above implies in particular $\bar{\psi}>0$, $\P\otimes \d t \otimes \d x$-almost everywhere.  To this end, we let $\zeta\in C(\mathbb{T}^3\times[0,T])$ be non-negative and claim that for any fixed $\kappa>0$, we have
		\begin{align}\label{weakconvergence-1}
			\zeta^{1/2} h_\kappa(\bar{\psi}_{\delta})^{1/2}\nabla_{\sym}\bar{u}_{\delta}\rightharpoonup\zeta^{1/2} h_\kappa(\bar{\psi})^{1/2}\nabla_{\sym}\bar{u}, 	
		\end{align}
		and
		\begin{align}\label{weakconvergence-2}
			\zeta^{1/2} h_\kappa(\bar{\psi}_{\delta})^{1/2}\nabla\bar{\psi}_{\delta}\rightharpoonup\zeta^{1/2} h_\kappa(\bar{\psi})^{1/2}\nabla\bar{\psi}, 	
		\end{align}
		as $\delta\rightarrow0$,  in $L^2_{t,x}$, $\bar{\P}$-almost surely.  Indeed, because
			$ h_\kappa(\bar{\psi}_{\delta})^{1/2} \to h_\kappa(\bar{\psi})^{1/2}$,
			 $\bar{\P} \otimes \d t \otimes\d x$-a.e., up to taking a further subsequence, we have for any  test function $\chi \in L^2([0,T]\times \T^3;\R^3)$ that
			 \[ 
			 \int_0^T \int_{\T^3} 
			 \chi  \zeta^{1/2} h_\kappa(\bar{\psi}_{\delta})^{1/2}\nabla_{\sym}\bar{u}_{\delta}
			 \,\d x\, \d t\,\to \,  \int_0^T \int_{\T^3}
			 \chi   \zeta^{1/2} h_\kappa(\bar{\psi})^{1/2}\nabla_{\sym}\bar{u} \,\d x\, \d t,
			 \]
		by dominated convergence and the weak convergence of $\nabla_\sym \bar{u}_\delta$. This yields \eqref{weakconvergence-1} and a repetition of this argument also \eqref{weakconvergence-2}.  Thus, by lower semicontinuity of the $L^2_{t,x}$-norm with respect to weak convergence, we find 
		\begin{align}\int_{[0,T]\times \T^3} \zeta\, \d \bar{q}  \,&=\,\lim_{\delta \to 0}
		\int_{[0,T]\times \T^3} \zeta \,\d \bar{q_\delta} \,=\, \lim_{\delta \to 0}
		 \int_0^T \int_{\T^3}\zeta h_\delta(\bar{\psi}_{\delta})\bigl(
		|\nabla \bar{\psi}_{\delta}|^2 + |\nabla_\sym \bar{u}_{\delta}|^2 + \epsilon
		\bigr) \,\d x \,\d t\\&\ge\, \liminf_{\delta\to 0}
		 \int_0^T \int_{\T^3}\zeta h_\kappa(\bar{\psi}_{\delta})\bigl(
		|\nabla \bar{\psi}_{\delta}|^2 + |\nabla_\sym \bar{u}_{\delta}|^2 + \epsilon
		\bigr) \,\d x \,\d t \\& \ge\,
			 \int_0^T \int_{\T^3} \zeta h_\kappa(\bar{\psi})\bigl(
		|\nabla \bar{\psi}|^2 + |\nabla_\sym \bar{u}|^2 + \epsilon
		\bigr) \,\d x \,\d t ,
		\end{align}
		using \eqref{eqn3278474328742} and that $h_\delta \ge h_\kappa$ for small enough $\delta$, cf.\ Assumption \ref{Ass:nu_N}. It remains to take $\kappa \to 0$ using monotone convergence to deduce \eqref{eqn37}.
		
		 \textit{Step 7 (Conclusion).}
		We summarize how the preceding steps imply the assertion, and start by deriving the estimates \eqref{total-energy-es-es}--\eqref{eq_meas_bound_eps}. The latter follows readily by 
		\begin{align}\label{eqn_lower_fatou}\bar{\E}\bigl[
			\| \bar{q} \|_{ \M([0,T]\times \T^3) }\bigr] \,&\le\, \liminf_{\delta\to 0}
			\bar{\E}\bigl[
			\| \bar{q}_\delta \|_{ \M([0,T]\times \T^3) }\bigr]\,=\, \liminf_{\delta\to 0}
			{\E}\bigl[
			\| {q}_\delta \|_{ \M([0,T]\times \T^3) }\bigr]\\& \le\, 	\biggl(1 \,+\, \frac{4F_2 +G_2 }{8}T \biggr)  \biggl(
			\frac{1}{2}\int_{\T^3}|u_0|^2 +|\psi_0|^2 \,\d x\,+\,2 \epsilon T
			\biggr)^{1/2},
		\end{align}
		based on lower semicontinuity of the norm with respect to weak-${}^*$ convergence, Fatou's lemma and Lemma \ref{lem-L1-es}. For the energy estimate, we firstly observe that  \eqref{total-energy-es}  and \eqref{identity-jointlaws} imply that 
		\begin{align}
			 \frac{1}{2}\int_{\T^3}
			|\bar{u}_\delta(t)|^2 +\bar{\psi}_\delta^2(t) \,\d x 	+ \epsilon \int_0^t \int_{\T^3} 
			|\Delta \bar{u}_{\delta}|^2
			+ |\Delta\bar{\psi}_{\delta}|^2\,\d x\,\d s\,\le\, 
			\frac{1}{2}\int_{\T^3}
			|u_0|^2 +\psi_0^2 \,\d x\,+\,\epsilon t ,
		\end{align}
		$\bar{\P}\otimes \d t$-almost everywhere.
		Moreover, after passing to another subsequence in $\delta$, we have also $\bar{\P}\otimes \d t$-a.e.\ $\bar{u}_\delta \to \bar{u}$ in $L^2(\T^3;\R^3)$ and $\bar{\psi}_\delta \to \bar{\psi}$ in $L^2(\T^3)$. Thereby, it follows that, $\P$-a.s., 
		\begin{align}&
		\epsilon \int_0^t\int_{\T^3}  |\Delta \bar{u}|^2 + |\Delta \bar{\psi}|^2	\,\d x \,\d s \\& \quad \le\, 
		 \liminf_{\delta\to 0} \epsilon \int_0^t\int_{\T^3}  |\Delta \bar{u}_\delta|^2 + |\Delta \bar{\psi}_\delta|^2	\,\d x \,\d s 
		 \\&\quad \le \,\lim_{\delta\to 0} \biggl(
			\frac{1}{2}\int_{\T^3}
		|u_0|^2 +\psi_0^2 \,\d x \,-\, 	\frac{1}{2}\int_{\T^3}
		|\bar{u}_\delta(t)|^2 +\bar{\psi}_\delta^2(t) \,\d x\,+\,\epsilon t\biggr)
		\\& \quad=\, 
		\frac{1}{2}\int_{\T^3}
		|u_0|^2 +\psi_0^2 \,\d x \,-\, 	\frac{1}{2}\int_{\T^3}
		|\bar{u}_t|^2 +\bar{\psi}_t^2\,\d x\,+\,\epsilon t,
		\end{align}
	for almost all $t$. Rearranging yields \eqref{total-energy-es-es} and implies in particular the additional regularity required in Definition \ref{eqn354865865}~\eqref{Item_defi_sol_1_ep}. Part \eqref{Item_defi_sol_2_ep} was shown in \eqref{eqn37}. Assertion \eqref{Item_defi_sol_3_ep} follows by combining \eqref{eqn33_lim} and \eqref{eqn_M_1_vp_SI}, whereas \eqref{Item_defi_sol_4_ep} is the result of \eqref{eqn34_lim} and \eqref{eqn_M_2_zeta_SI},  finishing this proof.
	\end{proof}

	\subsection{Construction of weak solutions to the original equation}\label{subsec-4-3}
	The aim of this subsection is to complete the proof of Theorem \ref{thm-weak-existence}, i.e., to show the existence of globally admissible weak martingale solutions   to 
	\begin{align}\begin{split}\label{original-eq}	\d u\,=\,&\nabla\cdot(\nabla_{\sym}u) \,\d t\,-\, \Pi \nabla \cdot (u \otimes u)\,\d t\,+\, \frac{F_1}{4}\Delta u\,\d t\, -\, \frac{1}{\sqrt{2}}\sum_{n \in \Z^3}\Pi \nabla\cdot(\psi f_n   \,\d B_n),\qquad \nabla\cdot u=0 ,
			\\
			\d \psi\,  =\, &(1+F_1{/2} +G_1/2)\Delta \psi \, \d t + \,\frac{1}{\psi }\bigl(
			|\nabla \psi|^2 \,+\, |\nabla_\sym u|^2 
			\bigr) \,\d t\,-\, \nabla \cdot ( u\psi )\,\d t \,-\, (F_2{/2}+ G_2/8) \psi \,\d t \\&- \sum_{n\in \Z^3} \nabla \cdot ( \psi g_n \,\d W_n ) \,
			+\, \frac{1}{2} \sum_{n \in \Z^3} \psi \nabla \cdot ( g_n \, \d  W_n) \,-\,  \frac{1}{\sqrt{2}}\sum_{n\in \Z^3}(\nabla  u : f_n\,\d  B_n),
			\\ (u(0),\psi(0))\, =\,&( u_0,\psi_0),
	\end{split}\end{align}
 as
	defined in Definition \ref{defi_solution}. 
	In the following, for $\epsilon\in (0,1)$, we let $(u_{\epsilon},\psi_{\epsilon})$ be the weak martingale solution to \eqref{ep_approx}--\eqref{eqn354865865} constructed in Proposition \ref{prp-existence-ep}, for which we aim to take the limit $\epsilon\to 0$. As in the preceding subsection, we may assume that these solutions are defined with respect to the same stochastic basis, see the comments below Definition \ref{defi_solution_ep}.

	\subsubsection{Uniform estimates and tightness}
	Here we collect estimates on the solutions to \eqref{ep_approx}--\eqref{eqn354865865}, which are uniform in $\epsilon$. While the already stated energy estimate \eqref{total-energy-es-es} and $L^1$-estimate \eqref{eq_meas_bound_eps} are uniform in $\epsilon$, the artificial dissipation vanishes when $\epsilon\to 0$. As a substitute, we record the following  interpolation estimate.
	\begin{proposition}Let 
		 $(u_{\epsilon} ,\psi_{\epsilon},q_\epsilon)$ be the weak martingale solutions to  \eqref{ep_approx}--\eqref{eqn354865865} with initial data $(u_0,\psi_0)$ constructed in Proposition \ref{prp-existence-ep}. Then there exists a constant $C= C(u_0,\psi_0, F_2,G_2,T)$, such that
		\begin{align}\begin{split}\label{sobolev-es}	\E \bigl[ \|u_\epsilon \|_{L^{7/3}([0,T]\times\T^3;\R^3)}^{7/3}\,+\,
				\| {\psi_\epsilon}\|_{L^{7/3}([0,T] \times \T^3)}^{7/3}
			\bigr]\,\le\, C,& \\\E\bigl[
		 \| u_\epsilon  \|_{L^{7/5}(0,T;W^{1,7/5}(\T^3;\R^3))}^{7/5}
				\,+\, 
			\|\psi_\epsilon \|_{L^{7/5}(0,T;W^{1,7/5}(\T^3))}^{7/5}\bigr]
			\,\le\, C,
			\end{split}
		\end{align}
	for any $\epsilon\in (0,1)$.
	\end{proposition}
	\begin{proof}We start by noticing that \eqref{eq_meas_bound_eps} together with domination property of $q_\epsilon$, see Definition \ref{defi_solution_ep}~\eqref{Item_defi_sol_1_ep} implies that
		\begin{equation}\label{eqn_here}
			\E \biggl[
			\int_0^T \int_{\T^3} \frac{1}{\psi_{\epsilon}}(|\nabla \psi_{\epsilon}|^2 + |\nabla_\sym u_{\epsilon}|^2+\epsilon) \,\d x \, \d t
			\biggr]\,\le\, C,
		\end{equation}
	for $C= C(u_0,\psi_0, F_2,G_2,T)$, which we allow to vary from line to line. The above implies that $\psi_\epsilon>0 $, $\P \otimes \d t \otimes \d x$-almost everywhere, cf.\ Convention \ref{conv_inverse}, and thereby setting 
	\begin{align}
		R_\delta (r) \,=\, \frac{1}{2}\int_0^r \sqrt{h_\delta}(s) \,\d s ,
	\end{align}
 we find that $R_\delta(\psi_\epsilon) \to \sqrt{\psi_\epsilon}$ in $L^2([0,T]\times \T^3)$, $\P$-a.s., as $\delta \to 0$. On the other hand, $\nabla (R_\delta(\psi_\epsilon)) = \sqrt{h_\delta(\psi_\epsilon)} \nabla \psi_\epsilon /2$ is $\P$-a.s.\ uniformly in $ \delta$ bounded in $L^2([0,T]\times \T^3)$, by \eqref{eqn_here}. By extracting a weakly convergent subsequence, we find that $\nabla \sqrt{\psi_\epsilon} = \nabla \psi_\epsilon /(2\sqrt{\psi_\epsilon})$ as functions  in $L^2([0,T]\times \T^3)$ with
 \begin{align}
 	\E \bigl[
 	\| \sqrt{\psi_\epsilon}\|_{L^2(0,T;L^6(\T^3))}^2
 	\bigr]\,\lesssim \, 
 	\E \bigl[
 	\| \sqrt{\psi_\epsilon}\|_{L^2(0,T;H^1(\T^3))}^2
 	\bigr]\,\le\, C,
 \end{align}
where we also invoked the Sobolev embedding and the total energy bound \eqref{total-energy-es-es}. The latter  entails at the same time 
 \[
 \| \sqrt{\psi_\epsilon}\|_{L^\infty(0,T;L^4(\T^3))} \,\le\, C,
 \]
 so that H\"older's inequality yields 
 \begin{align}
 	\E [  \| \sqrt{\psi_\epsilon}\|_{L^{14/3}([0,T] \times \T^3)}^{14/3} ] \,\le\, 
 	\E \bigl[\bigl( \| \sqrt{\psi_\epsilon}\|_{L^2(0,T;L^6(\T^3))}^{3/7}
 	\| \sqrt{\psi_\epsilon}\|_{L^\infty(0,T;L^4(\T^3))}^{4/7}\bigr)^{14/3}
 	\bigr]\,\le\, C,
 \end{align}
which can then be rephrased as 
\begin{align}
	\E \bigl[
	 \| {\psi_\epsilon}\|_{L^{7/3}([0,T] \times \T^3)}^{7/3}
	\bigr]\,\le\, C.
\end{align}
From here, we proceed as Remark \ref{rem_weak_sol}~\eqref{rem_weak_sol_I2} and use H\"older's inequality,  together with \eqref{eqn_here} to estimate
\begin{align}
	\E \biggl[
	\int_0^T \int_{\T^3}|\nabla \psi_{\epsilon}|^{7/5}  \,\d x \, \d t
	\biggr]\,&=\, 
		\E \biggl[
	\int_0^T \int_{\T^3}\psi_\epsilon^{7/10} \psi_\epsilon^{-7/10} |\nabla \psi_{\epsilon}|^{7/5}  \,\d x \, \d t
	\biggr]
	\\&  \le \,
	\E \bigl[
	\| {\psi_\epsilon}\|_{L^{7/3}([0,T] \times \T^3)}^{7/3}
	\bigr]^{3/10}	\E \biggl[
		\int_0^T \int_{\T^3} \frac{1}{\psi_{\epsilon}}|\nabla \psi_{\epsilon}|^2  \,\d x \, \d t
		\biggr]^{7/10}\,\le\, C,
\end{align}
and analogously
\begin{align}
		\E \biggl[
	\int_0^T \int_{\T^3}|\nabla_\sym u_{\epsilon}|^{7/5}  \,\d x \, \d t
	\biggr]\,\le\, C
\end{align}
To leverage the latter, we use that $u_{\epsilon}$ is divergence free, so that $W^{1,p}_x$-estimates for the elliptic equation
$$
\frac{1}{2}\Delta u_{\epsilon} = \nabla \cdot (\nabla_{\sym} u_{\epsilon}),
$$
yield also
\begin{align}
	\E \bigl[
	\|u_\epsilon\|_{L^{7/5}(0,T;W^{1,7/5}(\T^3;\R^3))}^{7/5}
	\bigr]\,\le\, C
\end{align}
Finally, by the Sobolev embedding we estimate 
	\[
\E \bigl[  \| u_\epsilon  \|_{L^{7/5}(0,T; {L}^{21/8}(\T^3 ;\R^3)) }^{7/5} \bigr] \,\le\, C,
\]
and using H\"older's inequality that 
\begin{equation}
	\E \bigl[ \|u_\epsilon \|_{L^{7/3}([0,T]\times\T^3;\R^3)}^{7/3}\bigr] \,\le\, \E\bigl[ \bigl(\| u_\epsilon  \|_{L^\infty(0,T;L^2(\T^3;\R^3))}^{2/5}\| u_\epsilon   \|_{L^{7/5}(0,T; {L}^{21/8}(\T^3 ;\R^3)) }^{3/5} \bigr)^{7/3}\bigr]\,\le\, C,
\end{equation}
in light of \eqref{total-energy-es-es}. Collecting all the bounds results in \eqref{sobolev-es}
  \end{proof}

	Since the time regularity estimates from Lemma \ref{lem-timeregularity} are uniform in $\epsilon$, they transfer to the $\epsilon$-regularization.
	
		\begin{lemma}Let 
			$(u_{\epsilon} ,\psi_{\epsilon},q_\epsilon)$ be the weak martingale solutions to  \eqref{ep_approx}--\eqref{eqn354865865} with initial data $(u_0,\psi_0)$ constructed in Proposition \ref{prp-existence-ep}. Then, for every $\alpha \in (0,1/2)$ and $l > 5/2$, there exists a constant \[C = C(\alpha,l, u_0,\psi_0, F_1,F_2, G_1,G_2,T) <\infty, \] such that
		\begin{align}\label{timeregularity-es}
			\mathbb{E}\|u_{\epsilon}\|_{C^{\alpha}([0,T];H^{-l}(\mathbb{T}^3;\R^3))}+\mathbb{E}\|\psi_{\epsilon}\|_{W^{\alpha,2}([0,T];H^{-l}(\mathbb{T}^3))} \leq C,
		\end{align}
			for all $\epsilon\in (0,1)$.
	\end{lemma}
	\begin{proof}
		This follows from Fatou's lemma and the lower semicontinuity of the involved norms with respect to $L^2_{t,x}$-convergence, analogous to \eqref{eqn_lower_fatou}.
	\end{proof}
	
	 Combining  the above estimates, we have the following tightness result.

	\begin{lemma}\label{tightness-2}
		 Let  
		 $(u_{\epsilon} ,\psi_{\epsilon},q_\epsilon)$ be the weak martingale solutions to  \eqref{ep_approx}--\eqref{eqn354865865} with initial data $(u_0,\psi_0)$ constructed in Proposition \ref{prp-existence-ep}. Then for every $\alpha\in(0,1/2)$ and $l>5/2$, the family
		  \[(u_{\epsilon},\psi_{\epsilon},\nabla u_{\epsilon},\nabla\psi_{\epsilon})_{\epsilon\in(0,1)}\] is tight in 
		\begin{align*}
		L^2([0,T];L^2(\mathbb{T}^3 ;\R^3))\cap &C^{\alpha}([0,T];H^{-l}(\T^3;\R^3))\times L^2([0,T];L^2(\mathbb{T}^3))\\
		&\times(L^{7/5}([0,T]\times \mathbb{T}^3 ;\R^{3\times 3}),w)\times(L^{7/5}([0,T]\times\mathbb{T}^3 ;\R^3),w). 
		\end{align*}
		
	\end{lemma}
	\begin{proof}
		With the help of the uniform estimates \eqref{sobolev-es}, $(\nabla u_{\epsilon},\nabla\psi_{\epsilon})_{\epsilon\in(0,1)}$ is seen to be tight in $$(L^{7/5}([0,T];L^{7/5}(\mathbb{T}^3;\R^{3\times3})),w)\times(L^{7/5}([0,T];L^{7/5}(\mathbb{T}^3;\R^3)),w).$$
		Combining the bounds \eqref{sobolev-es} and \eqref{timeregularity-es} with the Aubin--Lions--Simon compactness criterion, we find also that $(u_{\epsilon},\psi_{\epsilon})_{\epsilon\in(0,1)}$ is tight in 
		$$
		L^{7/5}([0,T];L^{2}(\mathbb{T}^3;\R^3))\times L^{7/5}([0,T];L^{2}(\mathbb{T}^3)).  
		$$
		Combining this with the bound
		$$\operatorname*{ess\,sup}_{t\in[0,T]}\,(\|u_t\|_{L^2(\mathbb{T}^3;\R^3)}^2+\|\psi_t\|_{L^2(\mathbb{T}^3)}^2)\leq C(u_0,\psi_0).$$
		from \eqref{total-energy-es-es}, and setting 
		\[
		Y_0 \,=\, L^{7/5}([0,T];L^{2}(\mathbb{T}^3;\R^3)),\qquad Y_{3/10} \,=\, L^2([0,T]\times \T^3;\R^3),\qquad 
		Y_1 \,=\, L^{\infty}([0,T];L^{2}(\mathbb{T}^3;\R^3)),
		\]
		for $(u_\epsilon)_{\epsilon\in (0,1)}$, we obtain that $u_\epsilon$ is also tight in $ Y_{3/10}$ by Lemma \ref{lemma_app_c}. Analogously, we conclude that $(\psi_\epsilon)_{\epsilon\in (0,1)}$ is tight in $L^2([0,T]\times \T^3)$, while tightness of $(u_\epsilon)_{\epsilon\in (0,1)}$ in $C^\alpha([0,T];H^{-l} (\T^3;\R^3))$ is a consequence of \eqref{timeregularity-es}, up to increasing $\alpha$ and $-l$ slightly.
	\end{proof}
	
	Analogously to \eqref{eqn862346d}--\eqref{eqn862346e}, we let  
	\begin{align}\label{eqn_M1_eps}
		M^1_{\epsilon}=-\frac{1}{\sqrt{2}}\sum_{n\in\mathbb{Z}^3}\int^t_0\Pi\nabla\cdot(\psi_{\epsilon} f_n \,\d B_{n}),
	\end{align}
	and 
	\begin{align}\label{eqn_M2_eps}
		M^2_{\epsilon}=-\sum_{n\in\mathbb{Z}^3}\int^t_0\nabla\cdot(\psi_{\epsilon} g_n\,\d W_n)+\frac{1}{2}\sum_{n\in\mathbb{Z}^3}\int^t_0\psi_{\epsilon}\nabla\cdot(g_n\,\d W_n)-\frac{1}{\sqrt{2}}\sum_{n\in\mathbb{Z}^3}\int^t_0(\nabla u_{\epsilon}:f_n\,\d B_n). 
	\end{align}
	
	\begin{lemma}\label{lem-tight-martingales-2}
		Let 
		 $(u_{\epsilon} ,\psi_{\epsilon},q_\epsilon)$ be the weak martingale solutions to  \eqref{ep_approx}--\eqref{eqn354865865} with initial data $(u_0,\psi_0)$ constructed in Proposition \ref{prp-existence-ep}. Then,
		for   every  $\gamma\in(0,1/2)$ and every  $s>1$, $(M^1_{\epsilon})_{\epsilon\in(0,1)}$ and $(M^2_{\epsilon})_{\epsilon\in(0,1)}$ are tight in $C^{\gamma}([0,T];H^{-s}(\mathbb{T}^3;\mathbb{R}^3))$ and $C^{\gamma}([0,T];H^{-s}(\mathbb{T}^3))$, respectively. 
	\end{lemma}
	\begin{proof}
		The proof follows along the lines of Lemma \ref{lem-tight-martingales} based on the energy estimate \eqref{total-energy-es-es}. 
	\end{proof}
	\subsubsection{Limiting procedure: $\epsilon\to 0$.}  Finally, we take the limit $\epsilon \to 0$ by another application of the stochastic compactness method. The proof is very similar to the proof of Proposition \ref{prp-existence-ep}, so that we  omit many details. Indeed, in terms of the tightness derived in Lemma \ref{tightness-2}, the only difference compared to Lemma \ref{tightness-1}, is that the gradients are controlled in a weaker space, namely $L^{7/5}_{t,x}$ instead of $L^2_{t,x}$.
	
	\begin{proof}[Proof of Theorem \ref{thm-weak-existence}] As mentioned above, we can largely follow the proof of Proposition  \ref{prp-existence-ep} and assume also here that the weak martingale solutions $(u_\epsilon,\psi_\epsilon,q_\epsilon)$ are defined with respect to the same stochastic basis. Then, fixing   $\alpha\in(0,1/2)$ and $l>5/2$,  the sequence
		\begin{align*}
			(u_{\epsilon}, \psi_{\epsilon}, \nabla u_{\epsilon}, \nabla \psi_{\epsilon})
		\end{align*}
		is tight in
		\begin{align*}
			L^2([0,T]; L^2(\mathbb{T}^3;\R^3))\cap &C^{\alpha}([0,T];H^{-l}(\T^3;\R^3))\times L^2([0,T]; L^2(\mathbb{T}^3))\\ &\times(L^{7/5}([0,T];L^{7/5}(\mathbb{T}^3;\R^{3\times3})),w)\times(L^{7/5}([0,T];L^{7/5}(\mathbb{T}^3;\R^3)),w),
		\end{align*}
	by Lemma \ref{tightness-2}. In light of Lemma \ref{lem-tight-martingales-2} also $M_\epsilon^1$ and $M_\epsilon^2$ defined in \eqref{eqn_M1_eps}--\eqref{eqn_M2_eps}
		are tight in $$C^{\gamma}([0,T];H^{-s}(\mathbb{T}^3;\mathbb{R}^3)) \qquad \text{and}\qquad C^{\gamma}([0,T];H^{-s}(\mathbb{T}^3)),$$respectively, for $\gamma\in(0,1/2)$ and $s>1$, which we fix for this proof as well. 
		Finally, we denote
		$$X_{\epsilon} := \Big(u_{\epsilon}, \psi_{\epsilon}, \nabla u_{\epsilon}, \nabla \psi_{\epsilon},q_{\epsilon}, M^1_{\epsilon},M^2_{\epsilon}\Big),$$
		and let 
		\begin{align*}
			\mathbb{X}:=&L^2([0,T]; L^2(\mathbb{T}^3;\R^3))\cap C^{\alpha}([0,T];H^{-l}(\T^3;\R^3))\times L^2([0,T]; L^2(\mathbb{T}^3)) \times (L^{7/5}([0,T]; L^{7/5}(\mathbb{T}^3;\R^{3\times3})), w)\\
			&\times (L^{7/5}([0,T]; L^{7/5}(\mathbb{T}^3;\R^3)), w)\times(\mathcal{M}([0,T]\times\mathbb{T}^3),w^*)\times C^{\gamma}([0,T];H^{-s}(\mathbb{T}^3;\mathbb{R}^3))\times C^{\gamma}([0,T];H^{-s}(\mathbb{T}^3)).
		\end{align*}
		Then, we have tightness of the joint sequence
			$$
		({X}_{\epsilon}, ({B}_{n})_{n\in\mathbb{Z}^3},({W}_{n})_{n\in\mathbb{Z}^3})_{\epsilon\in(0,1)}
		$$
		in 
		$$
		\mathbb{X} \times C([0,T];\mathbb{R}^{3\times 3})^{\mathbb{Z}^3} \times C([0,T];\mathbb{R}^{3})^{\mathbb{Z}^3}.
		$$
		By the Skorokhod--Jakubowski  theorem \cite{Jak97}, there exists a new probability space 
		$(\bar{\Omega}, { \bar{\mathfrak{A}}}, \bar{\mathbb{P}})$ and a sequence of random variables
		$$
		(\bar{X}_{\epsilon}, (\bar{B}_{n,\epsilon})_{n\in\mathbb{Z}^3},(\bar{W}_{n,\epsilon})_{n\in\mathbb{Z}^3})\,\to\, (\bar{X}, (\bar{B}_{n})_{n\in\mathbb{Z}^3},(\bar{W}_{n})_{n\in\mathbb{Z}^3}) ,
		$$
		$\bar{\P}$-a.s., 
		in 
		$$
		\mathbb{X} \times C([0,T];\mathbb{R}^{3\times 3})^{\mathbb{Z}^3} \times C([0,T];\mathbb{R}^{3})^{\mathbb{Z}^3},
		$$
		such that the joint laws satisfy
		\begin{align}\label{identity-jointlaws-3}
			(\bar{X}_{\epsilon}, (\bar{B}_{n,\epsilon})_{n\in\mathbb{Z}^3},(\bar{W}_{n,\epsilon})_{n\in\mathbb{Z}^3}) \overset{d}{\,=\,}& (X_{\epsilon}, (B_{n})_{n\in\mathbb{Z}^3},(W_{n})_{n\in\mathbb{Z}^3}), 
		\end{align}
		for a subsequence $\epsilon\to 0$. As in the proof of Proposition \ref{prp-existence-ep}, we write  $\bar{X}_{\epsilon} = (\bar{u}_{\epsilon}, \bar{\psi}_{\epsilon},\nabla \bar{u}_{\epsilon},\nabla \bar{\psi}_{\epsilon},\bar{q}_{\epsilon},\bar{M}^1_{\epsilon},\bar{M}^2_{\epsilon})$ and $\bar{X} = (\bar{u}, \bar{\psi},\nabla \bar{u},\nabla \bar{\psi},\bar{q},\bar{M}^1,\bar{M}^2)$, where we used already that one can identify the weak gradients.
		
		From here onward, we can follow the individual steps of the aforementioned proof. Transferring the relations of the components of $X_\epsilon$ to  the new variables $\bar{X}_{\epsilon} $, as performed in the second step is entirely analogous. For the third step some comments are in order: We need to show the convergence of 
		\begin{align}\bar{M}_{\varphi,\epsilon}^1(t) \,=\, &\int_{\T^3}   \bar{u}_{\epsilon}(t)\cdot \vp\,\d x \,-\, 
			\int_{\T^3} u_0 \cdot \vp\,\d x+\epsilon\int_0^t\int_{\T^3}\bar{u}_{\epsilon}\cdot\Delta^2\vp\,\d x\, \d s
			\\& +\int_0^t \int_{\T^3}   \Bigl(\frac{1}{2}+ \frac{F_1}{4}\Bigr) \nabla \bar{u}_{\epsilon}: \nabla  \vp \,-\, \bar{u}_{\epsilon}\otimes \bar{u}_{\delta}:\nabla \vp \,\d x\,\d s
		\end{align}
		to 
		\begin{align}\begin{split}&\int_{\T^3}  \bar{u}(t)\cdot \vp\,\d x \,-\, 
				\int_{\T^3} u_0 \cdot \vp\,\d x  +\int_0^t \int_{\T^3}   \Bigl(\frac{1}{2}+ \frac{F_1}{4}\Bigr) \nabla \bar{u}: \nabla  \vp \,-\, \bar{u}\otimes \bar{u}:\nabla \vp \,\d x\d s, \end{split}
		\end{align}
	as $\epsilon\to 0$,
		and a similar statement involving the equation for $ 
		\bar{M}_{\zeta,\epsilon}^2$, defined as in \eqref{eqn737232}. The passage in the convective term transfers to the current setting, since it relies on the $L^2_{t,x}$-convergence of $\bar{u}_\epsilon \to \bar{u}$, $\P$-almost surely. That \[
		\epsilon\int_0^t\int_{\T^3}\bar{u}_{\epsilon}\cdot\Delta^2\vp\,\d x \,\d s \,\to\, 0,
		\]
		follows from the uniform bound \eqref{total-energy-es-es}, and for the equation involving $ 
		\bar{M}_{\zeta,\epsilon}^2$ one can argue in the same manner.
		 The fourth step moreover, which only involves the Brownian motions, is virtually identical, and  also the fifth step carries over due to the  energy estimate \eqref{total-energy-es-es}. 
		 
		 The only significant change is in  step six, because this time the convergences $\nabla \bar{u}_\epsilon \to \bar{u}$ and $\nabla \bar{\psi}_\epsilon \to \bar{\psi}$ take place in the weak topology of  $L^{7/5}_{t,x}$ instead of $L^2_{t,x}$. Moreover, we  aim here for the relation  that, $\bar{\P}$-a.s.,
			\begin{equation}\label{eqn37-new}
				\bar{q} \,\ge\, \frac{1}{\bar{\psi}} \Big(|\nabla\bar{\psi}|^2 + |\nabla_{\sym} \bar{u}|^2\Big).
			\end{equation}
		Let   $\zeta\in C(\mathbb{T}^3\times[0,T])$ be non-negative. We want to argue again that, for fixed $\kappa>0$,
		\begin{align}\label{weakconvergence-1-new}
			\zeta^{1/2} h_\kappa(\bar{\psi}_{\epsilon})^{1/2}\nabla_{\sym}\bar{u}_{\epsilon}\rightharpoonup\zeta^{1/2} h_\kappa(\bar{\psi})^{1/2}\nabla_{\sym}\bar{u}, 	
		\end{align}
		and
		\begin{align}\label{weakconvergence-2-new}
			\zeta^{1/2} h_\kappa(\bar{\psi}_{\epsilon})^{1/2}\nabla\bar{\psi}_{\epsilon}\rightharpoonup\zeta^{1/2} h_\kappa(\bar{\psi})^{1/2}\nabla\bar{\psi}, 	
		\end{align}
		as $\epsilon\rightarrow0$,  in $L^2_{t,x}$, $\bar{\P}$-a.s., cf.\ \eqref{weakconvergence-1}--\eqref{weakconvergence-2}. For this  we use the estimate 
		\begin{align}
			\int_0^T \int_{\T^3} \frac{1}{\bar{\psi}_\epsilon}\bigl(
			|\nabla \bar{\psi}_\epsilon|^2 +
			|\nabla_\sym \bar{u}_\epsilon|^2+\epsilon
			\bigr)\,\d x \,\d t \,\le\, \| \bar{q}_\epsilon\|_{\M([0,T]\times\T^3)},
		\end{align}
		and that the latter is uniformly in $\epsilon$ bounded, $\P$-a.s., by the weak-${}^*$ convergence $\bar{q}_\epsilon \to \bar{q}$. Therefore, using that $h_\kappa (r)\le 1/r$, cf.\ Assumption \ref{Ass:nu_N}, up to passing to another subsequence (depending on $\omega$), we have the convergence
			\begin{align*}
			\zeta^{1/2} h_\kappa(\bar{\psi}_{\epsilon})^{1/2}\nabla_{\sym}\bar{u}_{\epsilon}\,\rightharpoonup\,  h,
		\end{align*}
		weakly to some  $h  \in L^2([0,T]\times \T^3; \R^{3\times 3})$. Taking a test function $g\in C(\mathbb{T}^3\times[0,T];\mathbb{R}^{3\times 3})$,  we find
		\begin{align*}
		\int^T_0\int_{\mathbb{T}^3}h:g\,\d x\,\d t=&\lim_{\epsilon\rightarrow0}\int^T_0\int_{\mathbb{T}^3}\zeta^{1/2} h_\kappa(\bar{\psi}_{\epsilon})^{1/2}\nabla_{\sym}\bar{u}_{\epsilon}:g\,\d x\,\d t\\
			=&\lim_{\epsilon\rightarrow0}\int^T_0\int_{\mathbb{T}^3}\zeta^{1/2}\left( h_\kappa(\bar{\psi}_{\epsilon})^{1/2}- h_\kappa(\bar{\psi})^{1/2}\right)\nabla_{\sym}\bar{u}_{\epsilon}:g\,\d x\,\d t\\
			&+\lim_{\epsilon\rightarrow0}\int^T_0\int_{\mathbb{T}^3}\zeta^{1/2} h_\kappa(\bar{\psi})^{1/2}\nabla_{\sym}\bar{u}_{\epsilon}:g\,\d x\,\d t\\
			=&\lim_{\epsilon\rightarrow0}\int^T_0\int_{\mathbb{T}^3}\zeta^{1/2}\left( h_\kappa(\bar{\psi}_{\epsilon})^{1/2}- h_\kappa(\bar{\psi})^{1/2}\right)\nabla_{\sym}\bar{u}_{\epsilon}:g\,\d x\,\d t\\
			&+\int^T_0\int_{\mathbb{T}^3}\zeta^{1/2} h_\kappa(\bar{\psi})^{1/2}\nabla_{\sym}\bar{u}:g\,\d x\,\d t,
		\end{align*}
		since $\nabla_\sym \bar{u}_\epsilon \rightharpoonup \nabla_\sym \bar{u}$ in $L^{7/5}([0,T]\times \T^3;\R^{3\times 3})$, $\P$-almost surely. For the former term on the above right-hand side, we observe additionally that $\zeta^{1/2}\left( h_\kappa(\bar{\psi}_{\epsilon})^{1/2}- h_\kappa(\bar{\psi})^{1/2}\right)g \to 0$ in $L^{7/2}([0,T]\times \T^3; \R^{3\times 3})$, by dominated convergence, identifying 
		\[
		h\,=\, \zeta^{1/2} h_\kappa(\bar{\psi})^{1/2}\nabla_{\sym}\bar{u}.
		\] Because this can performed starting from any subsequence, the weak convergence \eqref{weakconvergence-1-new} follows, while  \eqref{weakconvergence-2-new} can be obtained analogously. Then, we can use  the lower semicontinuity of the norm with respect to weak convergence, to argue  that
		\begin{align}\int_{[0,T]\times \T^3} \zeta \,\d \bar{q}  \,&=\,\lim_{\epsilon \to 0}
			\int_{[0,T]\times \T^3} \zeta\, \d \bar{q}_\epsilon \,\ge \, \liminf_{\epsilon \to 0}
		 \int_0^T \int_{\T^3} \zeta \frac{1}{\bar{\psi}_{\epsilon}} \bigl(
			|\nabla \bar{\psi}_{\epsilon}|^2 + |\nabla_\sym \bar{u}_{\epsilon}|^2 + \epsilon
			\bigr) \,\d x \,\d t\\&\ge\, \liminf_{\epsilon\to 0}
		 \int_0^T \int_{\T^3} \zeta h_\kappa(\bar{\psi}_{\epsilon})\bigl(
			|\nabla \bar{\psi}_{\epsilon}|^2 + |\nabla_\sym \bar{u}_{\epsilon}|^2
			\bigr) \,\d x \,\d t \\& \ge\,
			 \int_0^T \int_{\T^3}\zeta h_\kappa(\bar{\psi})\bigl(
			|\nabla \bar{\psi}|^2 + |\nabla_\sym \bar{u}|^2
			\bigr) \,\d x \,\d t .
		\end{align}
		Sending $\kappa \to 0$ yields \eqref{eqn37-new} by the monotone convergence theorem.
		This completes step six, and the concluding seventh step  is again entirely analogous.
	\end{proof}

\section{Local in time existence of strong solutions}\label{Sec_strong_ex}
In this section, we prove the local in time existence of strong solutions to \eqref{Eq_new_var} for regular initial data and noise coefficients. More precisely, we show the following. 
\begin{theorem}[Short-time strong existence]\label{Prop_ex_strong_solution}
	Let $U_0 \in H^3(\T^3; \R^3)$ be divergence-free,  $\Psi_0\in H^3(\T^3)$ satisfy $\inf_{x\in \T^3} \Psi_0(x)>0$ and the noise coefficients from Assumption \ref{Ass:noise} satisfy $f,g\in C^4(\T^3;\ell^2(\Z^3))$.  Then there exists a local strong solution to \eqref{Eq_new_var} starting from $(U_0,\Psi_0)$, up to some $\P$-a.s.\ positive stopping time $\tau$.
\end{theorem}
We remark that, since Lemma \ref{prop_cons_of_energy} implies that strong solutions constitute in particular globally admissible weak solutions, uniqueness of strong solutions follows from the weak-strong uniqueness principle proved in Section \ref{Sec:WS_uniqueness}. Thereby, the strong solution from Theorem \ref{Prop_ex_strong_solution} is guaranteed to be unique up to the stopping time $\tau$. We also remind ourselves that strong solutions are characterized via their formulation in terms of $U$ and the temperature $\Theta = \Psi^2/2$, as is the content of Proposition \ref{prop_prop_SL}. As a consequence, the above yields also the short time existence and uniqueness of strong solutions to the system \eqref{Eq_intro}.
Regarding the proof of Theorem \ref{Prop_ex_strong_solution}, we provide in the preceding  subsection  elementary proofs of a composition estimate in Sobolev spaces and a general truncation lemma. With this at hand, we show  Theorem \ref{Prop_ex_strong_solution} by  setting up a fixed-point argument in sufficiently regular spaces for a suitably modified version of \eqref{Eq_new_var}, in  Subsection \ref{SS:proof_strong_existence}.

\subsection{Two elementary ingredients}
We prepare the proof of Theorem \ref{Prop_ex_strong_solution}.
\begin{lemma}[Composition in $H^2(\T^3)$]\label{lemma:comp}
	Let $h \in  W^{3,\infty}(\R)$ with $h(0) = 0$, then it holds that 
	\begin{align}
		\| h (\Phi )\|_{H^2(\T^3)} \,&\lesssim \, 
		\| h'\|_{W^{1,\infty}(\R)} \bigl(1+ \|\Phi \|_{H^2(\T^3)}^2  \bigr),
		\label{Eq:bddness_nu}
		\\
		\| h(\Psi ) -h(\Phi)\| _{H^2(\T^3)} \,&\lesssim \,
		\| h'\|_{W^{2,\infty}(\R)} \bigl(1+ \|\Psi \|_{H^2(\T^3)}^2 + \|\Phi \|_{H^2(\T^3)}^2  \bigr) \|\Psi -\Phi \|_{H^2(\T^3)},
		\label{Eq:diff_nu}
	\end{align}
	for all $\Psi ,\Phi \in H^2(\T^3)$.
\end{lemma}
\begin{proof}We start by checking the more delicate \eqref{Eq:diff_nu}. To this end, using that $(1-\Delta)\colon H^2(\T^3) \to L^2(\T^3)$ is an isomorphism, it suffices to estimate
	\begin{align}\label{Eq:nu_est_pf1}
		\| h(\Psi ) -h(\Phi)\| _{H^2(\T^3)}  \,\eqsim \, 	\| h(\Psi ) -h(\Phi)\| _{L^2(\T^3)} \,+\, \|\Delta (h(\Psi ) - h(\Phi))\| _{L^2(\T^3)}.
	\end{align}
	Concerning the former part, using the Lipschitz continuity of $h$, we find 
	\begin{equation}\label{Eq:nu_est_pf2}
		\| h(\Psi ) -h(\Phi)\| _{L^2(\T^3)} \,\le\,  \|h'\|_{L^\infty(\R)}\|\Psi - \Phi \|_{L^2(\T^3)}.
	\end{equation}
	Observing moreover that $ \Delta  (h(\Psi)) = \nabla\cdot ( h'(\Psi ) \nabla \Psi ) = h''(\Psi) |\nabla \Psi|^2 + h'(\Psi) \Delta \Psi$, we can estimate the latter term of \eqref{Eq:nu_est_pf1} by 
	\begin{align}\begin{split}\label{Eq:nu_est_pf3}&
			\|h''(\Psi) |\nabla \Psi|^2 - h''(\Phi) |\nabla \Phi|^2\|_{L^2(\T^3)} \,+\, 
			\| h'(\Psi) \Delta \Psi -  h'(\Phi) \Delta \Phi\|_{L^2(\T^3)}  \\&\quad \le \|(h''(\Psi ) -h''(\Phi) )|\nabla \Psi |^2\|_{L^2(\T^3)} \,+\, \| h''(\Phi ) (\nabla \Psi +\nabla \Phi)\cdot  (\nabla \Psi - \nabla \Phi)\|_{L^2(\T^3)}\\&
			\qquad +\, \|(h'(\Psi ) -h'(\Phi) ) \Delta  \Psi \|_{L^2(\T^3)} \,+\, \| h'(\Phi ) (\Delta \Psi - \Delta \Phi)\|_{L^2(\T^3)}\\& \quad\le \, \| h'''\|_{L^\infty(\R)}\| \nabla \Psi \|_{L^4(\T^3;\R^3)}^2 \| \Psi - \Phi\|_{L^\infty(\T^3)} \\&\qquad +\, \| h''\|_{L^\infty(\R)} \bigl(
			\|\nabla \Psi  \|_{L^4(\T^3;\R^3)} + \|\nabla \Phi  \|_{L^4(\T^3;\R^3)} 
			\bigr)\| \nabla \Phi - \nabla \Psi \|_{L^4(\T^3; \R^3)}\\&
			\qquad +\, \|h''\|_{L^\infty(\R)} \|\Delta\Psi  \|_{L^2(\T^3)} \| \Psi -\Phi \|_{L^\infty(\T^3)} \,+\, \|h'\|_{L^\infty (\R)} \|\Delta \Psi -\Delta \Phi \|_{L^2(\T^3)}.\end{split}
	\end{align}
	Using the chain of Sobolev embeddings $H^2(\T^3) \hookrightarrow W^{1,4}(\T^3) \hookrightarrow L^\infty(\T^3)$, we arrive at \eqref{Eq:diff_nu} upon combining \eqref{Eq:nu_est_pf1}--\eqref{Eq:nu_est_pf3}\noeqref{Eq:nu_est_pf2}. To obtain \eqref{Eq:bddness_nu}, it suffices to set $\Psi =0$ in the above and use that therefore $h(\Psi)=0$ by assumption.
\end{proof}
\begin{lemma}[Lipschitz truncation]\label{Lemma:trunc}Consider a normed vector space $\mathscr{X}$ and let $R>0$. Then the mapping 
	\begin{equation}\label{Eq:definition_trunc}
		\Xi^R_\mathscr{X}\colon \mathscr{X}\to \mathscr{X},
		x \,\mapsto \, \frac{Rx}{\max\{R, \|x\|_\mathscr{X}\}}
	\end{equation}
	satisfies 
	\begin{align}&
		\label{Eq:trunc}
		\|
		\Xi^R_\mathscr{X}(x)\|_{\mathscr{X}} \,\le\, R,\\&
		\label{Eq:Lip}\| \Xi^R_\mathscr{X}(x)- \Xi^R_\mathscr{X}(y)\|_\mathscr{X}\,\le\, 2 \|x-y\|_{\mathscr{X}},
	\end{align}
	for all $x,y\in\mathscr{X}$.
\end{lemma}
\begin{proof}The bound \eqref{Eq:trunc} follows by design, while for \eqref{Eq:Lip}  we can assume that $\|y\|_{\mathscr{X}}\ge \|x\|_{\mathscr{X}}$ and estimate 
	\begin{align}
		\|\Xi^R_\mathscr{X}(x)  - \Xi^R_\mathscr{X}(y) \|_\mathscr{X} \,\le\, \Bigl(
		\frac{R}{\max\{R, \|x\|_\mathscr{X}\}}- \frac{R}{\max\{R, \|y\|_\mathscr{X}\}}
		\Bigr) \|x\|_\mathscr{X}\,+\, \frac{R}{\max\{R, \|y\|_\mathscr{X}\}}\|x-y\|_{\mathscr{X}},
	\end{align}
	using the triangle inequality. We further bound
	\begin{align}
		&\Bigl(
		\frac{R}{\max\{R, \|x\|_\mathscr{X}\}} - \frac{R}{\max\{R, \|y\|_\mathscr{X}\}}
		\Bigr) \|x\|_\mathscr{X}
		\\&\quad \le \, \bigl(\min \{R, \|x\|_{\mathscr{X}} \}\, - \,\min \{
		R,\|y\|_{\mathscr{X}}
		\} \bigr)\,+\, \frac{R}{\max\{R,\|y\|_\mathscr{X}\}}\bigl(
		\|y\|_\mathscr{X}-\|x\|_\mathscr{X}
		\bigr),
	\end{align}
	and since the  former term on the right-hand side is non-positive, we deduce \eqref{Eq:Lip}.
\end{proof}
\subsection{Proof of strong existence}\label{SS:proof_strong_existence} 
Before starting with the detailed proof of Theorem \ref{Prop_ex_strong_solution}, let us outline the idea: One writes the system \eqref{Eq_new_var}, posed for $(U,\Psi)$ and for small times, as 
\begin{align}\begin{split}\label{eq_linear}
		\d U \,=&\,\nabla\cdot(\nabla_{\sym} U)\, \d t\, + \,\frac{F_1}{4}\Delta  U\,\d t \,-\, \frac{1}{\sqrt{2}}\sum_{n \in \Z^3}\Pi \nabla\cdot(\Psi f_n  \, \d B_n) \,+\, N_1( U,\Psi)\,\d t,\qquad \nabla\cdot  U\,=\,0, 
		\\
		\d \Psi \, = &\,(1+F_1{/2} +G_1/2)\Delta \Psi \,\d t\, - \,(F_2{/2}+ G_2/8) \Psi\, \d t \\&- \,\sum_{n\in \Z^3} \nabla \cdot ( \Psi g_n \,\d W_n ) \,
		+\, \frac{1}{2} \sum_{n \in \Z^3} \Psi \nabla \cdot ( g_n \, \d  W_n) \,-\,  \frac{1}{\sqrt{2}}\sum_{n\in \Z^3}(\nabla   U: f_n\,\d  B_n) +  N_2(U,\Psi)\,\d t,
		\\ ( U(0),\Psi(0))\, =&\,(  U_0,\Psi_0),
	\end{split}
\end{align}
and couples it to the conditions 
\begin{align}\begin{split}\label{Eq:nonlinearities}&
		N_1( U,\Psi) \,=\, 	N_1( U) \,=\, - \Pi \nabla \cdot ( U \otimes  U),\\&
		N_2( U,\Psi ) \,=\, N_2^a(\Psi ) \,+\, N^b_2( U,\Psi)\,+\, N^c_2( U,\Psi ) \,=\,   h(\Psi)
		|\nabla \Psi|^2 \,+\,  h(\Psi )|\nabla_\sym  U|^2 
		- \nabla \cdot (  U\Psi ).
\end{split}\end{align}
Here, as we aim to show only small time existence,  we may pass to a  modified nonlinearity $h \in C^\infty ( \R)$ with compact support in $(0,\infty)$, which satisfies  $ h(r) = 1/r$ for all $r$  from an open interval  $(\chi_l,\chi_u)$ containing  the range of $\Psi_0$. Below we demonstrate that, after composing \eqref{Eq:nonlinearities} with the truncation operator from Lemma \ref{Lemma:trunc}, the nonlinearities are Lipschitz continuous in suitable spaces. 
Combined with stochastic maximal $L^2$-regularity estimates for the linear but 
inhomogeneous part of \eqref{eq_linear}, i.e., 
\begin{align}\begin{split}\label{Eq_Linear}
		\d U\,=&\,\nabla\cdot(\nabla_{\sym}U)\, \d t\, + \,\frac{F_1}{4}\Delta U\,\d t\, -\, \frac{1}{\sqrt{2}}\sum_{n \in \Z^3}\Pi \nabla\cdot(\Psi f_n   \,\d B_n) \,+\, I \,\d t, \qquad \nabla\cdot U\,=\,0, 
		\\
		\d \Psi \, =\, &(1+F_1{/2} +G_1/2)\Delta \Psi \,\d t \,- \,(F_2{/2}+ G_2/8) \Psi\, \d t \\&- \sum_{n\in \Z^3} \nabla \cdot ( \Psi g_n \,\d W_n ) \,
		+\, \frac{1}{2} \sum_{n \in \Z^3} \Psi \nabla \cdot ( g_n \, \d  W_n) \,-\,  \frac{1}{\sqrt{2}}\sum_{n\in \Z^3}(\nabla  U : f_n\,\d  B_n) \,+ \, J\,\d t,
	\end{split}
\end{align}
for suitable $I$ and $J$,
one is able to apply the Banach fixed-point theorem to conclude the existence  of a strong solution to the modified equation. 
But since the modified nonlinearities coincide with the original ones for values close to $(U_0,\Psi_0)$, the solution solves the original equation 
up to some possibly small but positive stopping  time. Thus, the latter constitutes a local strong solution to \eqref{Eq_new_var} as desired.

\begin{proof}[Proof of Theorem \ref{Prop_ex_strong_solution}]
	We follow the strategy laid out above and aim to  set up a fixed-point argument using the spaces 
	\begin{align}\begin{split}\label{Eq:spaces}
			\Hb \,=\, &\bigl\{ (U,\Psi ) \in H^3(\T^3; \R^3) \times H^3(\T^3) \, \big|\, \nabla \cdot U=0 \bigr\} ,\\
			\Vb \,=\,& \bigl\{ (U,\Psi ) \in H^4(\T^3; \R^3) \times H^4(\T^3)  \, \big|\, \nabla \cdot U = 0 \bigr\} ,\\
			\Wb \,=\, &\bigl\{ (U,\Psi ) \in H^2(\T^3; \R^3) \times H^2(\T^3)  \, \big|\,  \nabla \cdot U = 0 \bigr\}, 
	\end{split}\end{align}
	equipped with their natural norms. We remark that by the characterization of Sobolev spaces based on the Bessel potential, they 
	induce  a Gelfand triple $\Vb \subset \Hb \subset \Vb^*$ with $\Vb^* \eqsim \Wb$. 
	For the sake of  clarity, we divide the proof  into several steps: 
	In the first step, we derive  stochastic maximal $L^2$-regularity estimates in the spaces \eqref{Eq:spaces}
	for the linear equation \eqref{Eq_Linear}. In the second step, we prove compatible local Lipschitz estimates on the nonlinearities \eqref{Eq:nonlinearities}. In the third step, we deduce existence of solutions to a  version of \eqref{eq_linear}--\eqref{Eq:nonlinearities} with additionally truncated nonlinearities. In the  concluding part of this proof, we elaborate how this implies the desired result.

	\textit{Step 1 (Stochastic maximal $L^2$-regularity estimates): For every progressively measurable $(I,J) \in L^2(\Omega\times [0,T];  \Wb)$, there exists a unique adapted, continuous $\Hb$-valued  process $(U,\Psi)$ with 
		\begin{equation}\label{Eq_7213}
			\E\biggl[
			\sup_{t\in [0,T]} \|(U,\Psi)\|_{\Hb}^2 \,+\, \int_0^T \|(U,\Psi)\|_{\Vb}^2\,\d t
			\biggr]\,<\,\infty
		\end{equation}
		satisfying for every $\vp\in C^{\reg}(\T^3; \R^3)$ divergence-free  and $\eta \in C^{\reg}(\T^3)$
		\begin{align}\label{Eq_7214}\\&
			\int_{\T^3}  U_t\cdot \vp\,\d x \,-\, 
			\int_{\T^3} U_0 \cdot \vp\,\d x
			\\&\quad =\,  \int_0^t \int_{\T^3} -  \Bigl(\frac{1}{2}+ \frac{F_1}{4}\Bigr) \nabla U: \nabla  \vp \,\d x\, \d s \,+\, \frac{1}{\sqrt{2}}\sum_{n \in \Z^3}\int_0^t \int_{\T^3}\Psi f_n \nabla \vp\,\d x :\d B_{n,s}
			\ +\, \int_0^t \int_{\T^3}   I\cdot \vp\,\d x \, \d s
			,\\&
			\int_{\T^3} \Psi_t  \eta  \, \d x  \,-\, \int_{\T^3}
			\Psi_0 \eta
			\,\d x \\&\quad =\,  \int_0^t \int_{\T^3} - \bigl(1+ {F_1}/2 +G_1/2\bigr) \nabla \Psi\cdot \nabla  \eta    \,-\, (F_2/2+G_2/8) \Psi \eta \,\d x\, \d s \, +\, \int_0^t \int_{\T^3}   J\eta \,\d x \,\d s
			\\&\qquad + \sum_{i=1}^3
			\sum_{n \in \Z^3}\int_0^t \int_{\T^3} \Psi g_n \partial_i \eta + \frac{1}{2} \Psi  \eta\partial_ig_n  \,\d x \,  \d W_{n,s}^i
			\,-\, \frac{1}{\sqrt{2}}\sum_{n \in \Z^3}\int_0^t\int_{\T^3}\eta f_n \nabla U\,\d x :\d B_{n,s} ,
		\end{align}
		for all $t\in [0,T]$, $\P$-almost surely. Moreover, $(U,\Psi)$ satisfies the estimate 
		\begin{align}\label{Eq42}\E\biggl[
			\sup_{s\in [0,t]} \|(U,\Psi)\|_{\Hb}^2 \,+\, \int_0^t \|(U,\Psi)\|_{\Vb}^2\,\d s
			\biggr]\,
			\lesssim_{(f,g,T)}\,
			\|U_0\|_{\Hb}^2 + \|\Psi_0\|_{\Hb}^2	
			\,+\, \E\bigl[\|(I,J)\|_{L^2(0,t;\Wb)}^2 \bigr] ,
		\end{align} 
		for any $t\in [0,T]$.
	}

	We only show the a priori estimate \eqref{Eq42} for an adapted, continuous $\Hb$-valued  process $(U,\Psi)$ satisfying \eqref{Eq_7213}--\eqref{Eq_7214}. This immediately implies the assertion regarding uniqueness, because, by the linearity of \eqref{Eq_7214}, the difference of two solutions satisfies the same equation but with vanishing initial data  and right-hand side. Moreover, given the a priori estimate \eqref{Eq42}, the existence of solutions follows by a standard Galerkin approximation due to the linearity of the equation, as performed in  \cite[Chapter 4]{LR_book}. 
	
	To prove \eqref{Eq42}, 
	we apply It\^o's formula in order to compute the evolution of the norm
	\begin{equation}\label{Eq991} 
		\frac{1}{2}\|(U_t,\Psi_t)\|_{L^2(\T^3;\R^3)\times L^2(\T^3)}^2 \,+\, \frac{1}{2}\sum_{j,k,l=1}^3\|(\partial_{jkl}U_t,\partial_{jkl}\Psi_t)\|_{L^2(\T^3;\R^3)\times L^2(\T^3)}^2\,\eqsim \,
		\frac{1}{2}\|(U_t,\Psi_t)\|_{\Hb}^2 .
	\end{equation}
	Deriving the It\^o expansion of the former term on the left-hand side amounts to the same calculations as for the proof of the relative energy estimate from Lemma \ref{thm_rel_energy_est}, but in a technically less delicate setting. Indeed, It\^o's formula as stated in \cite[Theorem 4.2.5]{LR_book} applies here and thereby the evolution can be seen from the formula \eqref{Eq_Ito_product} with $(V,\Phi) =(u,\psi^+) = (U,\Psi)$ upon replacing the non-quadratic terms by $2(U,I)_{L^2(\T^3;\R^3)}+2(\Psi,J)_{L^2(\T^3)}$. More precisely, using the cancellation of stochastic integrals remarked below said equation, we have
	\begin{align}\begin{split}
			\label{Eq_Ito_L_2_part}&
			\frac{1}{2} \bigl(
			\|U_t\|_{L^2(\T^3;\R^3)}^2 + \|\Psi_t\|_{L^2(\T^3)}^2
			\bigr) \,=\, 
			\frac{1}{2} \bigl(
			\|U_0\|_{L^2(\T^3;\R^3)}^2 + \|\Psi_0\|_{L^2(\T^3)}^2 \bigr)  \\& \qquad -\, \int_0^t \|\nabla_\sym U\|_{L^2(\T^3;\R^{3\times 3})}^2\,+\,  \|\nabla \Psi\|_{L^2(\T^3;\R^{3})}^2 \,\d s \,+\, 
			\int_0^t (U,I)_{L^2(\T^3;\R^3)}\,+\,  (\Psi,J)_{L^2(\T^3)} \,\d s .
		\end{split}
	\end{align}
	To treat also the leading order term in \eqref{Eq991}, we  take $ -\partial_{jkl}{\vp}$ as test function in \eqref{Eq_7214} to find that 
	\begin{align} \begin{split}&
			\int_{\T^3} \partial_{jkl} U_t	\cdot {\vp}\,\d x\,-\, \int_{\T^3} \partial_{jkl} U_0	\cdot {\vp}\,\d x\,\\&\quad  =\,
			\int_0^t \int_{\T^3} -  \Bigl(\frac{1}{2}+ \frac{F_1}{4}\Bigr) \nabla (\partial_{jkl}U): \nabla  \vp \,\d x\,\d s \, +\,  \frac{1}{\sqrt{2}}\sum_{n \in \Z^3}\int_0^t \int_{\T^3} \partial_{jkl}\Psi f_n \nabla {\vp}\,\d x :\d B_{n,s}
			\\&\qquad +
			\,  \frac{1}{\sqrt{2}}\sum_{n \in \Z^3}\int_0^t \int_{\T^3} [\partial_{jkl} , f_n ] (\Psi) \nabla {\vp}\,\d x :\d B_{n,s}\, +\, \int_0^t \int_{\T^3}  \partial_{jkl} I\cdot \vp\,\d x \,\d s
			,
	\end{split}	\end{align}
	for every divergence-free ${\vp}\in C^{\reg}(\T^3;\R^3)$,
	where 
	\begin{equation}\label{Eq43}
		[\partial_{jkl} , f ] (\Psi) \,=\, \partial_{jkl} (f  \Psi) - f \partial_{jkl}  \Psi
	\end{equation}
	denotes the commutator of differentiation and multiplication with $f$.
	The above is the weak formulation of the equation
	\begin{align}
		& \d \partial_{jkl} U \,=\,  \nabla\cdot\bigl(\nabla_{\sym} (\partial_{jkl} U )\bigr) \,\d t \,+\, \frac{F_1}{4} \Delta (\partial_{jkl} U \bigr) \,\d t \,	- \,\frac{1}{\sqrt{2}}\sum_{n \in \Z^3}\Pi \nabla\cdot((\partial_{jkl} \Psi) f_n  \, \d B_n)
		\\&\qquad  - \frac{1}{\sqrt{2}}\sum_{n \in \Z^3}\Pi \nabla\cdot([\partial_{jkl} , f_n ] (\Psi) \, \d B_n) \,+\, \partial_{jkl} I\,  \d t.
		\label{Eq_u_derivatives}
	\end{align}
	We record for later use also the estimate 
	\begin{equation}\label{eq_comm_est1}
		\| [\partial_{jkl} , f ] (\Psi)  \|_{H^1(\T^3;\ell^2(\Z^3))} \,\lesssim \, \| f\|_{C^{4} (\T^3; \ell^2(\Z^3) )}\| \Psi \|_{H^3(\T^3)},
	\end{equation}
	which holds since  $\Psi$ is differentiated at most twice in \eqref{Eq43} after canceling the leading order contributions. Similarly, by using $- \partial_{jkl} \eta$ as a test function in \eqref{Eq_7214}, we find  that 
	\begin{align}\begin{split}\label{Eq_psi_derivatives}
			\d\bigl( \partial_{jkl} \Psi \bigr)\,=\, 	&(1+F_1{/2} +G_1/2)\Delta \bigl(\partial_{jkl} \Psi\bigr) \, \d t \,-\,  (F_2{/2}+ G_2/8) \bigl(\partial_{jkl} \Psi\bigr) \, \d t \\&- \sum_{n\in \Z^3} \nabla \cdot ( (\partial_{jkl} \Psi )g_n \, \d W_n ) \,
			+\, \frac{1}{2} \sum_{n \in \Z^3} (\partial_{jkl} \Psi) \nabla \cdot ( g_n \,  \d  W_n) \\&-\,  \frac{1}{\sqrt{2}}\sum_{n\in \Z^3}(\nabla  (\partial_{jkl}U) : f_n\, \d  B_n)
			\\&
			- \sum_{n\in \Z^3} \nabla \cdot ( [ \partial_{jkl} ,g_n] (\Psi)\, \d W_n ) \,
			+\, \frac{1}{2} \sum_{i=1}^3\sum_{n \in \Z^3} [\partial_{jkl} ,\partial_{i} g_n](\Psi) \,  \d  W_n^i \\&-\,  \frac{1}{\sqrt{2}}\sum_{n\in \Z^3}( [\partial_{jkl} ,f_n \nabla]  (U) :\d  B_n)
			\, + \,  \partial_{jkl}  J \, \d t,
		\end{split}
	\end{align}
	holds weakly, together with the bounds 
	\begin{align}\begin{split}\label{eq_comm_est2}&\|  [ \partial_{jkl} ,g] (\Psi)\|_{H^1(\T^3; \ell^2(\Z^3))} \,\lesssim \, \|g\|_{C^4(\T^3; \ell^2(\Z^3))}\| \Psi \|_{H^3(\T^3)},
			\\& \|[\partial_{jkl} ,\partial_{i} g](\Psi)  \|_{L^2(\T^3;\ell^2(\Z^3))} \,\lesssim \, \|g\|_{C^4(\T^3; \ell^2(\Z^3))}\| \Psi \|_{H^2(\T^3)},
			\\&\|[\partial_{jkl} ,f\nabla ]  (U)  \|_{L^2(\T^3 ; \ell^2 (\Z^3;\R^{3\times 3}))} \,\lesssim \, \|f\|_{C^3(\T^3; \ell^2(\Z^3))}\| U \|_{H^3(\T^3; \R^{3})},
	\end{split}\end{align}
	on the emerging commutators.
	
	Next, we use It\^o's formula as stated in  \cite[Theorem 4.2.5]{LR_book} to compute the $L^2$-norm of $\partial_{jkl} U $ and $\partial_{jkl} \Psi$, which is applicable by the regularity required in \eqref{Eq_7213}. Moreover, since the system defined through \eqref{Eq_u_derivatives} and \eqref{Eq_psi_derivatives} is a perturbation of \eqref{Eq_Linear}
	with $I$ and $J$ replaced by $\partial_{jkl} I$ and $\partial_{jkl} J$, 
	the resulting It\^o expansion is a perturbation of \eqref{Eq_Ito_L_2_part}. Namely, we have 
	\begin{align}\begin{split}\label{Eq75}&
			\frac{1}{2} \bigl(
			\|\partial_{jkl}U_t\|_{L^2(\T^3;\R^3)}^2 + \|\partial_{jkl}\Psi_t\|_{L^2(\T^3)}^2
			\bigr) \,=\, 
			\frac{1}{2} \bigl(
			\|\partial_{jkl}U_0\|_{L^2(\T^3;\R^3)}^2 + \|\partial_{jkl}\Psi_0\|_{L^2(\T^3)}^2 \bigr)  \\& \qquad -\, \int_0^t \|\nabla_\sym (\partial_{jkl}U)\|_{L^2(\T^3;\R^{3\times 3})}^2\, +\,  \|\nabla( \partial_{jkl})\Psi\|_{L^2(\T^3;\R^{3})}^2 \,\d s \\&\qquad+\, 
			\int_0^t \langle \partial_{jkl}U,\partial_{jkl}I\rangle_{H^1(\T^3;\R^3)\times H^{-1}(\T^3; \R^3)}\,+\,  \langle \partial_{jkl}\Psi,\partial_{jkl}J\rangle_{H^1(\T^3)\times H^{-1}(\T^3)} \,\d s \\&\qquad - \frac{1}{\sqrt{2}}\sum_{n \in \Z^3} \int_0^t (  (\partial_{jkl}U), \nabla\cdot ( [\partial_{jkl} , f_n ] (\Psi)\, \d B_{n,s} ))_{L^2(\T^3;\R^3)}
			\\&\qquad - \sum_{n\in \Z^3} \int_0^t ( \partial_{jkl} \Psi, \nabla \cdot ( [ \partial_{jkl} ,g_n] (\Psi)\,\d W_{n,s} ))_{L^2(\T^3)} \\&\qquad 
			+\, \frac{1}{2}\sum_{i=1}^3 \sum_{n \in \Z^3}\int_0^t (\partial_{jkl} \Psi, [\partial_{jkl} ,\partial_i g_n](\Psi) )_{L^2(\T^3)} \, \d  W_{n,s}^i \\&\qquad-\,  \frac{1}{\sqrt{2}}\sum_{n\in \Z^3}\int_0^t  ( \partial_{jkl}\Psi , ( [\partial_{jkl} ,f_n \nabla]  (U) : \d  B_{n,s}))_{L^2(\T^3)}		\,+\, \frac{1}{2}{\mathrm{CV}}_{\mathrm{rest}, U}^{{j,k,l}}(t) \,+\, \frac{1}{2}{\mathrm{CV}}_{\mathrm{rest}, \Psi}^{{j,k,l}}(t),
		\end{split}
	\end{align}
	where the remaining contributions of the scalar  covariation process  are given by 
	\begin{align*} \mathrm{CV}_{\mathrm{rest} , U}^{j,k,l}(t) \,=\,&
		\frac{1}{2}	\biggl\langle 
		\sum_{n \in \Z^3}\int_0^t \Pi \nabla\cdot([\partial_{jkl} , f_n ] (\Psi)  \,\d B_{n,s}),
		2\sum_{n \in \Z^3}\int_0^t \Pi \nabla\cdot((\partial_{jkl} \Psi) f_n  \, \d B_{n,s}) \\&\qquad   +\sum_{n \in \Z^3}\int_0^t \Pi \nabla\cdot([\partial_{jkl} , f_n ] (\Psi) \, \d B_{n,s})
		\biggr\rangle_t,
	\end{align*}
	and
	\begin{align*} 
		\mathrm{CV}_{\mathrm{rest} , \Psi}^{j,k,l}(t) \,=\, & \frac{1}{2}\biggl\langle 
		\sum_{n\in \Z^3}\int_0^t ( [\partial_{jkl} ,f_n\nabla]  (U) : \d  B_{n,s}), 2\sum_{n\in \Z^3} \int_0^t (\nabla  (\partial_{jkl}U) : f_n\, \d  B_{n,s}) \\&\qquad +\,\sum_{n\in \Z^3}\int_0^t ( [\partial_{jkl} ,f_n\nabla]  (U): \d  B_{n,s})
		\biggr\rangle_t
		\\& +\,\biggl\langle
		- \sum_{n\in \Z^3} \int_0^t \nabla \cdot ( [ \partial_{jkl} ,g_n] (\Psi)\, \d W_{n,s} ) \,
		+\, \frac{1}{2}\sum_{i=1}^3 \sum_{n \in \Z^3}\int_0^t [\partial_{jkl} ,\partial_i g_n](\Psi)  \, \d  W_{n,s}^i,
		\\&\qquad  - 2\sum_{n\in \Z^3} \int_0^t \nabla \cdot ( (\partial_{jkl} \Psi )g_n \,\d W_{n,s} ) \,
		+\, \sum_{n \in \Z^3} \int_0^t  (\partial_{jkl} \Psi) \nabla \cdot ( g_n \, \d  W_{n,s})
		\\&\qquad  
		- \sum_{n\in \Z^3} \int_0^t  \nabla \cdot ( [ \partial_{jkl} ,g_n] (\Psi)\,\d  W_{n,s} ) \,
		+\, \frac{1}{2}\sum_{i=1}^3  \sum_{n \in \Z^3} \int_0^t  [\partial_{jkl} , \partial_i g_n ](\Psi) \, \d  W_{n,s}^i
		\biggr\rangle_t,
	\end{align*}
	paired in $L^2(\T^3;\R^3)$ and $L^2(\T^3)$, respectively. 
	Without precisely computing the latter, based on the previous estimates \eqref{eq_comm_est1} and \eqref{eq_comm_est2} together with 
	\begin{align}\begin{split}\label{eq_comm_est3}&
			\|  (\partial_{jkl}\Psi ) f\|_{H^1(\T^3; \ell^2(\Z^3))}\,\lesssim\, \| f\|_{C^1(\T^3; \ell^2(\Z^3))} \| \Psi \|_{H^4(\T^3)} ,\\&
			\|f \nabla  (\partial_{jkl}U) \|_{L^2(\T^3;  \ell^2(\Z^3;\R^{3\times 3}))} \,\lesssim \| f\|_{C(\T^3; \ell^2(\Z^3))}\|U\|_{H^4(\T^3;\R^3)},\\&
			\| (\partial_{jkl} \Psi )g\|_{H^1(\T^3; \ell^2(\Z^3))}\,\lesssim\, \| g\|_{C^1(\T^3; \ell^2(\Z^3))} \|\Psi\|_{H^4(\T^3)},\\&
			\|(\partial_{jkl} \Psi)\partial_{i} g  \|_{L^2(\T^3; \ell^2(\Z^3))}\,\lesssim \, \|g\|_{C^1(\T^3; \ell^2(\Z^3))}\|\Psi\|_{H^3(\T^3)},
		\end{split}
	\end{align} 
	we deduce that
	\begin{align*}&
		\bigl|\mathrm{CV}_{\mathrm{rest},U}^{j,k,l}(t)\,+\, \mathrm{CV}_{\mathrm{rest},\Psi }^{j,k,l}(t) \bigr| \,\lesssim_{f,g}\,\int_0^t \bigl(\|U\|_{H^4(\T^3; \R^3)} \,+\, \|\Psi\|_{H^4(\T^3)}\bigr)\bigl(\|U\|_{H^3(\T^3; \R^3)} \,+\, \|\Psi\|_{H^3(\T^3)}\bigr)\,\d s
		\\&\quad \le \epsilon \int_0^t  \|U\|_{H^4(\T^3; \R^3)}^2 \,+\, \|\Psi\|_{H^4(\T^3)}^2\,\d s
		\,+\, 
		C_\epsilon\int_0^t  \sup_{\tau \in [0,s]}\bigl(\|U\|_{H^3(\T^3; \R^3)}^2 \,+\, \|\Psi\|_{H^3(\T^3)}^2\bigr) \, \d s ,
	\end{align*}
	fo any $\epsilon>0$ by Young's inequality. 
	Inserting the above in \eqref{Eq75}, taking the supremum until some time $t\in [0,T]$ and the expectation yields  by the Burkholder--Davis--Gundy inequality that
	\begin{align*}&
		\E\Bigl[\sup_{s\in [0,t]}
		\bigl(
		\|\partial_{jkl}U\|_{L^2(\T^3;\R^3)}^2 + \|\partial_{jkl}\Psi\|_{L^2(\T^3)}^2
		\bigr)
		\Bigr]
		\\&\quad +\, \E\biggl[
		\int_0^t \|\nabla_\sym (\partial_{jkl}U)\|_{L^2(\T^3;\R^{3\times 3})}^2\, +\,  \|\nabla( \partial_{jkl})\Psi\|_{L^2(\T^3;\R^{3})}^2 \,\d s
		\biggr]\\&\quad\lesssim_{f,g} \,
		\|\partial_{jkl}U_0\|_{L^2(\T^3;\R^3)}^2 + \|\partial_{jkl}\Psi_0\|_{L^2(\T^3)}^2	
		\,+\, \E\bigl[ (\mathrm{QV}^{j,k,l})^{1/2}(t)\bigr] \\&\qquad +\, \E\biggl[ \int_0^t \langle \partial_{jkl}U,\partial_{jkl}I\rangle_{H^1(\T^3;\R^3)\times H^{-1}(\T^3; \R^3)}\,+\,  \langle \partial_{jkl}\Psi,\partial_{jkl}J\rangle_{H^1(\T^3)\times H^{-1}(\T^3)} \,\d s\biggr] \\&\qquad+\,
		\epsilon\E\biggl[ \int_0^t  \|U\|_{H^4(\T^3; \R^3)}^2 \,+\, \|\Psi\|_{H^4(\T^3)}^2\,\d s\biggr]
		\,+\, 
		C_\epsilon\int_0^t \E\Bigl[ \sup_{\tau \in [0,s]}\bigl(\|U\|_{H^3(\T^3; \R^3)}^2 \,+\, \|\Psi\|_{H^3(\T^3)}^2\bigr)\Bigr] \, \d s,
	\end{align*}
	where $\mathrm{QV}^{j,k,l}$ is the quadratic variation process of the martingale part of \eqref{Eq75}. Concerning the latter, analogous considerations as for $ 	|\mathrm{CV}_{\mathrm{rest},U}^{j,k,l}(t)\,+\, \mathrm{CV}_{\mathrm{rest},\Psi }^{j,k,l}(t) |$, based on the bounds \eqref{eq_comm_est1} and \eqref{eq_comm_est2}, lead to  the estimate
	\begin{align*}&\E\bigl[ (\mathrm{QV}^{j,k,l})^{1/2}(t)\bigr]\\&
		\quad \lesssim_{f,g}  \E\biggl[
		\biggl(
		\int_0^t 
		\|U\|_{H^3(\T^3; \R^3)}^4 +\|\Psi\|_{H^3(\T^3)}^4
		\,\d s
		\biggr)^{1/2}
		\biggr]
		\\& \quad \le\,\E\biggl[
		\Bigl(\sup_{s\in [0,t]} \|U\|_{H^3(\T^3; \R^3)}^2 \,+\, \|\Psi\|_{H^3(\T^3)}^2\Bigr)^{1/2}\biggl(
		\int_0^t 
		\|U\|_{H^3(\T^3; \R^3)}^2 +\|\Psi\|_{H^3(\T^3)}^2
		\,\d s
		\biggr)^{1/2}
		\biggr]
		\\&\quad \le \, \epsilon \E\Bigl[
		\sup_{s\in [0,t]} \bigl(\|U\|_{H^3(\T^3; \R^3)}^2 \,+\, \|\Psi\|_{H^3(\T^3)}^2\bigr)
		\Bigr] \,+\, 
		C_\epsilon\int_0^t \E\Bigl[ \sup_{\tau \in [0,s]}\bigl(\|U\|_{H^3(\T^3; \R^3)}^2 \,+\, \|\Psi\|_{H^3(\T^3)}^2\bigr)\Bigr] \, \d s.
	\end{align*} 
	Another application of Young's inequality in the term with the inhomogeneities $I$ and $J$ and a suffciently small choice of $\epsilon>0$ yields the estimate
	\begin{align}\begin{split}\label{Eq41}&
			\E\Bigl[\sup_{s\in [0,t]}
			\bigl(
			\|\partial_{jkl}U\|_{L^2(\T^3;\R^3)}^2 + \|\partial_{jkl}\Psi\|_{L^2(\T^3)}^2
			\bigr)
			\Bigr]
			\\&\quad +\, \E\biggl[
			\int_0^t \|\nabla (\partial_{jkl}U)\|_{L^2(\T^3;\R^{3\times 3})}^2\, +\,  \|\nabla( \partial_{jkl})\Psi\|_{L^2(\T^3;\R^{3})}^2 \,\d s
			\biggr]\\&\quad\lesssim_{f,g} \,
			\|\partial_{jkl}U_0\|_{L^2(\T^3;\R^3)}^2 + \|\partial_{jkl}\Psi_0\|_{L^2(\T^3)}^2	
			\,+\, \E\bigl[\|I\|_{L^2(0,t;H^2(\T^3; \R^3))}^2 + \|J\|_{L^2(0,t; H^2(\T^3))}^2\bigr] \\&\qquad +\,
			\int_0^t \E\Bigl[ \sup_{\tau \in [0,s]}\bigl(\|U\|_{H^3(\T^3; \R^3)}^2 \,+\, \|\Psi\|_{H^3(\T^3)}^2\bigr)\Bigr] \, \d s,
		\end{split}
	\end{align}
	where we also use that $\partial_{jkl} U $ is divergence-free, so that 
	\[
	\| \nabla_\sym \partial_{jkl} U\|_{L^2(\T^3 ; \R^{3\times 3})}^2 \,=\, 
	\frac{1}{2}\| \nabla \partial_{jkl}U\|_{L^2(\T^3 ; \R^{3\times 3})}^2. 
	\]
	Similarly, taking the supremum until time $t$ in \eqref{Eq_Ito_L_2_part}, using Young's inequality and taking the expectation yields that also 
	\begin{align*}&
		\E\Bigl[\sup_{s\in [0,t]}
		\bigl(
		\|U\|_{L^2(\T^3;\R^3)}^2 + \|\Psi\|_{L^2(\T^3)}^2
		\bigr)
		\Bigr]\, +\, \E\biggl[
		\int_0^t \|\nabla U\|_{L^2(\T^3;\R^{3\times 3})}^2\, +\,  \|\nabla \Psi\|_{L^2(\T^3;\R^{3})}^2 \,\d s
		\biggr]\\&\quad\lesssim\,
		\|U_0\|_{L^2(\T^3;\R^3)}^2 + \|\Psi_0\|_{L^2(\T^3)}^2	
		\,+\, \E\bigl[\|I\|_{L^2(0,t;H^{-1}(\T^3; \R^3))}^2 + \|J\|_{L^2(0,t; H^{-1}(\T^3))}^2\bigr] .
	\end{align*}
	Adding this to the sum of \eqref{Eq41} over all combinations $j,k,l\in\{1,2,3\}$ together with the norm equivalences \eqref{Eq991} and 
	\[
	\|(U,\Psi)\|_{\Vb}^2 \,\eqsim \, 
	\|(U,\Psi)\|_{L^2(\T^3;\R^3)\times L^2(\T^3)}^2 \,+\, \sum_{j,k,l=1}^3\|(\nabla \partial_{jkl}U,\nabla \partial_{jkl}\Psi)\|_{L^2(\T^3;\R^{3\times 3})\times L^2(\T^3; \R^3)}^2
	\]
	yields finally 
	\begin{align*}
		\E\Bigl[\sup_{s\in [0,t]} \|(U,\Psi)\|_{\Hb}^2 \,+\, \int_0^t \|(U,\Psi )\|_{\Vb}^2 \, \d s\Bigr] \,\lesssim_{f,g}\, 
		\|(U_0,\Psi_0)\|_{\Hb}^2
		\,+\,
		\int_0^t \E\Bigl[ \sup_{\tau \in [0,s]}\|(U,\Psi)\|_{\Hb}^2 \Bigr]\,+\, \E\bigl[\|(I,J)\|_{\Wb}^2\bigr] \, \d s.
	\end{align*}
	The desired \eqref{Eq42} follows by Gr\"onwall's inequality.

	\textit{Step 2 (Estimates on the nonlinearities): It holds 
		\begin{align}\label{Est_nonlinearity_tracial}
			\bigl\|\bigl( N_1(U) -N_1(V) , N_2(U,\Psi) - N_2(V,\Phi)\bigr)\bigr\|_{\Wb} \,\lesssim_h \,\bigl( 1+ \|(U,\Psi) \|_{\Hb}^4 + \|(V,\Phi) \|_{\Hb}^4 \bigr) \bigl\|(U,\Psi) -(V,\Phi)\bigr\|_{\Hb}
		\end{align}
		for all $(U,\Psi), (V,\Phi) \in \Hb$. }

	We let  $(U,\Psi), (V,\Phi) \in \Hb$ and recall the notation from \eqref{Eq:nonlinearities}. Then, we estimate 
	\begin{align}\begin{split}\label{Eq:N_1-est}
			\|(N_1(U) -N_1(V)) \|_{H^2(\T^3; \R^3)}\,\lesssim &  \, \| U\otimes U -V\otimes V\|_{H^3(\T^3; \R^{3\times 3})} \\ \le  & \, \| U\otimes ( U -V) \|_{H^3(\T^3;\R^{3\times 3})} \,+\,  \| V\otimes ( U -V) \|_{H^3(\T^3;\R^{3\times 3})} 
			\\\lesssim & \,\bigl(
			\|U\|_{H^3(\T^3; \R^3)} \,+\, \| V\|_{H^3(\T^3; \R^3)} \bigr)\|U-V\|_{H^3(\T^3; \R^3)}
			,\end{split}
	\end{align}
	by the algebra property of $H^3(\T^3)$. Completely analogously, we find that 
	\begin{align}\begin{split}\label{Eq:N_2^c-est}
			\|	N_2^c(U,\Psi) -  N_2^c(V,\Phi) \|_{H^2(\T^3)} \,\lesssim &  \, \|\Psi  U -\Phi V \|_{H^3(\T^3; \R^3)} \\\le  & \, \| \Psi  ( U -V) \|_{H^3(\T^3;\R^{3})} \,+\,  \| (\Psi-\Phi) V \|_{H^3(\T^3;\R^{3})} 
			\\\lesssim & \,
			\|\Psi \|_{H^3(\T^3)} \|U-V\|_{H^3(\T^3; \R^3)}\,+\, \|V\|_{H^3(\T^3; \R^3)} \|\Psi - \Phi\|_{H^3(\T^3)}
			.\end{split}
	\end{align}
	Concerning $N_2^a$, we estimate on the other hand 
	\begin{align}
		\|	N_2^a(\Psi) -  N_2^a(\Phi) \|_{H^2(\T^3)} = &  \, \|h(\Psi) |\nabla \Psi |^2  - h(\Phi) |\nabla \Phi |^2\|_{H^2(\T^3)} \\\le  & \, \|(h ( \Psi) -h(\Phi))  |\nabla \Psi |^2   \|_{H^2(\T^3)} \,+\,  \| h(\Phi ) (\nabla \Psi   + \nabla \Phi   )\cdot (\nabla \Psi   -\nabla \Phi   ) \|_{H^2(\T^3)} 
		\\\lesssim & \,\|\nabla \Psi \|_{H^2(\T^3; \R^3)}^2
		\|h (\Psi) - h(\Phi) \|_{H^2(\T^3)} \\& \;+\, \|h (\Phi)\|_{H^2(\T^3)}\bigl(
		\|\nabla \Psi \|_{H^2(\T^3; \R^3)} + 
		\|\nabla \Phi \|_{H^2(\T^3; \R^3)}
		\bigr) \|\nabla \Psi - \nabla \Phi\|_{H^2(\T^3; \R^3)}
		,
	\end{align}
	since also $H^2(\T^3)$ is an algebra. Estimating $\|h (\Psi) - h(\Phi) \|_{H^2(\T^3)}$ and $\|h (\Phi)\|_{H^2(\T^3)} $ using \eqref{Eq:bddness_nu}--\eqref{Eq:diff_nu} from Lemma \ref{lemma:comp} and Young's inequality for products, we obtain that
	\begin{align}\begin{split}\label{Eq:N_2^a-est}
			\|	N_2^a(\Psi) -  N_2^a(\Phi) \|_{H^2(\T^3)}  \,\lesssim &\, 
			\| h'\|_{W^{2,\infty}(\R)} \bigl(1+ \|\Psi \|_{H^3(\T^3)}^4 + \|\Phi \|_{H^3(\T^3)}^4  \bigr) \|\Psi -\Phi \|_{H^2(\T^3)}
			\\&\;+\, 
			\| h'\|_{W^{1,\infty}(\R)} \bigl(1+ \|\Psi \|_{H^3(\T^3)}^3 + \|\Phi \|_{H^3(\T^3)}^3  \bigr)  \|\Psi - \Phi \|_{H^3(\T^3)}\\\lesssim &\,
			\| h'\|_{W^{2,\infty}(\R)} \bigl(1+ \|\Psi \|_{H^3(\T^3)}^4 + \|\Phi \|_{H^3(\T^3)}^4  \bigr) \|\Psi -\Phi \|_{H^3(\T^3)}.\end{split}
	\end{align}
	For $N^b_2$, we proceed analogously to deduce that 
	\begin{align}\begin{split}
			\label{Eq:N_2^b-est}
			&	\|	N_2^b(U,\Psi) -  N_2^b(V, \Phi) \|_{H^2(\T^3)} \\& \quad =  \, \|h(\Psi) |\nabla_{\sym} U |^2  - h(\Phi) |\nabla_\sym V |^2\|_{H^2(\T^3)} \\& \quad \le   \, \|(h ( \Psi) -h(\Phi))  |\nabla_\sym U |^2   \|_{H^2(\T^3)} \, +\,  \| h(\Phi ) (\nabla_\sym U   + \nabla_\sym V  ): (\nabla_\sym U   -\nabla_\sym V   ) \|_{H^2(\T^3)} 
			\\ &\quad \lesssim \,\|\nabla U \|_{H^2(\T^3; \R^{3\times 3})}^2
			\|h (\Psi) - h(\Phi) \|_{H^2(\T^3)} \\ &\qquad +\, \|h (\Phi)\|_{H^2(\T^3; \R^3)}\bigl(
			\|\nabla U \|_{H^2(\T^3; \R^{3\times 3})} + 
			\|\nabla V\|_{H^2(\T^3; \R^{3\times 3})}
			\bigr) \|\nabla U - \nabla V\|_{H^2(\T^3; \R^{3\times 3})}
			\\&\quad 
			\lesssim 
			\| h'\|_{W^{2,\infty}(\R)} \bigl(1+ \|\Psi \|_{H^3(\T^3)}^4 + \|\Phi \|_{H^3(\T^3)}^4 + \|U \|_{H^3(\T^3;\R^3)}^4 + \|V \|_{H^3(\T^3;\R^3)}^4   \bigr) \|\Psi -\Phi \|_{H^2(\T^3)}
			\\&\qquad +\, 
			\| h'\|_{W^{1,\infty}(\R)} \bigl(1+ \|\Psi \|_{H^3(\T^3)}^3 + \|\Phi \|_{H^3(\T^3)}^3 + \|U \|_{H^3(\T^3;\R^3)}^3 + \|V \|_{H^3(\T^3;\R^3)}^3 \bigr)  \|U - V \|_{H^3(\T^3;\R^3)}\\&\quad \lesssim \,
			\| h'\|_{W^{2,\infty}(\R)}\bigl(1+ \|\Psi \|_{H^3(\T^3)}^4 + \|\Phi \|_{H^3(\T^3)}^4 + \|U \|_{H^3(\T^3;\R^3)}^4 + \|V \|_{H^3(\T^3;\R^3)}^4   \bigr)
			\\ &\qquad \times   \bigl( \|\Psi -\Phi \|_{H^3(\T^3)} + \|U-V\|_{H^3(\T^3;\R^3)} \bigr).
		\end{split}
	\end{align}
	Combining \eqref{Eq:N_1-est}--\eqref{Eq:N_2^b-est}\noeqref{Eq:N_2^c-est}\noeqref{Eq:N_2^a-est}, we obtain \eqref{Est_nonlinearity_tracial}.
	
	\textit{Step 3 (Truncation and fixed-point argument):  For any $R > \|(U_0,\Psi_0)\|_{\Hb}$, there exists   $T_* \in (0,T]$, and an adapted, continuous $\Hb$-valued  process $(U,\Psi)$, which satisfies  with 
		for every divergence-free $\vp\in C^{\reg}(\T^3; \R^3)$ and $\eta \in C^{\reg}(\T^3)$
		\begin{align}&\label{Eq744}
			\int_{\T^3} U_t\cdot \vp\,\d x \,-\, 
			\int_{\T^3} U_0 \cdot \vp\,\d x\,=\, \int_0^t \int_{\T^3}   N_1(\Xi^R_\Hb(U,\Psi))\cdot \vp\,\d x \,\d s
			\\&\qquad +\,  \int_0^t \int_{\T^3} -  \Bigl(\frac{1}{2}+ \frac{F_1}{4}\Bigr) \nabla U: \nabla  \vp \,\d x\,\d s \,+\, \frac{1}{\sqrt{2}}\sum_{n \in \Z^3}\int_0^t \int_{\T^3}\Psi f_n \nabla \vp\,\d x :\d B_{n,s}
			,\\&\label{Eq745}
			\int_{\T^3} \Psi_t  \eta  \, \d x  \,-\, \int_{\T^3}
			\Psi_0 \eta
			\,\d x\,=\,  \int_0^t \int_{\T^3}   N_2(\Xi^R_\Hb(U,\Psi))\eta \,\d x\, \d s \\&\qquad +\,  \int_0^t \int_{\T^3} - \bigl(1+ {F_1}/2 +G_1/2\bigr) \nabla \Psi\cdot \nabla  \eta    \,-\, (F_2/2+G_2/8) \Psi \eta \,\d x\,\d s 
			\\&\qquad + \sum_{i=1}^3
			\sum_{n \in \Z^3}\int_0^t \int_{\T^3} \Psi g_n \partial_i \eta  + \frac{1}{2} \Psi \eta \partial_ig_n  \,\d x \, \d W_{n,s}^i
			\,-\, \frac{1}{\sqrt{2}}\sum_{n \in \Z^3}\int_0^t\int_{\T^3}\eta f_n \nabla U\,\d x :\d B_{n,s} ,
		\end{align}
		for all $t\in [0,T_*]$, $\P$-almost surely, where $\Xi^R_{\Hb}$ is defined in \eqref{Eq:definition_trunc}.}
	
	We denote the mapping that assigns each initial value and right-hand side its solution to \eqref{Eq_Linear} provided by the first step, by $\Sol((U_0,\Psi_0), (I,J))$. In the following, we show that given $R > \|(U_0,\Psi_0)\|_{\Hb}$, for suitably chosen   $T_* \in (0,T]$, the composition 
	\begin{align}
		\Lambda\colon \, &M_{T_*} \to M_{T_*}
		,\,(U,\Psi) \mapsto \Sol\bigl((U_0,\Psi_0),(N_1(\Xi^{R}_\Hb(U,\Psi)), N_2(\Xi^{R}_\Hb(U,\Psi)) )\bigr)
	\end{align}
	is a Lipschitz contraction, where 
	\begin{equation}
		M_{T_*}\,=\, \Bigl\{ (U,\Psi)\colon \Omega\times [0,T_*] \to \Hb\text{ adapted, continuous}\,\Big|\,(U,\Psi) \,\in \, L^2(\Omega;C([0,T_*]; \Hb) \cap L^2(0,T_*; \Vb))  \Bigr\},
	\end{equation}
	with its natural norm. 
	To see that $\Lambda$ maps $M_{T_*}$ into itself, we observe that for $(U,\Psi)\in M_{T_*}$, we have 
	\begin{equation}\label{Eq9999}
		\|\Xi^{R}_\Hb(U,\Psi) \|_{C([0,T_*];\Hb)}\,\le \, R ,
	\end{equation}
	$\P$-a.s., by Lemma \ref{Lemma:trunc}. Using that $N_1(0,0)=0$ and $N_2(0,0)=0$, cf.\ \eqref{Eq:nonlinearities}, this implies by the estimate \eqref{Est_nonlinearity_tracial} derived in the second step  that 
	\[
	\bigl\|\bigl(N_1(\Xi^{R}_\Hb(U,\Psi)),N_2(\Xi^{R}_\Hb(U,\Psi))\bigr)\bigr\|_{C([0,T_*];\Wb)} \,\lesssim_h \, \bigl(
	1+R^4
	\bigr)R,
	\]
	$\P$-almost surely. 
	In particular, we have that
	\[
	(N_1(\Xi^{R}_\Hb(U,\Psi)),N_2(\Xi^{R}_\Hb(U,\Psi))) \,\in  \, L^2(\Omega\times [0,T_*]; \Wb)
	\]
	is progressively measurable and therefore an admissible right-hand side for the operator $\Sol$. In light of  \eqref{Eq_7213}, the latter maps however to $M_{T_*}$. 
	
	To show also the contraction property, we consider the difference $\Lambda (U,\Psi) - \Lambda (V,\Phi)$, which also solves \eqref{Eq_Linear}, but started from $0$ and with right-hand sides $I = N_1(\Xi^{R}_\Hb(U,\Psi))- N_1(\Xi^{R}_\Hb(V,\Phi))$ and $J = N_2(\Xi^{R}_\Hb(U,\Psi))- N_2(\Xi^{R}_\Hb(V,\Phi))$. Using the estimate \eqref{Eq42} we find that
	\begin{align}&
		\| \Lambda (U,\Psi) - \Lambda (V,\Phi)\|_{M_{T_*}} \\&\quad\lesssim_{(f,g,T)}\, \E \bigl[ \bigl\| \bigl(N_1(\Xi^{R}_\Hb(U,\Psi))- N_1(\Xi^{R}_\Hb(V,\Phi)),N_2(\Xi^{R}_\Hb(U,\Psi))- N_2(\Xi^{R}_\Hb(V,\Phi))\bigr)\bigr\|_{L^2(0,T_*;\Wb)}^2 \bigr]^{1/2}
		\\&\quad \lesssim_h \, (1+R^4)\,\E \bigl[ \bigl\|\Xi^{R}_\Hb(U,\Psi) - \Xi^{R}_\Hb(V,\Phi)\bigr\|_{L^2(0,T_*; \Hb)}^2\bigr]^{1/2},
	\end{align}
	based on \eqref{Est_nonlinearity_tracial} and \eqref{Eq9999}. Lastly, we  use H\"older's inequality in time, to deduce that 
	\[
	\| \Lambda (U,\Psi) - \Lambda (V,\Phi)\|_{M_{T_*}}\, \lesssim_{(f,g,h,T)}\, \sqrt{T_*}(1+R^4) \| (U,\Psi) -(V,\Phi)\|_{M_{T_*}}.
	\]
	Choosing $T_*\in (0,T]$ small enough, the above becomes a contraction estimate. An application of Banach's fixed-point theorem yields the claim of this step.

	\textit{Conclusion.} It remains to take $R > \|(U_0,\Psi_0)\|_{\Hb}$ and $T_*$ as well as $(U,\Psi)$ as constructed in the previous step. Then, we  define the stopping time 
	\[
	\tau \,=\, \inf\Bigl\{
	t\in [0,T_*]\,\Big|\, 
	\inf_{x\in \T^3}\Psi_t(x) \le \chi_l\;\text{or}\;
	\sup_{x\in \T^3} \Psi_t(x) \ge \chi_u\;\text{or}\; \sup_{s\in [0,t]}\|(U_t,\Psi_t)\|_{\Hb} \,\ge \, R
	\Bigr\} \,\wedge \, T_*,
	\]
	where we recall that $\chi_l$ and $\chi_u$ were introduced below \eqref{Eq:nonlinearities}. Since $(U,\Psi)$ starts from $(U_0,\Psi_0)$ and is continuous in \begin{equation}\label{Eq746}\Hb \hookrightarrow C^1(\T^3; \R^3)\times C^1(\T^3),
	\end{equation}
	we find that $\tau >0$, $\P$-almost surely. In particular, we also have 
	\[
	\Xi_{\Hb}^{R} (U,\Psi) \,=\, (U,\Psi)\qquad \text{and} \qquad h ( \Psi) \,=\, \frac{1}{\Psi},
	\]
	on $[0,\tau)$, which implies that the modified nonlinearities in \eqref{Eq744} and \eqref{Eq745} coincide with the ones from \eqref{Eq_new_var} on said interval. The embedding \eqref{Eq746} yields also the continuity required in Definition \ref{defi_strong_sol}, which finishes the proof.
\end{proof}

\appendix
\renewcommand{\appendixname}{Appendix~\Alph{section}}
\renewcommand{\theequation}{A.\arabic{equation}}
\section{Verification of Jacobi's identity in the GENERIC framework}\label{sec-app-A}
Here we verify \eqref{Jacobis_identity} for the operator $\mathcal{L}$ defined in \eqref{eqn_M_and_L}. This implies that $\{A,B\}_\mathcal{L}$ is a Poisson bracket, confirming the Hamiltonian structure of the evolution $\partial_t \mathbf{z} = \mathcal{L} \cdot \frac{\delta \mathcal{E}}{\delta \mathbf{z}}$. To this end, we let $A$, $B$ and $C$  be functionals  for which we record that  their second derivatives can be represented by a symmetric $\R^{4\times 4}$-valued kernel by Schwarz's theorem, e.g.,  we have
\begin{align}\label{eqn1059}
	\frac{\delta^{2}A}{\delta \mathbf{z}^2} (x,y) \,=\, \frac{\delta } {\delta \mathbf{z}} \Bigl( \frac{\delta A}{\delta \mathbf{z}}(x) \Bigr) (y) \,=\, \frac{\delta } {\delta \mathbf{z}} \Bigl( \frac{\delta A}{\delta \mathbf{z}}(y) \Bigr) (x).
\end{align}
 We also introduce the notation   $\frac{\delta A}{\delta \mathbf{z}} = (v_A ,\theta_A) $, and analogous conventions  for $B$ and $C$. As a first step, we calculate the functional $ \{B,C\}_\mathcal{L}$ to be 
 \begin{align}&
 		\begin{bmatrix}
 		v_B
 		\\
 		\theta_B
 	\end{bmatrix} \cdot 
 	\begin{bmatrix} -\Pi (\nabla \cdot ( v_C \otimes u ) +  (\nabla v_C)\cdot u )  - \Pi (\vt \nabla \theta_C) 
 		\\
 		-\nabla \cdot  (\vt  v_C )
 	\end{bmatrix}\\&\quad =\,- \int_{\T^3}  v_B \cdot (\nabla \cdot(  v_C \otimes u )) \,\d x\,-\, \int_{\T^3} u \cdot  ( (v_B \cdot \nabla)   v_C )  \,\d x
 \,-\, 
 \int_{\T^3}  \vt v_B\cdot \nabla \theta_C \,\d x\,-\, 
 \int_{\T^3}  \theta_B  \nabla \cdot(\vt v_C ) \,\d x
 \\& \quad = \, \int_{\T^3} u \cdot \bigl(
( v_C \cdot \nabla ) v_B - ( v_B \cdot \nabla ) v_C
 \bigr)\,\d x\, +\, \int_{\T^3} \vt \bigl(v_C \cdot \nabla \theta_B - v_B \cdot \nabla \theta_C \bigr) \,\d x,
\end{align}
using integration by parts. Introducing also
\begin{equation}\label{eqn1044}
	\biggl\{   	\begin{bmatrix}
		v_1
		\\
		\theta_1
	\end{bmatrix} , 
\begin{bmatrix}
	v_2
	\\
	\theta_2
\end{bmatrix} \biggr\}_x\,=\,(v_2\cdot \nabla_x )  \begin{bmatrix}
v_1\\
\theta_1
\end{bmatrix} \,-\, (v_1\cdot \nabla_x )  \begin{bmatrix}
v_2
\\
\theta_2
\end{bmatrix},
\end{equation}
we may rewrite the above to 
\begin{equation}
	\label{eqn1046}
	\{B,C\}_\mathcal{L} \,=\, 
	  \int_{\T^3} 	\begin{bmatrix}
	 u
	 	\\
	 	\vt 
	 \end{bmatrix}  \cdot 	\biggl\{   	\begin{bmatrix}
	 v_B
	 \\
	 \theta_B
 \end{bmatrix} , 
\begin{bmatrix}
v_C
\\
\theta_C
\end{bmatrix} \biggr\}_x \,\d x\,=\, \int_{\T^3} \mathbf{z} \cdot \Bigl\{ 
\frac{\delta B}{\delta \mathbf{z}}, 
\frac{\delta C}{\delta \mathbf{z}}
\Bigr\}_x \,\d x.
\end{equation}
A direct calculation shows moreover that 
\begin{align}\label{jacobi_fin_dim}
	\biggl\{ 	\begin{bmatrix}
		v_1
		\\
		\theta_1
	\end{bmatrix},	\biggl\{   	\begin{bmatrix}
		v_2
		\\
		\theta_2
	\end{bmatrix} , 
	\begin{bmatrix}
		v_3
		\\
		\theta_3
	\end{bmatrix} \biggr\}_x \biggr\}_x  \,+\, 	\biggl\{ 	\begin{bmatrix}
	v_2
	\\
	\theta_2
\end{bmatrix},	\biggl\{   	\begin{bmatrix}
v_3
\\
\theta_3
\end{bmatrix} , 
\begin{bmatrix}
v_1
\\
\theta_1
\end{bmatrix} \biggr\}_x \biggr\}_x \,+\, 	\biggl\{ 	\begin{bmatrix}
v_3
\\
\theta_3
\end{bmatrix},	\biggl\{   	\begin{bmatrix}
v_1
\\
\theta_1
\end{bmatrix} , 
\begin{bmatrix}
v_2
\\
\theta_2
\end{bmatrix} \biggr\}_x \biggr\}_x \,=\, 0, 
\end{align}
i.e.,  the bracket \eqref{eqn1044} satisfies Jacobi's identity. 
The formula $\eqref{eqn1046}$ allows us to calculate
\begin{align}
	\delta \{B,C\}_\mathcal{L}(y)  \,=\, \Bigl\{ 
	\frac{\delta B}{\delta \mathbf{z}}, 
	\frac{\delta C}{\delta \mathbf{z}}
	\Bigr\}_y(y) \,+\, \int_{\T^3} \mathbf{z}(x) \cdot \Bigl\{ 
	 \frac{\delta^2 B}{\delta\mathbf{z}^2}  , 
	\frac{\delta C}{\delta \mathbf{z}}
	\Bigr\}_x (x,y)\,\d x\,+\, \int_{\T^3} \mathbf{z}(x) \cdot \Bigl\{ 
	\frac{\delta B}{\delta \mathbf{z}}, 
	\frac{\delta^2 C}{\delta\mathbf{z}^2} 
	\Bigr\}_x(x,y) \,\d x,
\end{align}
representing the second derivative as in \eqref{eqn1059}.
Applying the formula \eqref{eqn1046} once more, we obtain
\begin{align}
	\{A,  \{B,C\}_\mathcal{L}  \}_\mathcal{L}\,=\, &
	\int_{\T^3} \mathbf{z}(y) \cdot \Bigl\{ \frac{\delta A}{\delta \mathbf{z}}, \Bigl\{ 
	\frac{\delta B}{\delta \mathbf{z}}, 
	\frac{\delta C}{\delta \mathbf{z}}
	\Bigr\}_y \Bigr\}_y (y)\,\d y\\& +\, 
	\int_{\T^3 } 	\int_{\T^3} \mathbf{z}\otimes  \mathbf{z}(x,y) :  \biggl\{ \frac{\delta A}{\delta \mathbf{z}}, \Bigl\{ 
	\frac{\delta^2 B}{\delta\mathbf{z}^2}  , 
	\frac{\delta C}{\delta \mathbf{z}}
	\Bigr\}_x \biggr\}_y (x,y)\,\d x \,\d y
	\\&-\, 
	\int_{\T^3 } \int_{\T^3}\mathbf{z}\otimes \mathbf{z}(x,y) :\biggl\{ \frac{\delta A}{\delta \mathbf{z}}, \Bigl\{ \frac{\delta^2 C}{\delta\mathbf{z}^2}  , 
	\frac{\delta B}{\delta \mathbf{z}}
	\Bigr\}_x \biggr\}_y (x,y)\,\d x \,\d y,
\end{align}
where we used the antisymmetry of \eqref{eqn1044} in the last line. Analogously, we find
\begin{align}
	\{B,  \{C,A\}_\mathcal{L}  \}_\mathcal{L}\,=\, &
	\int_{\T^3} \mathbf{z}(y) \cdot \Bigl\{ \frac{\delta B}{\delta \mathbf{z}}, \Bigl\{ 
	\frac{\delta C}{\delta \mathbf{z}}, 
	\frac{\delta A}{\delta \mathbf{z}}
	\Bigr\}_y \Bigr\}_y (y)\,\d y\\&+\, 
	\int_{\T^3 } 	\int_{\T^3} \mathbf{z}\otimes  \mathbf{z}(x,y) :  \biggl\{ \frac{\delta B}{\delta \mathbf{z}}, \Bigl\{ 
	\frac{\delta^2 C}{\delta\mathbf{z}^2}  , 
	\frac{\delta A}{\delta \mathbf{z}}
	\Bigr\}_x \biggr\}_y (x,y)\,\d x \,\d y
	\\&-\, 
	\int_{\T^3 } \int_{\T^3}\mathbf{z}\otimes \mathbf{z}(x,y) :\biggl\{ \frac{\delta B}{\delta \mathbf{z}}, \Bigl\{ \frac{\delta^2 A}{\delta\mathbf{z}^2}  , 
	\frac{\delta C}{\delta \mathbf{z}}
	\Bigr\}_x \biggr\}_y (x,y) \,\d x \,\d y,
\end{align}
and
 \begin{align}
	\{C,  \{A,B\}_\mathcal{L}  \}_\mathcal{L}\,=\, &
	\int_{\T^3} \mathbf{z}(y) \cdot \Bigl\{ \frac{\delta C}{\delta \mathbf{z}}, \Bigl\{ 
	\frac{\delta A}{\delta \mathbf{z}}, 
	\frac{\delta B}{\delta \mathbf{z}}
	\Bigr\}_y \Bigr\}_y (y)\,\d y\\&+\, 
	\int_{\T^3 } 	\int_{\T^3} \mathbf{z}\otimes  \mathbf{z}(x,y) :  \biggl\{ \frac{\delta C}{\delta \mathbf{z}}, \Bigl\{ 
	\frac{\delta^2 A}{\delta\mathbf{z}^2}  , 
	\frac{\delta B}{\delta \mathbf{z}}
	\Bigr\}_x \biggr\}_y (x,y)\,\d x \,\d y
	\\&-\, 
	\int_{\T^3 } \int_{\T^3}\mathbf{z}\otimes \mathbf{z}(x,y) :\biggl\{ \frac{\delta C}{\delta \mathbf{z}}, \Bigl\{ \frac{\delta^2 B}{\delta\mathbf{z}^2}  , 
	\frac{\delta A}{\delta \mathbf{z}}
	\Bigr\}_x \biggr\}_y (x,y)\,\d x \,\d y.
\end{align}
When adding all these terms together,  the first terms of  each right-hand side cancels, due to \eqref{jacobi_fin_dim}. Regardung the remainder, we observe  that 
\begin{align}&\bigl\{ l, \bigl\{M, r\bigr\}_x \bigr\}_y \,=\, \bigl\{ l, r^v \cdot\nabla_x M  - M^{\cdot v} \cdot \nabla_x r \bigr\}_y  \\&\quad=\, (
	r^v\cdot\nabla_x M^{v \cdot} ) \cdot \nabla_y l \,- \, M^{v v} : (\nabla_x \otimes \nabla_y) r\otimes l
	\,-\, l^v\otimes  r^v : \nabla_y \otimes \nabla_x M \,+\, (l^v \cdot \nabla_y M^{\cdot v} ) \cdot \nabla_x r  ,
\end{align}
for $\R^4$-valued functions $l = (l^v(y),l^\theta(y)) $ and $r= (l^v(x),l^\theta(x)) $ and matrix-valued kernels 
\[M = 
\begin{bmatrix}
	M^{vv}  (x,y)& M^{v \theta}(x,y) \\
	M^{\theta v} (x,y) & M^{\theta\theta}(x,y)
\end{bmatrix}.
\]
Using the symmetry of \eqref{eqn1059},  we deduce
\[
\biggl\{ \frac{\delta A}{\delta \mathbf{z}}, \Bigl\{ 
\frac{\delta^2 B}{\delta\mathbf{z}^2}  , 
\frac{\delta C}{\delta \mathbf{z}}
\Bigr\}_x \biggr\}_y (x,y) \,=\, \biggl\{ \frac{\delta C}{\delta \mathbf{z}}, \Bigl\{ \frac{\delta^2 B}{\delta\mathbf{z}^2}  , 
\frac{\delta A}{\delta \mathbf{z}}
\Bigr\}_x \biggr\}_y (y,x),
\]
and analogous identities for the other double brackets, so that also these cancel each other.

\section{It\^o form of the Stratonovich equation for $(u,\psi)$}\label{sec-app-B}

As outlined in the introduction, the Navier--Stokes--Fourier system \eqref{Eq_intro} turns into the transformed system \eqref{Eq_new_var_formal}--\eqref{Eq_new_u} after the change of variables $\psi = \sqrt{2 \theta}$. The formal application  of the chain rule  is justified through the use of Stratonovich noise in both equations. 
Due to  the rigorous mathematical framework provided by It\^o integration,  we reformulate in this appendix the resulting Stratonovich equation in  It\^o sense by calculating the emerging correction terms.
To this end, we impose the expansions \eqref{Eq_noise_expansion} of the noises $\xi_F$ and $\xi_G$ as well as Assumption \ref{Ass:noise}. Then,
after integrating in time, \eqref{Eq_new_var_formal}--\eqref{Eq_new_u} can be written as 
\begin{align}\label{Eq_vel_app}
	\d u\,=\,&\nabla\cdot(\nabla_{\sym}u)\, \d t\,-\,\Pi \nabla \cdot (u \otimes u )\,\d t \,-\,\frac{1}{\sqrt{2}}\sum_{n \in \Z^3}\Pi \nabla\cdot(\psi f_n  \circ \d B_n),
	\\
	\label{Eq_SR_app}
	\d \psi \, = \,&\Delta \psi \,\d t\, +\, \frac{1}{\psi }\bigl(
	|\nabla \psi|^2 \,+\, |\nabla_\sym u|^2 
	\bigr)\, \d t\,- \,\nabla \cdot ( u\psi )\,\d t\, -\, \sum_{n\in \Z^3} \nabla \cdot ( \psi g_n\circ \d W_n )  \\&
	+\, \frac{1}{2} \sum_{n \in \Z^3} \psi \nabla \cdot ( g_n \circ \d  W_n) \,-\,  \frac{1}{\sqrt{2}}\sum_{n\in \Z^3}\nabla  u : f_n\circ\d  B_n.
\end{align}
We also  recall that the Stratonovich integral can be expressed as the corrected It\^o integral 
\[
f \circ \d \beta \, =\,  
\frac{1}{2}\, \d [f,\beta] \,+\, 
f  \d \beta,
\]
for a single Brownian motion $\beta$ and a semimartingale $f$. To identify the emerging quadratic variations
we therefore need to compute the relevant stochastic integrals in the semimartingale expansion of the integrands in \eqref{Eq_vel_app} and \eqref{Eq_SR_app}. For the former, we observe that the stochastic integral with respect to $B$  in the It\^o expansion of $\partial_i (\psi f_n)$ reads
\[
- 	\frac{1}{\sqrt{2}}\sum_{m \in \Z^3} \partial_i \bigl( f_n \nabla u : f_m \,\d B_m \bigr) \,=\,
- \frac{1}{\sqrt{2}}\sum_{k,l=1}^d\sum_{m \in \Z^3} \partial_i \bigl( f_n \partial_k u_l  f_m \,\d B_m^{lk} \bigr) 
.
\]
Consequently,  using the covariance structure \eqref{Eq_noise_matrix_covariation} we deduce that 
\begin{align}
	-	\frac{1}{\sqrt{2}} \partial_i(\psi f_n) \circ \d B^{ji}_n\,=&\, 
	\frac{1}{4} \partial_i( f_n (\partial_i u_j ) f_n)\, \d t \,+\, \frac{1}{4} \partial_i( f_n (\partial_j u_i ) f_n) \,\d t\,-\, 
	\frac{1}{\sqrt{2}} \partial_i(\psi f_n)\,  \d B^{ji}_n .
\end{align}
Summing over $i=1,2,3$ and $n\in \Z^3$ results in 
\begin{align}
	- 	\sum_{i=1}^{3}\sum_{n\in \Z^3} 	\frac{1}{\sqrt{2}} \partial_i(\psi f_n) \circ \d B^{ji}_n \,&=\, 
	\sum_{i=1}^{3}\sum_{n\in \Z^3}  \biggl(	\frac{1}{4}  \partial_i( f_n (\partial_i u_j ) f_n) \,\d t \,+\, \frac{1}{4} \partial_i( f_n (\partial_j u_i ) f_n)\, \d t\,-\, 
	\frac{1}{\sqrt{2}} \partial_i(\psi f_n)  \,\d B^{ji}_n \biggr)
	\\
	&= \,\frac{1}{4} F_1 \Delta u_j\, \d t \,-\, \frac{1}{\sqrt{2}} 	\sum_{i=1}^{3}\sum_{n\in \Z^3} \partial_i(\psi f_n) \, \d B^{ji}_n,
\end{align}
where we used the assumption \eqref{Eq_decay_condition_FG} and that $\nabla \cdot u = 0$ in the last equality. 
Applying the Helmholtz projection to the above vector yields 
\[
- \frac{1}{\sqrt{2}}\sum_{n \in \Z^3} \Pi \nabla\cdot(\psi f_n   \d B_n)
\,=\, \frac{F_1}{4} \Pi \Delta u\, \d t 
- \frac{1}{\sqrt{2}}\sum_{n \in \Z^3} \Pi \nabla\cdot(\psi f_n   \,\d B_n),
\]
so that we can rephrase \eqref{Eq_vel_app} as 
\[
\d u\,=\,\nabla\cdot(\nabla_{\sym}u)\, \d t\,-\,\Pi \nabla \cdot (u\otimes u)\,\d t\,+\, \frac{F_1}{4}\Delta u\,\d t\, - \,\frac{1}{\sqrt{2}}\sum_{n \in \Z^3}\Pi \nabla\cdot(\psi f_n  \, \d B_n).
\]

To proceed similarly for \eqref{Eq_SR_app}, we need to also identify the relevant martingale parts of the stochastic integrands $\partial_i (\psi g_n)$, $\psi \partial_ig_n$ and $(\partial_iu_j)f_n$. The It\^o expansion of  $\partial_i (\psi g_n)$ contains the following stochastic integrals with respect to $W$:
\[
- \sum_{j=1}^d\sum_{m\in \Z^3} \partial_i( g_n  \partial_j(\psi g_m) ) \,\d W_m^j \,+\, \frac{1}{2}
\sum_{j=1}^d \sum_{m\in \Z^3}  \partial_i( g_n \psi  \partial_j g_m )\, \d W_m^j .
\]
Therefore,
\begin{align}\begin{split}
		- \partial_i(\psi g_n ) \circ \d W_n^i \,=\,  \frac{1}{2}\partial_i( g_n  \partial_i(\psi g_n) ) \,\d t
		\,-\, \frac{1}{4}  \partial_i( g_n \psi  \partial_i g_n)  \,\d t \,-\, \partial_i(\psi g_n ) \d W_n^i,
	\end{split}
\end{align}
so that summing over $i=1,2,3$ and $n\in \Z^3$ results in 
\begin{equation}\label{Eq1}
	- \sum_{n\in \Z^3} \nabla \cdot ( \psi g_n\circ \d W_n ) \,=\,\frac{G_1}{2}\Delta \psi \,\d t- \sum_{n\in \Z^3} \nabla \cdot ( \psi g_n \d W_n ) ,
\end{equation}
where we used \eqref{Eq_decay_condition_FG}.
For $\psi \partial_i g_n$, the stochastic integrals in its It\^o expansion with respect to $W$ are
\[
- \sum_{j=1}^d\sum_{m\in \Z^3} (\partial_i g_n)  \partial_j(\psi g_m)\,  \d W_m^j \,+\, \frac{1}{2}
\sum_{j=1}^d \sum_{m\in \Z^3}  (\partial_i g_n) \psi  \partial_j g_m\,  \d W_m^j,
\]
and therefore
\begin{align}\begin{split}
		\frac{1}{2}\psi\partial_ig_n   \circ \d W_n^i \,=\, - \frac{1}{4}(\partial_ig_n ) \partial_i(\psi g_n)  \,\d t
		\,+\, \frac{1}{8} ( \partial_i g_n) \psi  \partial_i g_n  \,\d t \,+\, 	\frac{1}{2} \psi \partial_i g_n   \,\d W_n^i.
	\end{split}
\end{align}
Thus, summing over $i=1,2,3$ and $n\in \Z^3$, we find 
\begin{equation}\label{Eq2}
	\frac{1}{2} \sum_{n \in \Z^3} \psi \nabla \cdot ( g_n \circ \d  W_n) \,=\, -\frac{G_2}{8}\psi  \,\d t \,+\, 
	\frac{1}{2} \sum_{n \in \Z^3} \psi \nabla \cdot ( g_n  \, \d  W_n),
\end{equation}
based on \eqref{Eq_decay_condition_FG}.

Lastly we consider $(\partial_iu_j)f_n$, for which we identify 
\begin{align}
	- \frac{1}{\sqrt{2}}
	\sum_{m\in \Z^3} f_n \partial_i (\Pi \nabla\cdot(\psi f_m   \,\d B_m^{j \cdot }))\,=\,& 
	- \frac{1}{\sqrt{2}}
	\sum_{k=1}^3\sum_{m\in \Z^3} f_n \partial_i (\partial_k (\psi f_m ) )\, \d B_m^{j k }
	\\& + {\frac{1}{\sqrt{2}}
		\sum_{k,l=1}^3\sum_{m\in \Z^3} f_n \partial_{ij} \Delta^{-1}(\partial_{lk} (\psi f_m ) ) \,\d B_m^{l k }}
\end{align}
as its martingale part, recalling that $\Pi  u =  u -\nabla \Delta^{-1}\nabla \cdot u$. Whence we have 
\[
-\frac{1}{\sqrt{2}} (\partial_iu_j ) f_n \circ \d B_n^{ji} \,=\,  \frac{1+\delta_{ij}}{4}
f_n \partial_i^2(\psi f_n) 
\,\d t\,-\, 
\frac{1}{2} f_n \partial_{ij} \Delta^{-1} \partial_{ij}(\psi f_n)\,\d t\,
-\,\frac{1}{\sqrt{2}} (\partial_iu_j ) f_n\,  \d B_n^{ji},
\]
in light of \eqref{Eq_noise_matrix_covariation} and $\partial_{ij} = \partial_{ji}$.
Summing over $i,j =1,2 ,3 $ and $n\in \Z^3$, and using 
\begin{align}
	\sum_{i,j=1}^3 \frac{1+\delta_{ij}}{4}
	f_n \partial_i^2(\psi f_n) 
	\,\d t \,&=\,\sum_{i=1}^3 f_n \partial_{i}^2(\psi f_n )\,=\, f_n \Delta (\psi f_n),
	\\-
	\frac{1}{2} 
	\sum_{i,j=1}^3 \sum_{n\in \Z^3}f_n \partial_{ij} \Delta^{-1} \partial_{ij}(\psi f_n) \,&=\, 
	-\frac{1}{2}\sum_{n\in \Z^3}f_n  \Delta(\psi f_n), 
\end{align} 
we find 
due to \eqref{Eq_decay_condition_FG}, that
\begin{equation}\label{Eq3}
	\,-\,  \frac{1}{\sqrt{2}}\sum_{n\in \Z^3}\nabla  u : f_n\circ\d  B_n \,=\,
	F_1/2 \Delta \psi \,\d t \,-\ F_2/2 \psi  \d t 
	\, \d t
	\,-\,  \frac{1}{\sqrt{2}}\sum_{n\in \Z^3} \nabla  u : f_n\,\d  B_n.
\end{equation}
Inserting \eqref{Eq1}--\eqref{Eq3}\noeqref{Eq2} into \eqref{Eq_SR_app} yields the It\^o formulation 
\begin{align}
	\d \psi \, =\, &(1+F_1/2 +G_1/2)\Delta \psi\, \d t \,+\, \frac{1}{\psi }\bigl(
	|\nabla \psi|^2 \,+\, |\nabla_\sym u|^2 
	\bigr)\, \d t\,-\, \nabla \cdot ( u\psi )\,\d t\, -\, (F_2/2+ G_2/8) \psi\, \d t \\&- \sum_{n\in \Z^3} \nabla \cdot ( \psi g_n \,\d W_n ) \,
	+\, \frac{1}{2} \sum_{n \in \Z^3} \psi \nabla \cdot ( g_n  \,\d  W_n) \,-\,  \frac{1}{\sqrt{2}}\sum_{n\in \Z^3}\nabla  u : f_n\,\d  B_n
\end{align}
of \eqref{Eq_SR_app}.
\section{Difficulties related to the entropy estimate}\label{app:entropy}
The purpose of this appendix is to consolidate our claim in the introduction, that estimates in terms of the entropy \eqref{eq_ent} do not close at all noise intensities $F_1$. For this, we may informally apply It\^o's formula to compute the evolution of $\log\vt$, starting from \eqref{eqn1001}. This yields that 
	\begin{align}\begin{split}
		\d \log( \vt)\,  =\, &(1+F_1{/2} +G_1/2)\biggl(\Delta \log(\vt) + \frac{1}{\vt^2}|\nabla \vt|^2\biggr) \, \d t \, + \,\frac{1}{\vt}\bigl( (1+F_1/2) |\nabla_\sym u|^2 \,-\,F_1  |\nabla  \sqrt{\vt}|^2
		\bigr) \,\d t\\&-\, \nabla \cdot ( u \log(\vt)  )\,\d t \,-\, F_2 \,\d t \,-\, \frac{1}{2}\biggl(
		G_2 + \frac{G_1}{\vt^2}|\nabla \vt|^2
		\biggr)\,\d t
		\,-\, \frac{2F_1}{\vt} |\nabla_\sym u|^2 \,\d t  
		\\&- \sum_{n\in \Z^3} \frac{1}{\vt}\nabla \cdot ( \vt  g_n \,\d W_n ) \,-\,  \sum_{n\in \Z^3}\biggl( \frac{1}{\sqrt{\vt }}\nabla  u : f_n\,\d  B_n \biggr),
	\end{split}
\end{align}
in light of \eqref{Eq_noise_matrix_covariation} and \eqref{Eq_decay_condition_FG}. 
After simplifying and integrating in space we obtain
\begin{align}
	- \d \int_{\T^3}  \log( \vt) \,\d x \, = \,& (3F_1/2 - 1)\int_{\T^3} \frac{1}{\vt}|\nabla_\sym u|^2\,\d x \,-\, ( F_1/4+1) \int_{\T^3} \frac{1}{\vt}|\nabla \vt|^2\,\d x  \d t\,- \,(F_2 +G_2/2) \, \d t
	\\&+  \sum_{n\in \Z^3} \frac{1}{\vt}\nabla \cdot ( \vt  g_n \,\d W_n ) \,+\,  \sum_{n\in \Z^3}\biggl( \frac{1}{\sqrt{\vt }}\nabla  u : f_n\,\d  B_n \biggr).
\end{align}
While the quantity on the left-hand side may be negative, its negative part is controlled by the energy, cf.\ \cite[Section 2.2.3]{feireisl_book}. The local martingales on the right-hand side may be eliminated by taking the expectation and a stopping time argument. Regarding the deterministic terms on the right-hand side however, an entropy dissipation estimate may only be established for $F_2\le 2/3$. 
\section{Tightness in interpolation spaces}

In Section \ref{Sec:weak_ex} we need the following straightforward  lemma. 
\begin{lemma}\label{lemma_app_c}
	Let $Y_1\hookrightarrow  Y_\alpha \hookrightarrow Y_0$ be Banach spaces, such that $Y_0 \ni  y\mapsto \|y\|_{Y_1} \in \R$ is lower semicontinuous and  the interpolation inequality
	\begin{equation}\label{eqn_app}
		\|  y \|_{Y_\alpha}\,\le\, \|  y \|_{Y_0}^{1-\alpha} \|  y \|_{Y_1}^\alpha
	\end{equation}
	holds for some $\alpha \in (0,1)$. Let $(X_\epsilon)_{\epsilon\in (0,1)}$ be a family of $Y_\alpha$-valued random variables  that is uniformly tight in $ Y_0$ and has uniform tail estimates in $Y_1$, i.e., 
	\[\biggl(\sup_{\epsilon\in (0,1)} \P ( \|X_\epsilon \|_{Y_1} \ge N) \biggr)\,\to\, 0,
	\]
	as $N\to\infty$. Then $(X_\epsilon)_{\epsilon\in (0,1)}$ is also tight in $Y_\alpha$.
\end{lemma}
\begin{proof}
	Let $\delta >0$. Our goal is to identify a compact set $K_\delta \subset Y_\alpha$, such that 
	\[ \P ( X_\epsilon  \notin K_\delta ) \,\le\, \delta  ,
	\]
	for all $\epsilon\in (0,1)$. By tightness of $(X_\epsilon)_{\epsilon\in (0,1)}$ in $Y_0$, there exists a set $J_{\delta/2}$, such that
	\[ \P ( X_\epsilon  \notin J_{\delta/2} ) \,\le\, \delta/2  ,
	\]
	while for $N$ large enough, we have 
	\[
	\P(\| X_\epsilon \|_{Y_1} \ge N   ) \,\le\, \delta/2,
	\]
	for any $\epsilon\in (0,1)$. Defining $K_\delta := \{ y\in J_{\delta/2} |  \| y \|_{Y_1} \ge N  \}$, it remains to argue that the latter set is compact in $Y_\alpha$. To this end, we take  a sequence $(y_k)_{k \in \mathbb N} $ in $ K_\delta $, so that $y_{k_l} \to y$ in $Y_0$ for a suitable subsequence. By lower semicontinuity we deduce that $y\in Y_1$ with  $\|y\|_{Y_1}\le N$.  But then \eqref{eqn_app} implies that
	\begin{align}
		\| y_{k_l} - y\|_{Y_\alpha} \,\le\, \| y_{k_l} - y\|_{Y_0}^{1-\alpha} \| y_{k_l} - y \|_{Y_1}^\alpha\,\le\, (2N)^\alpha \| y_{k_l} - y\|_{Y_0}^{1-\alpha}\,\to\, 0,
	\end{align}
	as $N\to\infty$, showing that the convergence takes also place  in $Y_\alpha$. This implies the desired compactness of $K_\delta\subset Y_\alpha$ completing the proof.	
\end{proof}

\noindent
\textbf{Data availability.} This manuscript has no associated data.

\medskip

\noindent
\textbf{Declaration -- Conflict of interest.} The authors have no conflict of interest.

	\bibliographystyle{plain}	
	\bibliography{NSF_bib}

\begin{thebibliography}{10}

\bibitem{AV19_QSEE_1}
Antonio Agresti and Mark~C. Veraar.
\newblock Nonlinear parabolic stochastic evolution equations in critical spaces
  {P}art {I}. {S}tochastic maximal regularity and local existence.
\newblock {\em Nonlinearity}, 35(8):4100, 2022.

\bibitem{AV19_QSEE_2}
Antonio Agresti and Mark~C. Veraar.
\newblock Nonlinear parabolic stochastic evolution equations in critical spaces
  part {II}.
\newblock {\em Journal of Evolution Equations}, 22(2):1--96, 2022.

\bibitem{AV_PTRF}
Antonio Agresti and Mark~C. Veraar.
\newblock The critical variational setting for stochastic evolution equations.
\newblock {\em Probab. Theory Related Fields}, 188(3-4):957--1015, 2024.

\bibitem{AV_CMP}
Antonio Agresti and Mark~C. Veraar.
\newblock Stochastic {Navier--Stokes} equations for turbulent flows in critical
  spaces.
\newblock {\em Communications in Mathematical Physics}, 405(2):43, 2024.

\bibitem{BGW07}
John~B. Bell, Alejandro~L. Garcia, and Sarah~A. Williams.
\newblock Numerical methods for the stochastic {L}andau-{L}ifshitz
  {N}avier-{S}tokes equations.
\newblock {\em Phys. Rev. E (3)}, 76(1):016708, 12, 2007.

\bibitem{BF20}
Dominic Breit and Eduard Feireisl.
\newblock Stochastic {Navier--Stokes--Fourier} equations.
\newblock {\em Indiana University Mathematics Journal}, 69(4):911--975, 2020.

\bibitem{BFH}
Dominic Breit, Eduard Feireisl, and Martina Hofmanov{\'a}.
\newblock {\em Stochastically forced compressible fluid flows}, volume~3.
\newblock Walter de Gruyter GmbH \& Co KG, 2018.

\bibitem{Bresch_Desj}
Didier Bresch and Beno\^it Desjardins.
\newblock On the existence of global weak solutions to the {N}avier-{S}tokes
  equations for viscous compressible and heat conducting fluids.
\newblock {\em J. Math. Pures Appl. (9)}, 87(1):57--90, 2007.

\bibitem{brzesniak_moytl}
Zdzis{\l}aw Brze\'zniak and El\.zbieta Motyl.
\newblock Existence of a martingale solution of the stochastic
  {N}avier-{S}tokes equations in unbounded 2{D} and 3{D} domains.
\newblock {\em J. Differential Equations}, 254(4):1627--1685, 2013.

\bibitem{BFM09}
Miroslav Bul\'i\v{c}ek, Eduard Feireisl, and Josef M\'alek.
\newblock A {N}avier-{S}tokes-{F}ourier system for incompressible fluids with
  temperature dependent material coefficients.
\newblock {\em Nonlinear Anal. Real World Appl.}, 10(2):992--1015, 2009.

\bibitem{cannizzaro}
Giuseppe Cannizzaro and Jacek Kiedrowski.
\newblock Stationary stochastic {N}avier-{S}tokes on the plane at and above
  criticality.
\newblock {\em Stoch. Partial Differ. Equ. Anal. Comput.}, 12(1):247--280,
  2024.

\bibitem{CFH_Euler}
Dan Crisan, Franco Flandoli, and Darryl~D. Holm.
\newblock Solution properties of a 3{D} stochastic {E}uler fluid equation.
\newblock {\em J. Nonlinear Sci.}, 29(3):813--870, 2019.

\bibitem{DP_D}
Giuseppe Da~Prato and Arnaud Debussche.
\newblock Two-dimensional {N}avier-{S}tokes equations driven by a space-time
  white noise.
\newblock {\em J. Funct. Anal.}, 196(1):180--210, 2002.

\bibitem{da2014stochastic}
Giuseppe Da~Prato and Jerzy Zabczyk.
\newblock {\em Stochastic equations in infinite dimensions}, volume 152.
\newblock Cambridge university press, 2014.

\bibitem{dafermos}
Constantine~M. Dafermos.
\newblock The second law of thermodynamics and stability.
\newblock {\em Arch. Rational Mech. Anal.}, 70(2):167--179, 1979.

\bibitem{DHV_16}
Arnaud Debussche, Martina Hofmanov\'a, and Julien Vovelle.
\newblock Degenerate parabolic stochastic partial differential equations:
  quasilinear case.
\newblock {\em Ann. Probab.}, 44(3):1916--1955, 2016.

\bibitem{DV10}
Arnaud Debussche and Julien Vovelle.
\newblock Scalar conservation laws with stochastic forcing.
\newblock {\em J. Funct. Anal.}, 259(4):1014--1042, 2010.

\bibitem{DP_Lions_Boltzmann}
Ronald~J. DiPerna and Pierre-Louis Lions.
\newblock On the {C}auchy problem for {B}oltzmann equations: global existence
  and weak stability.
\newblock {\em Ann. of Math. (2)}, 130(2):321--366, 1989.

\bibitem{FG23}
Benjamin Fehrman and Benjamin Gess.
\newblock Non-equilibrium large deviations and parabolic-hyperbolic {PDE} with
  irregular drift.
\newblock {\em Invent. Math.}, 234(2):573--636, 2023.

\bibitem{feireisl_book}
Eduard Feireisl.
\newblock {\em Dynamics of viscous compressible fluids}, volume~26 of {\em
  Oxford Lecture Series in Mathematics and its Applications}.
\newblock Oxford University Press, Oxford, 2004.

\bibitem{feireisl}
Eduard Feireisl and Anton\'in Novotn\'y.
\newblock {\em Singular limits in thermodynamics of viscous fluids}.
\newblock Advances in Mathematical Fluid Mechanics. Birkh\"auser Verlag, Basel,
  2009.

\bibitem{FN12}
Eduard Feireisl and Anton\'in Novotn\'y.
\newblock Weak-strong uniqueness property for the full
  {N}avier-{S}tokes-{F}ourier system.
\newblock {\em Arch. Ration. Mech. Anal.}, 204(2):683--706, 2012.

\bibitem{FG95}
Franco Flandoli and Dariusz Gatarek.
\newblock Martingale and stationary solutions for stochastic {Navier-Stokes}
  equations.
\newblock {\em Probab. Theory Related Fields}, 102(3):367--391, 1995.

\bibitem{Fourier1822}
Joseph Fourier.
\newblock {\em Th\'eorie analytique de la chaleur}.
\newblock Chez Firmin Didot, P\`ere et Fils, Paris, 1822.

\bibitem{GHW24}
Benjamin Gess, Daniel Heydecker, and Zhengyan Wu.
\newblock {Landau-Lifshitz-Navier-Stokes Equations}: Large deviations and
  relationship to the energy equality, 2024.

\bibitem{gess2026probabilisticallystrongsolutionsstochastic}
Benjamin Gess and Robert Lasarzik.
\newblock Probabilistically strong solutions to stochastic {Euler} equations,
  2026.

\bibitem{gubinelli_turra}
Massimiliano Gubinelli and Mattia Turra.
\newblock Hyperviscous stochastic {N}avier-{S}tokes equations with white noise
  invariant measure.
\newblock {\em Stoch. Dyn.}, 20(6):2040005, 39, 2020.

\bibitem{hairer_rosati}
Martin Hairer and Tommaso Rosati.
\newblock Global existence for perturbations of the 2{D} stochastic
  {N}avier-{S}tokes equations with space-time white noise.
\newblock {\em Ann. PDE}, 10(1):Paper No. 3, 46, 2024.

\bibitem{HZZ_NS_3d}
Martina Hofmanov\'a, Rongchan Zhu, and Xiangchan Zhu.
\newblock Global existence and non-uniqueness for 3{D} {N}avier-{S}tokes
  equations with space-time white noise.
\newblock {\em Arch. Ration. Mech. Anal.}, 247(3):Paper No. 46, 70, 2023.

\bibitem{hornung2019quasilinear}
Luca Hornung.
\newblock Quasilinear parabolic stochastic evolution equations via maximal
  {$L^p$-regularity}.
\newblock {\em Potential Analysis}, 50(2):279--326, 2019.

\bibitem{Analysis1}
Tuomas~P. Hyt{\"o}nen, Jan~M.A.M. Van~Neerven, Mark~C. Veraar, and Lutz Weis.
\newblock {\em Analysis in {B}anach spaces. {V}ol. {I}. {M}artingales and
  {L}ittlewood-{P}aley theory}, volume~63 of {\em Ergebnisse der Mathematik und
  ihrer Grenzgebiete. 3. Folge.}
\newblock Springer, 2016.

\bibitem{Jak97}
Adam Jakubowski.
\newblock The almost sure {S}korokhod representation for subsequences in
  nonmetric spaces.
\newblock {\em Teor. Veroyatnost. i Primenen.}, 42(1):209--216, 1997.

\bibitem{jin2024fractionalstochasticlandaulifshitznavierstokes}
Ruhong Jin and Nicolas Perkowski.
\newblock Fractional stochastic {Landau-Lifshitz Navier-Stokes} equations in
  dimension $d \geq 3$: Existence and (non-)triviality, 2024.

\bibitem{kotitsas2025gaussianfluctuationsstochasticlandaulifshitz}
Sotiris Kotitsas, Marco Romito, Zhilin Yang, and Xiangchan Zhu.
\newblock Gaussian fluctuations for the stochastic {Landau-Lifshitz
  Navier-Stokes} equation in dimension $d\geq2$, 2025.

\bibitem{KU23}
Volker Kr\"atschmer and Mikhail Urusov.
\newblock A {K}olmogorov-{C}hentsov type theorem on general metric spaces with
  applications to limit theorems for {B}anach-valued processes.
\newblock {\em J. Theoret. Probab.}, 36(3):1454--1486, 2023.

\bibitem{krylov_Lp}
Nicolai~V. Krylov.
\newblock An analytic approach to {SPDE}s.
\newblock In {\em Stochastic partial differential equations: six perspectives},
  volume~64 of {\em Math. Surveys Monogr.}, pages 185--242. Amer. Math. Soc.,
  Providence, RI, 1999.

\bibitem{lad}
Olga~A. Ladyzhenskaya.
\newblock {\em The mathematical theory of viscous incompressible flow}.
\newblock Gordon and Breach Science Publishers, New York-London, english
  edition, 1963.
\newblock Translated from the Russian by Richard A. Silverman.

\bibitem{LL87}
Lew~D. Landau and Evgeny~M. Lifshitz.
\newblock {\em Course of theoretical physics. {V}ol. 6}.
\newblock Pergamon Press, Oxford, second edition, 1987.
\newblock Fluid mechanics, Translated from the third Russian edition by J. B.
  Sykes and W. H. Reid.

\bibitem{Leray}
Jean Leray.
\newblock Sur le mouvement d'un liquide visqueux emplissant l'espace.
\newblock {\em Acta Math.}, 63(1):193--248, 1934.

\bibitem{lifshitz}
Evgeny~M. Lifshitz and Lev~P. Pitaevskii.
\newblock {\em Statistical Physics, Part 2: Theory of the Condensed State},
  volume~9 of {\em Course of Theoretical Physics}.
\newblock Pergamon Press, 1980.

\bibitem{lions_book_1}
Pierre-Louis Lions.
\newblock {\em Mathematical topics in fluid mechanics. {V}ol. 1}, volume~3 of
  {\em Oxford Lecture Series in Mathematics and its Applications}.
\newblock The Clarendon Press, Oxford University Press, New York, 1996.
\newblock Incompressible models, Oxford Science Publications.

\bibitem{lions_book_2}
Pierre-Louis Lions.
\newblock {\em Mathematical topics in fluid mechanics. {V}ol. 2}, volume~10 of
  {\em Oxford Lecture Series in Mathematics and its Applications}.
\newblock The Clarendon Press, Oxford University Press, New York, 1998.
\newblock Compressible models, Oxford Science Publications.

\bibitem{LR_book}
Wei Liu and Michael R\"ockner.
\newblock {\em Stochastic partial differential equations: an introduction}.
\newblock Universitext. Springer, Cham, 2015.

\bibitem{MR2118862}
Remigijus Mikulevicius and Boris~L. Rozovskii.
\newblock Global {$L_2$}-solutions of stochastic {N}avier-{S}tokes equations.
\newblock {\em Ann. Probab.}, 33(1):137--176, 2005.

\bibitem{naumann}
Joachim Naumann.
\newblock An existence theorem for weak solutions to the equations of
  non-stationary motion of heat-conducting incompressible viscous fluids.
\newblock {\em J. Nonlinear Convex Anal.}, 7(3):483--497, 2006.

\bibitem{Navier1822}
Claude-Louis Navier.
\newblock M\'emoire sur les lois du mouvement des fluides.
\newblock {\em M\'emoires de l'Acad\'emie Royale des Sciences de l'Institut de
  France}, pages 389--440, 1822.

\bibitem{Ottinger}
Hans~C. \"{O}ttinger.
\newblock {\em Beyond equilibrium thermodynamics}.
\newblock John Wiley and Sons, 2005.

\bibitem{prodi}
Giovanni Prodi.
\newblock Un teorema di unicit\`a{} per le equazioni di {N}avier-{S}tokes.
\newblock {\em Ann. Mat. Pura Appl. (4)}, 48:173--182, 1959.

\bibitem{protter}
Philip~E. Protter.
\newblock {\em Stochastic integration and differential equations}, volume~21 of
  {\em Stochastic Modelling and Applied Probability}.
\newblock Springer-Verlag, Berlin, second edition, 2005.
\newblock Corrected third printing.

\bibitem{QY98}
Jeremy Quastel and Horng-Tzer Yau.
\newblock Lattice gases, large deviations, and the incompressible
  {N}avier-{S}tokes equations.
\newblock {\em Ann. of Math. (2)}, 148(1):51--108, 1998.

\bibitem{rock_fully_monotone}
Michael R\"ockner, Shijie Shang, and Tusheng Zhang.
\newblock Well-posedness of stochastic partial differential equations with
  fully local monotone coefficients.
\newblock {\em Math. Ann.}, 390(3):3419--3469, 2024.

\bibitem{rockner_tamed}
Michael R\"ockner and Tusheng Zhang.
\newblock Stochastic 3{D} tamed {N}avier-{S}tokes equations: existence,
  uniqueness and small time large deviation principles.
\newblock {\em J. Differential Equations}, 252(1):716--744, 2012.

\bibitem{serrin}
James Serrin.
\newblock On the interior regularity of weak solutions of the {N}avier-{S}tokes
  equations.
\newblock {\em Arch. Rational Mech. Anal.}, 9:187--195, 1962.

\bibitem{Simon}
Jacques Simon.
\newblock Compact sets in the space {$L^p(0,T;B)$}.
\newblock {\em Ann. Mat. Pura Appl. (4)}, 146:65--96, 1987.

\bibitem{Stokes1845}
George~Gabriel Stokes.
\newblock On the theories of the internal friction of fluids in motion, and of
  the equilibrium and motion of elastic solids.
\newblock {\em Transactions of the Cambridge Philosophical Society},
  8:287--305, 1845.

\bibitem{stoch_fubini}
Mark~C. Veraar.
\newblock The stochastic {F}ubini theorem revisited.
\newblock {\em Stochastics}, 84(4):543--551, 2012.

\bibitem{wiedemann}
Emil Wiedemann.
\newblock Weak-strong uniqueness in fluid dynamics.
\newblock In {\em Partial differential equations in fluid mechanics}, volume
  452 of {\em London Math. Soc. Lecture Note Ser.}, pages 289--326. Cambridge
  Univ. Press, Cambridge, 2018.

\end{thebibliography}

\end{document}